\documentclass{article}

\newcommand{\vect}[1]{\boldsymbol{#1}} 
\newcommand{\R}{\mathbb{R}} 
\newcommand{\1}{\vect{1}} 
\newcommand{\0}{\vect{0}}
\newcommand{\lyp}{\mathcal{L}} 
\newcommand{\diag}[1]{\textrm{diag}\left({#1}\right)} 
\newcommand{\id}[1]{\vect{1}_{\{{#1}\}}} 

\usepackage{threeparttable}
\usepackage{enumerate}

\usepackage{microtype}
\usepackage{graphicx}
\usepackage{subcaption}
\usepackage{booktabs} 

\usepackage{amsmath}
\usepackage{amssymb}
\usepackage{mathtools}
\usepackage{amsthm}
\usepackage{algorithm}
\usepackage{algorithmic}
\usepackage{multirow}

\theoremstyle{plain}
\newtheorem{theorem}{Theorem}[section]
\newtheorem{proposition}[theorem]{Proposition}
\newtheorem{lemma}[theorem]{Lemma}

\theoremstyle{definition}

\theoremstyle{remark}

\usepackage{hyperref}

\usepackage[preprint]{neurips_2026}

\usepackage[utf8]{inputenc} 
\usepackage[T1]{fontenc}    
\usepackage{hyperref}       
\usepackage{url}            
\usepackage{booktabs}       
\usepackage{amsfonts}       
\usepackage{nicefrac}       
\usepackage{microtype}      
\usepackage{xcolor}         

\title{Newton Method for Fixed-Support Doubly Entropic Wasserstein Barycenter}

\author{%
  \textbf{Jianting Pan}\textsuperscript{1},
  \textbf{Sirong Dai}\textsuperscript{1},
  \textbf{Lei Yang}\textsuperscript{2},
  \textbf{Yaomin Wang}\textsuperscript{1},
  \textbf{Ji'an Li}\textsuperscript{1},
  \textbf{Ming Yan}\textsuperscript{1}
  \\
  \textsuperscript{1}The Chinese University of Hong Kong, Shenzhen
  \\
  \textsuperscript{2}Sun Yat-sen University
}

\begin{document}

\maketitle

\begin{abstract}
We study the fixed-support doubly regularized Wasserstein barycenter problem. Using the semi-dual formulation of entropic optimal transport, we reformulate the problem as a smooth, unconstrained, convex optimization problem in the dual variables. We then derive explicit expressions for the gradient and Hessian and develop an exact Newton method for high-accuracy barycenter computation. To improve scalability, we propose a sparse Newton variant that sparsifies the transport probability matrices, thereby reducing the cost of Hessian-vector products. We establish theoretical results for the proposed methods, including Hessian approximation bounds and convergence results. Experiments on synthetic and real datasets show that the sparse Newton method converges faster than representative state-of-the-art solvers. 

\end{abstract}

\section{Introduction}

The Wasserstein barycenter problem seeks a representative probability measure that minimizes the weighted sum of Wasserstein distances to a collection of input probability measures. Since its introduction in~\cite{agueh2011barycenters}, Wasserstein barycenters have become a fundamental tool for averaging and interpolating probability distributions while preserving underlying geometry. They have found applications in economics~\cite{carlier2010matching,chiappori2010hedonic}, physics~\cite{buttazzo2012optimal,cotar2013density}, statistics~\cite{bassetti2006minimum,boissard2015distribution}, image processing~\cite{rabin2011wasserstein}, computer graphics~\cite{solomon2015convolutional}, and machine learning.

In the discrete setting, the Wasserstein barycenter problem can be formulated as:
\begin{equation}
\label{equ:discrete_wb_lp}
    \min_{\vect{v} \in \Delta_n} \sum_{k=1}^K w_k \min_{X_k \in \Omega_k(\vect{v})} \langle C_k, X_k \rangle, \ \text{where}\ \Omega_k(\vect{v}) := \{X_k \in \R_{+}^{n \times m}\mid X_k \1=\vect{v}, X_k^{\top} \1=\vect{\mu}_k\},
\end{equation}
where $C_k\in\mathbb{R}_{+}^{n\times m}$ is the ground cost matrix, $\{\vect{\mu}_k\}_{k=1}^K$ are $K$ discrete probability measures with
$\vect{\mu}_k \in \Delta_m$, and $w_k>0$ are prescribed weights satisfying
$\sum_{k=1}^K w_k=1$. The set
$\Delta_n := \{\vect{\mu}\in\mathbb{R}_{+}^{n}\mid \sum_{i=1}^n \mu_i=1\}$
denotes the probability simplex. The goal is to compute a barycenter
$\vect{v}\in\Delta_n$ supported on a \emph{fixed} set of $n$ points. Although
\eqref{equ:discrete_wb_lp} is a linear programming formulation, directly solving it becomes computationally demanding for large-scale problems due to the high dimensionality of the transport plans and the coupling induced by the shared barycenter variable~\cite{lin2020fixed}.

To improve computational tractability, we consider the \emph{doubly regularized entropic Wasserstein barycenter} (WB) problem studied in~\cite{chizat2025doubly} under the fixed-support settings, which incorporates both an inner entropic regularization and an outer Kullback-Leibler (KL) regularization:
\begin{equation}
\label{opt: barycenter problem}
    \min_{\vect{v} \in \Delta_n}\ F(\vect{v}) := \sum_{k=1}^K w_k\;\text{OT}_{k, \eta}(\vect{v}, \vect{\mu}_k) + \tau\;\text{KL}(\vect{v}, \vect{\pi}),
\end{equation}
where $\eta>0$ is the inner entropic regularization parameter, $\tau>0$ controls the outer regularization, and $\vect{\pi}\in\Delta_n$ is a prescribed measure. For each $k$, the entropic transport cost is defined by
\begin{equation}
\label{opt: EOT}
    \text{OT}_{k,\eta}(\vect{v}, \vect{\mu}_k) := \min_{X_k \in \Omega_k(\vect{v})} \ \langle C_k, X_k \rangle + \eta\;\text{KL}(X_k, \vect{v}\vect{\mu}_k^\top),
\end{equation}
where $\text{KL}(\cdot, \cdot)$ denotes the generalized Kullback-Leibler divergence extended to matrices. We refer to the unique minimizer of~\eqref{opt: barycenter problem} as the $(\eta,\tau)$-barycenter. The uniqueness follows from the strict convexity introduced by the outer KL regularization when $\tau>0$ and $\vect{\pi}$ having full support~\cite{vaskevicius2023computational}.

The formulation~\eqref{opt: barycenter problem} unifies several regularized barycenter models in the literature. Inner-regularized barycenters with $(\eta,0)$ are known to exhibit a shrinking bias~\cite{rigollet2018entropic}, while the choice $(\eta,\eta)$, considered in~\cite{benamou2015iterative,cuturi2014fast,cuturi2018semidual,kroshnin2019complexity,lin2020fixed}, leads to barycenters with a blurring bias. More recently, the choice $(\eta,\eta/2)$ has attracted attention~\cite{chizat2025doubly,vaskevicius2023computational}: for smooth densities, it yields an approximation bias of order $\eta^2$, whereas $\tau=\eta$ gives a larger bias of order $\eta$~\cite{chizat2025doubly}. Therefore, small regularization parameters are important for accurately approximating unregularized Wasserstein barycenters and preserving sharp geometric structures. However, in this regime, first-order solvers such as Sinkhorn-type and Bregman projection methods often require many iterations to reach high accuracy~\cite{cuturi2014fast,peyre2019computational,vaskevicius2023computational}. This motivates the development of second-order methods for the doubly regularized barycenter problem~\eqref{opt: barycenter problem}.

In this paper, we develop Newton-type algorithms to solve the doubly regularized Wasserstein barycenter problem~\eqref{opt: barycenter problem}. We first introduce an exact Newton method that exploits the smoothness of the regularized objective and achieves fast local convergence. However, computing the Newton direction can be expensive for large-scale instances, as it requires solving a large linear system involving dense operations with the transport probability matrices. To address this bottleneck, we propose a sparsification strategy applied directly to these matrices. For small entropic regularization parameters, the transport probabilities are typically concentrated on a small set of dominant entries, while the remaining entries contribute negligibly. Retaining only these dominant entries reduces the cost of Hessian-vector products while preserving the main curvature information in Newton updates. The resulting algorithm combines the high accuracy of second-order methods with improved scalability for large-scale barycenter computation. We establish convergence guarantees for both algorithms and demonstrate their empirical advantages over representative state-of-the-art methods.

\noindent{\bf Contributions.} Our main contributions are summarized as follows:
\begin{itemize}  \setlength\itemsep{0pt}
    \item \textbf{Newton Methods and Sparsification:} We develop an exact Newton method, NWB, for the doubly regularized Wasserstein barycenter problem and a sparse variant, SNWB, that sparsifies the transport probability matrices to reduce the cost of Hessian-vector products.
    \item \textbf{Theoretical Analysis:} We establish structural and spectral properties of the exact and `sparse' Hessians, derive approximation error bounds, and prove convergence results, including local quadratic convergence under suitable conditions. Proofs of all theoretical results are deferred to Appendix~\ref{app: Proofs of Theoretical Results}.
    \item \textbf{Numerical Experiments:} We evaluate the proposed methods on synthetic and real benchmark problems. The results show that NWB and SNWB compare favorably with the state-of-the-art solvers, with SNWB providing substantial acceleration in many cases.
\end{itemize}

\noindent{\bf Notation.} For a vector $\vect{x}\in\mathbb{R}^n$, we denote its $i$-th entry by $x_i$. For a block vector $\vect{y}=[\vect{y}_1^\top, \ldots, \vect{y}_K^\top]^ \top \in \R^{mK}$, where $\vect{y}_k \in \R^{m}$, we denote the $i$-th entry of the $k$-th block by $[\vect{y}_k]_i$. We denote $[K]:=\{1,\ldots,K\}$ and $\1_n$ and $\0_n$ as the $n$-dimensional vectors of all ones and all zeros, respectively. We use $\lambda_{\min}^+(\cdot)$ and $\lambda_{\max}(\cdot)$ to represent the smallest positive and largest eigenvalues of real symmetric matrices, respectively. For a matrix $A \in \R^{m \times n}$, let $A_{i\cdot}$ be the vector of the $i$-th row of $A$. The kernel of the matrix $A$ is defined as $\text{ker}(A) := \{\vect{x} \in \R^{n}\mid A\vect{x} = \0_m\}$. We denote $\|\cdot\|_p$ as the $\ell_p$ norm of a vector or $\ell_p$ induced norm of a matrix. For simplicity, we use $\|\cdot\|$ to denote the Euclidean norm of a vector or the $\ell_2$ induced norm of a matrix.


\section{Background}
\label{sec: background}

In this section, we briefly review the dual formulation of entropic optimal transport (EOT) in Section~\ref{subsec: Entropic Optimal Transport} and analytically derive the unconstrained reformulation of the doubly regularized Wasserstein barycenter problem~\eqref{opt: barycenter problem} in Section~\ref{subsec: Formulation of the Doubly Entropic Wasserstein Barycenter Problem}. 

\subsection{Entropic Optimal Transport}
\label{subsec: Entropic Optimal Transport}

For a fixed barycenter candidate $\vect{v} \in \Delta_n$ and a target measure $\vect{\mu}_k \in \Delta_m$, the EOT problem~\eqref{opt: EOT} admits a standard dual formulation~\cite{peyre2019computational}. Introducing dual variables $\vect{\alpha}_k \in \mathbb{R}^n$ and $\vect{\beta}_k \in \mathbb{R}^m$ associated with the marginal constraints, the two-dual formulation is
\begin{equation}
\label{opt: two-dual formulation}
\max_{\vect{\alpha}_k \in \R^n, \vect{\beta}_k \in \R^m} \langle\vect{\alpha}_k, \vect{v} \rangle + \langle\vect{\beta}_k, \vect{\mu}_k \rangle + \eta - \eta \sum_{i=1}^n \sum_{j=1}^m v_i [\vect{\mu}_k]_{j} \exp\left(\frac{[\vect{\alpha}_k]_{i} + [\vect{\beta}_{k}]_{j} - [C_k]_{ij}}{\eta}\right).
\end{equation}
For a fixed $\vect{\beta}_k$, by the first-order optimality condition, the optimal $\vect{\alpha}_k$ that maximizes the objective in~\eqref{opt: two-dual formulation} is given by:
\begin{equation}
\label{equ: semi-dual expression}
[\vect{\alpha}_k(\vect{\beta_k})]_{i} := -\eta \log\left(\sum_{j=1}^m [\vect{\mu}_k]_{j}\exp\left(\frac{[\vect{\beta}_{k}]_{j}-[C_k]_{ij}}{\eta}\right)\right), \quad \forall \ i \in [n]. 
\end{equation}
Substituting $\vect{\alpha}_k (\vect{\beta}_k)$ into the two-dual formulation~\eqref{opt: two-dual formulation} yields the semi-dual problem:
\begin{equation}
\label{opt: semi-dual EOT}
    \text{OT}_{k,\eta}(\vect{v}, \vect{\mu}_k) = \max_{\vect{\beta}_k \in \R^m} \ \langle \vect{\alpha}_k(\vect{\beta}_k), \vect{v} \rangle + \langle \vect{\beta}_k, \vect{\mu}_k \rangle.
\end{equation}

Thus, EOT can be evaluated by maximizing an unconstrained, smooth, concave function. Let $\vect{\beta}_k^{\star}$ be an optimal solution of~\eqref{opt: semi-dual EOT}. The corresponding optimal transport plan can be recovered as $X_k^{\star} = \diag{\vect{v}}P_k(\vect{\beta}_k^{\star})$, where $P_k(\vect{\beta}_k)$ is a row-stochastic matrix
($P_k(\vect{\beta}_k)\1_m = \1_n$) with entries:
\begin{equation}
\label{equ: P matrix}
[P_{k}(\vect{\beta_k})]_{ij} = \frac{[\vect{\mu}_{k}]_{j} \exp\left(([\vect{\beta}_{k}]_{j} - [C_k]_{ij}) / \eta\right)}{\sum_{l=1}^m [\vect{\mu}_{k}]_{l} \exp\left( ([\vect{\beta}_{k}]_{l} - [C_k]_{il}) / \eta\right)}.
\end{equation}

\subsection{Formulation of the Doubly Entropic Wasserstein Barycenter Problem}
\label{subsec: Formulation of the Doubly Entropic Wasserstein Barycenter Problem}

Substituting the semi-dual representation~\eqref{opt: semi-dual EOT} into the doubly entropic Wasserstein barycenter problem~\eqref{opt: barycenter problem} leads to the following min-max formulation:
\begin{equation}
\label{opt: min-max problem}
    \min_{\vect{v} \in \Delta_n} \max_{\vect{\beta}_1, \dots, \vect{\beta}_K \in \R^m} \sum_{k=1}^K w_k \left( \langle \vect{\alpha}_k(\vect{\beta}_k), \vect{v} \rangle + \langle \vect{\beta}_k, \vect{\mu}_k \rangle \right) + \tau \text{KL}(\vect{v}, \vect{\pi}).
\end{equation}
Since the objective is continuous jointly in $\vect{v}$ and $\{\vect{\beta}_k\}_{k=1}^K$, strictly convex in $\vect{v} \in \Delta_n$, and concave in $\{\vect{\beta}_k\}_{k=1}^K$, we can invoke the minimax theorem~\cite{rockafellar1997convex}. The compactness of the simplex $\Delta_n$ guarantees that we can safely interchange the order of minimization and maximization. Grouping the terms involving $\vect{v}$ yields the inner minimization problem:
\begin{equation}
    \min_{\vect{v} \in \Delta_n} \ \langle \Phi(\vect{\beta}), \vect{v} \rangle + \tau \text{KL}(\vect{v}, \vect{\pi}),
\end{equation}
where $\vect{\beta} = [\vect{\beta}_1^\top, \dots, \vect{\beta}_K^\top]^\top$ and $\Phi(\vect{\beta}) = \sum_{k=1}^K w_k \vect{\alpha}_k (\vect{\beta}_k)$. Clearly, this inner problem admits a closed-form solution given fixed $\vect{\beta}$:
\begin{equation}
\label{equ: optimal v}
    [\vect{v}(\vect{\beta})]_{i} = \frac{\pi_i \exp\left(-[\Phi(\vect{\beta})]_i / \tau\right)}{\sum_{l=1}^n \pi_l \exp\left(-[\Phi(\vect{\beta})]_l / \tau\right)}, \quad \forall i \in [n].
\end{equation}
Substituting this back into the problem~\eqref{opt: min-max problem}, we obtain the following unconstrained problem:
\begin{equation}
\label{opt: optimization problem}
    \min_{\vect{\beta} \in \R^{mK}} \ \lyp(\vect{\beta}) := - \sum_{k=1}^K w_k \langle \vect{\beta}_k, \vect{\mu}_k \rangle + \tau \log \left(\sum_{i=1}^n \pi_i \exp\left(-[\Phi(\vect{\beta})]_i /\tau\right)\right).
\end{equation}
The above objective $\lyp(\vect{\beta})$ is smooth and convex, providing an unconstrained formulation that enables scalable, second-order optimization methods.

\section{The Exact Newton Method}
\label{sec: newton method}

We now develop an exact Newton method for problem~\eqref{opt: optimization problem}. We first derive the gradient and Hessian of the dual objective in Section~\ref{subsec: gradient and hessian}, and then present the algorithm in Section~\ref{subsec: The Proposed Algorithm of NWB}.

\subsection{Gradient and Hessian}
\label{subsec: gradient and hessian}

For notational simplicity, throughout this section, we write $\vect{v}=\vect{v}(\vect{\beta})$ and $P_k=P_k(\vect{\beta}_k)$, as defined in~\eqref{equ: optimal v} and~\eqref{equ: P matrix}, respectively. With this notation, the gradient of $\vect{g} := \nabla \lyp(\vect{\beta})$ with respect to the $k$-th block variable $\vect{\beta}_k$ takes the simple form:
\begin{equation}
\label{equ: gradient}
    \vect{g}_k := \nabla_{k} \lyp(\vect{\beta}) = w_k ((P_k)^\top \vect{v} - \vect{\mu}_k).
\end{equation}
Thus, each gradient block measures the discrepancy between the transported marginal $(P_k)^\top \vect{v}$ and the prescribed target marginal $\vect{\mu}_k$. At optimality, this marginal discrepancy vanishes for every $k$.

Let $H:=\nabla^2\lyp(\vect{\beta})$ and denote its $(k,r)$-th block by $H_{kr}:=\nabla^2_{kr}\lyp(\vect{\beta})$. Define $\vect{\gamma}_k:=P_k^\top\vect{v}$. Then the diagonal and off-diagonal blocks of $H$ are
\begin{align}
    H_{kk} &= \frac{w_k^2}{\tau} \left( (P_k)^\top \diag{\vect{v}} P_k - \vect{\gamma}_k (\vect{\gamma}_k)^\top \right) + \frac{w_k}{\eta} \left( \diag{\vect{\gamma}_k} - (P_k)^\top \diag{\vect{v}} P_k \right), \label{equ: hessian_diag} \\
    H_{kr} &= \frac{w_k w_r}{\tau} \left( (P_k)^\top \diag{\vect{v}} P_r - \vect{\gamma}_k (\vect{\gamma}_r)^\top \right) \quad \text{for } k \neq r. \label{equ: hessian_nondiag}
\end{align}

The Hessian consists of two types of curvature contributions. The terms scaled by $1/\tau$ describe the global coupling induced by the barycenter variable, while the terms scaled by $1/\eta$ appear only in the diagonal blocks and correspond to the local curvature of the individual entropic transport problems.

The block-wise representation also makes explicit an invariance of the dual formulation: adding a constant to any dual block does not change the induced transport plan. Consequently, the Hessian is singular along block-wise constant directions. The following proposition characterizes this structure.

\begin{proposition}
\label{prop: kernel space of hessian}
    The Hessian matrix $H$ is symmetric and positive semidefinite. Moreover, its kernel is exactly the $K$-dimensional subspace of block-wise constant vectors, namely, 
    \begin{equation}
        \text{ker}(H) = \left\{ \vect{d} = (\vect{d}_1^\top, \dots, \vect{d}_K^\top)^\top \in \mathbb{R}^{mK} \mid \vect{d}_k = c_k \1_m, \, c_k \in \mathbb{R}, \, \forall k \in [K] \right\}.
    \end{equation}
    Furthermore, the gradient $\vect{g}$ is orthogonal to this kernel, i.e., $\nabla \lyp(\beta) \in \text{ker}(H)^\perp$. Consequently, the Newton system is consistent.
\end{proposition}

Proposition~\ref{prop: kernel space of hessian} shows that the only singular directions of $H$ are the block-wise constant shifts of the dual variables. Since $\nabla \lyp(\vect{\beta}) \in \text{ker}(H)^\perp$, the Newton direction can be computed within $\text{ker}(H)^\perp$. The next result gives spectral bounds for the Hessian on this restricted subspace.

\begin{proposition}
\label{prop: the minimum and maximum eigenvalues of the hessian}
We obtain the following properties of the Hessian matrix $H$. 
\begin{equation}
\lambda_{\min}^{+}(H) \ge \frac{m}{\eta n} \min_{k} (w_k p_k^2), \quad \lambda_{\max}(H) \le (\max_{k} w_k) \left( \frac{1}{2\eta} + \frac{1}{\tau} \right),
\end{equation}
where $p_k = \min_{j} [P_{k}]_{i^\star j}$ for $i^\star\in {\arg\max}_i{[\vect{v}]}_i$.
\end{proposition}

These bounds make explicit how the conditioning of the restricted Newton system depends on the regularization parameters, the barycentric weights, and the transport probability matrices. These properties ensure that $H$ is well-conditioned on $\text{ker}(H)^\perp$. Together, the gradient and Hessian formulas provide the basis for an exact Newton method and also reveal the structure exploited by the sparse Newton method in Section~\ref{sec: The Sparse Newton method}.


\subsection{The Proposed Algorithm}
\label{subsec: The Proposed Algorithm of NWB}

We now present the Newton method for the Wasserstein barycenters (NWB), summarized in Algorithm~\ref{alg: NWB}. The method applies a regularized Newton iteration to the smooth unconstrained formulation~\eqref{opt: optimization problem}, and recovers the barycenter variable through the closed-form relation~\eqref{equ: optimal v}.


The algorithm is initialized by any first-order barycenter solver, such as IBP~\cite{benamou2015iterative} for $(\eta,\eta)$-barycenters or EDSWB~\cite{vaskevicius2023computational} for $(\eta,\eta/2)$-barycenters. This preliminary stage provides a coarse approximation of the dual variable, denoted by $\vect{\beta}^{0}$, which serves as a warm start for the subsequent Newton iterations. 

At each iteration $t$, the algorithm computes the semi-dual potentials $\{\vect{\alpha}_k^t\}_{k=1}^K$, the barycenter mass vector $\vect{v}^t$, the transport probability matrices $\{P_k^t\}_{k=1}^K$, and the induced transported marginals $\{\vect{\gamma}_k^t\}_{k=1}^K$. These quantities are then used to compute the gradient by~\eqref{equ: gradient}. The Hessian matrix itself is not formed explicitly. Instead, we use the block expressions in~\eqref{equ: hessian_diag} and~\eqref{equ: hessian_nondiag} to define an implicit linear operator $z\mapsto H^t z$, which is sufficient for the conjugate gradient solver to compute the Newton direction. To improve numerical stability, we solve a shifted Newton system
\begin{align*}
\Delta \vect{\beta}^{t}
=
-\left(H^t+\|\vect{g}^t\| I\right)^{-1}\vect{g}^t.
\end{align*}
The shift vanishes as the iterates approach a stationary point, so the method gradually recovers the exact Newton direction in the local convergence regime. Similar regularization mechanisms have been adopted in recent Newton-type methods for large-scale optimization~\cite{tang2024safe,wang2025sparse,pan2026inexact}. A stepsize is then selected by Armijo backtracking to ensure sufficient descent of the objective.

\begin{algorithm}[t]
\caption{Newton method for Wasserstein Barycenter (NWB)}
\label{alg: NWB}
\begin{algorithmic}[1]
    \STATE {\bfseries Input:} $\vect{w} \in \Delta_K$, $\eta,\ \tau > 0$, $\{\vect{\mu}_k\}_{k=1}^K \in \Delta_m$, $\{C_k\}_{k=1}^K \in \R_+^{n \times m}$.
    \STATE Use a first-order solver, such as IBP~\cite{benamou2015iterative} or EDSWB~\cite{vaskevicius2023computational}, to obtain an initial dual variable $\vect{\beta}^{0}=[(\vect{\beta}_1^0)^\top,\ldots,(\vect{\beta}_K^0)^\top]^\top\in \R^{mK}$.
    \STATE Normalize each block by $\vect{\beta}_k^{0} \gets \vect{\beta}_k^{0} - (\frac{1}{m}\1_m^\top \vect{\beta}_k^{0})\1_m$ for $k \in [K]$. 
    \FOR{$t=0,1,..., T$}
    \STATE Calculate $[\vect{\alpha}_{k}^{t}]_{i} \gets -\eta \log\left(\sum_{j=1}^m [\vect{\mu}_k]_{j}\exp\left(([\vect{\beta}^{t}_{k}]_{j}-[C_k]_{ij})/\eta\right)\right)$ for $k \in [K]$. 
    \STATE Construct the weighted sum $\vect{\Phi}^t \gets \sum_k w_k \vect{\alpha}_k^t$. 
    \STATE Compute barycenter mass $\vect{v}$ by $  v_i^t \gets \frac{\pi_i \exp(-\Phi^t_i/\tau)}{\sum_{l} \pi_l \exp(-\Phi^t_l/\tau)}$ for $i \in [n]$. 
    \STATE Compute $P_k^t$ by $[P^t_{k}]_{ij} \gets \frac{[\vect{\mu}_k]_{j} \exp\left(([\vect{\beta}^{t}_{k}]_{j} - [C_k]_{ij}) / \eta\right)}{\sum_{l=1}^m [\vect{\mu}_k]_{l} \exp\left( ([\vect{\beta}^{t}_{k}]_{l} - [C_k]_{il}) / \eta\right)}$ for $k \in [K]$. 
    \STATE Compute the marginal $\vect{\gamma}^t_k \gets (P^t_k)^\top \vect{v}^t$ for $k \in [K]$. 
    \STATE Calculate the gradient $\vect{g}^t \gets \nabla \lyp (\vect{\beta}^t)$.
    \STATE Construct the Hessian matrix $H^t$.
    \STATE Compute $\Delta \vect{\beta}^t \gets -(H^t + \|\vect{g}^t\|I)^{-1}\vect{g}^t$ by the conjugate gradient.
    \STATE Select a stepsize $s^t$ by Armijo backtracking.
    \STATE Update $\vect{\beta}^{t+1} \gets \vect{\beta}^t + s^t \Delta \vect{\beta}^t$.
    \ENDFOR
    \STATE {\bfseries Output:} $\vect{v}^T$.
\end{algorithmic}
\end{algorithm}

\section{The Sparse Newton Method}
\label{sec: The Sparse Newton method}

In this section, we introduce a sparsification strategy to accelerate the computation of Newton directions. We first describe the construction of sparse transport probability matrices in Section~\ref{subsec: Sparsification Strategy}, and then present the resulting sparse Newton algorithm in Section~\ref{subsec: The Proposed Algorithm of SNWB}.

\subsection{Sparsification Strategy}
\label{subsec: Sparsification Strategy}

The Newton method in Algorithm~\ref{alg: NWB} requires solving the linear system $(H+\|\vect{g}\|I) (\Delta\vect{\beta}) = -\vect{g}$ using the conjugate gradient (CG), where the main cost of each iteration comes from Hessian-vector products involving the transport probability matrices $\{P_k\}_{k=1}^K$. For large-scale problems, treating these matrices as dense can be computationally expensive.

Our sparsification strategy is motivated by the structure of entropic transport. When $\eta$ is small, the transport plans and the corresponding probability matrices typically contain many near-zero entries, with most mass carried by a few dominant entries. Since the Hessian structure in~\eqref{equ: hessian_diag}-\eqref{equ: hessian_nondiag} is governed by these matrices, replacing $P_k$ with a normalized sparse approximation can reduce the cost of Hessian-vector products while retaining the main curvature information. 

For each $k$, we construct $P_{k,\rho}$ by thresholding small entries of $P_k$ and renormalizing each row. The row indexed by $i^\star\in\arg\max_{i\in[n]}v_i$ is kept unchanged to preserve the spectral lower bound used in the analysis. The procedure is summarized in Algorithm~\ref{alg: Sparsifying the transport probability matrix}.

\begin{algorithm}[t]
\caption{Sparsifying the transport probability matrices}
\label{alg: Sparsifying the transport probability matrix}
\begin{algorithmic}[1]
    \STATE {\bfseries Input:} $\{P_k\}_{k=1}^K$, threshold parameter $\rho \geq 0$, $i^\star \in \arg \max_{i \in [n]} v_i$.
    \FOR{$k=1,2,...,K$}
    \STATE Construct the sparsified matrix $P_{k, \rho}$ by setting 
    $$
    [P_{k,\rho}]_{ij} \gets \frac{[P_k]_{ij}\id{[P_k]_{ij} \geq \rho}}{\sum_{l=1}^m [P_k]_{il}\id{[P_k]_{il} \geq \rho}}\ 
    \forall i \ne i^{\star},\quad \ [P_{k,\rho}]_{i^{\star} \cdot} \gets [P_k]_{i^{\star}\cdot}.
    $$ 
    \ENDFOR
    \STATE {\bfseries Output:} $\{P_{k,\rho}\}_{k=1}^K$.
\end{algorithmic}
\end{algorithm}

Given the sparsified probability matrices $\{P_{k,\rho}\}_{k=1}^K$, we define $\vect{\gamma}_{k,\rho} = (P_{k,\rho})^\top \vect{v}$ and construct $H_\rho$ by replacing $P_k$ and $\vect{\gamma}_k$ in~\eqref{equ: hessian_diag}-\eqref{equ: hessian_nondiag} with $P_{k,\rho}$ and $\vect{\gamma}_{k,\rho}$. As in Section~\ref{subsec: The Proposed Algorithm of NWB}, $H_\rho$ need not be formed explicitly; the sparse matrices $\{P_{k,\rho}\}_{k=1}^K$ are sufficient for efficient Hessian-vector products. The block-wise expression of $H_\rho$ is given in~\eqref{equ: sparse hessian construction}.

The next results show that this construction preserves the singular structure of the exact Hessian and provides a controlled approximation error.

\begin{proposition}
\label{prop: sparse_hessian_kernel}
The sparsified Hessian $H_\rho$ has the same kernel as the exact Hessian $H$, namely, $\text{ker}(H_\rho) = \text{ker}(H)$. Moreover, $H_\rho$ satisfies the same positive semidefiniteness and spectral-type properties as those stated in Propositions~\ref{prop: kernel space of hessian} and~\ref{prop: the minimum and maximum eigenvalues of the hessian}.
\end{proposition}

\begin{theorem}
\label{thm: hessian diff}
For any $\rho \geq 0$, we have
\begin{equation*}
    \|H - H_\rho\| \le 8mn \rho (\max_k w_k) \left( \frac{1}{\tau} + \frac{1}{\eta} \right).
\end{equation*}
\end{theorem}

Proposition~\ref{prop: sparse_hessian_kernel} shows that sparsification preserves the kernel structure of the exact Hessian, while Theorem~\ref{thm: hessian diff} bounds the induced perturbation. The error scales linearly with the threshold $\rho$, providing a direct trade-off between computational efficiency and second-order accuracy.

\subsection{The Proposed Algorithm}
\label{subsec: The Proposed Algorithm of SNWB}

Algorithm~\ref{alg: SNWB} summarizes the sparse Newton method for the WB problem~\eqref{opt: barycenter problem}. It follows the same framework as NWB, including the first-order warm start, barycenter update, and Armijo line search. Its main difference is in the Newton direction: instead of using the exact Hessian $H^t$, SNWB uses a `sparse' approximation $H_\rho^t$ constructed from the sparsified transport probability matrices $\{P_{k,\rho}^t\}_{k=1}^K$.

At each iteration, the sparsification threshold is chosen adaptively as $\rho^t = C_\rho \|\vect{g}^t\|$, where $C_\rho>0$ is a prescribed parameter. This choice makes the approximation more aggressive when the iterate is far from optimal and gradually reduces the sparsification error as the gradient norm decreases. Consequently, the `sparse' Hessian becomes increasingly accurate when the iterate approaches the solution, which is important for preserving the local behavior of the Newton method. 
 
\begin{algorithm}[t]
\caption{Sparse Newton method for Wasserstein Barycenter (SNWB)}
\label{alg: SNWB}
\begin{algorithmic}[1]
    \STATE {\bfseries Input:} $\vect{w} \in \Delta_K$, $\eta,\ \tau > 0$, $\{\vect{\mu}_k\}_{k=1}^K \in \Delta_m$, $\{C_k\}_{k=1}^K \in \R_+^{n \times m}$, $C_{\rho} > 0$. 
    \STATE Follow the warm-start procedure in Algorithm~\ref{alg: NWB} to construct the initial dual variable $\vect{\beta}^0$.
    \FOR{$t=0,1,...,T$}
    \STATE Perform Steps 5-10 in Algorithm~\ref{alg: NWB} to calculate the gradient $\vect{g}^t, \vect{v}^t$ and $\{P_k^t\}_{k\in [K]}$.
    \STATE Update $\rho \gets C_\rho \|\vect{g}^t\|$.
    \STATE Construct $\{P_{k,\rho}\}_{k\in [K]}$ by Algorithm~\ref{alg: Sparsifying the transport probability matrix} and compute $\vect{\gamma}_k^t \gets (P_{k,\rho})^\top \vect{v}^t$. 
    \STATE Construct the sparse Hessian $H_\rho^t$ using~\eqref{equ: sparse hessian construction}.
    \STATE Compute $\Delta \vect{\beta}^t \gets -(H_{\rho}^t + \|\vect{g}^t\|I)^{-1}\vect{g}^t$ by the conjugate gradient.
    \STATE Update $\vect{\beta}^{t+1}$ by Steps 13-14 in Algorithm~\ref{alg: NWB}.
    \ENDFOR
    \STATE {\bfseries Output:} $\vect{v}^T$.
\end{algorithmic}
\end{algorithm}

\section{Theoretical Analysis}
\label{sec: Theoretical Analysis}

This section establishes the convergence of the proposed Newton-type methods. The analysis covers both NWB (Algorithm~\ref{alg: NWB}), corresponding to $\rho=0$, and SNWB (Algorithm~\ref{alg: SNWB}), corresponding to $\rho>0$. Since the gradient is orthogonal to the kernel space of $H$, all convergence statements are formulated on the restricted subspace orthogonal to its kernel.
We first show that the iterates remain in the restricted subspace once the initial point is chosen from it.

\begin{lemma}
\label{lem: sequence restrict in the subspace}
For any $\rho \geq 0$, let $\{\vect{\beta}^{t}\}_{t \geq 0}$ be the sequence generated by Algorithm~\ref{alg: NWB} or~\ref{alg: SNWB} with $\vect{\beta}^{0} \in \text{ker}(H_\rho)^\perp$, then $\vect{\beta}^{t} \in \text{ker}(H_\rho)^\perp$ for $\forall t\in [T]$.
\end{lemma}

Building upon this invariant property, which ensures the sequence is perpetually confined to the orthogonal complement of the kernel, we are now positioned to establish the global convergence of both algorithms within this restricted subspace.

\begin{theorem}
\label{thm: convergence}
Let $\{\vect{\beta}^t\}_{t\geq 0}$ be the sequence generated by Algorithm~\ref{alg: NWB} or~\ref{alg: SNWB} with initial point $\vect{\beta}^0 \in \text{ker}(H_\rho)^\perp$. Then we have  $\|\vect{g}^t\| \rightarrow 0$ as $t \rightarrow \infty$ and $\vect{\beta}^t \rightarrow \vect{\beta}^\star$, where $\vect{\beta}^\star$ is the unique minimizer of $\mathcal{L}(\vect{\beta})$ in the restricted subspace $\text{ker}(H_\rho)^\perp$.
\end{theorem}

While Theorem~\ref{thm: convergence} guarantees that the iterates asymptotically approach the unique optimal solution from any valid starting point, characterizing the algorithm's asymptotic efficiency requires analyzing its local behavior. The following theorem formally addresses this by demonstrating that, once the iterates enter a sufficiently small neighborhood of the optimum, the proposed methods attain a strict local quadratic rate of convergence.

\begin{theorem}
\label{thm: local_quadratic}
Let $\{\vect{\beta}^t\}_{t\geq 0}$ be generated by Algorithm~\ref{alg: NWB} or~\ref{alg: SNWB} and let $\vect{\beta}^\star$ denote the unique optimal solution of $\lyp(\vect{\beta})$ over the restricted subspace $\ker(H_\rho)^\perp$. If the Armijo parameter satisfies $\sigma \in (0, 1/2)$, when the iterates $\vect{\beta}^t$ are sufficiently close to $\vect{\beta}^\star$ and the Armijo line search accepts $s^t=1$, there exists a constant $M>0$ such that
\begin{equation*}
    \|\vect{\beta}^{t+1} - \vect{\beta}^\star\| \le M \|\vect{\beta}^t - \vect{\beta}^\star\|^2.
\end{equation*}
\end{theorem}

Finally, we provide an explicit bound on the number of iterations required to enter the pure Newton phase, where $s^t=1$ is accepted, and local quadratic convergence holds.
\begin{proposition}
\label{prop: phase_transition_complexity}
Assume $m=n, \eta = \mathcal{O}(\tau)$ and the barycentric weights satisfy $w_k = \Theta(1/K)$. For a given regularization parameter $\eta > 0$, the number of iterations required for the exact Newton method to transition into the region of pure quadratic convergence is bounded by
\begin{equation*}
    \mathcal{O}\left( \left( \frac{K^3}{\eta} + \frac{n\eta}{K} \right) \exp\left(\frac{\Delta}{\eta}\right) \right),
\end{equation*}
where $\Delta$ is a problem-dependent constant.
\end{proposition}
This complexity bound depends entirely on the intrinsic problem parameters, indicating that the algorithm guarantees a transition into a phase of rapid, quadratic convergence after a finite number of damped iterations.

\section{Numerical Experiments}
\label{sec: Numerical Experiments}

This section presents numerical experiments on synthetic and real datasets to assess the performance of the proposed algorithms. Unless otherwise specified, we consider the regularization setting $\tau=\eta$ and compare against three representative fixed-support Wasserstein barycenter solvers: IBP~\cite{benamou2015iterative}, FastIBP~\cite{lin2020fixed}, and SmWB~\cite{cuturi2018semidual}. The experiments are implemented in Python~3.12 and run on a MacBook Pro with macOS~13.0, an Apple M2 Pro processor, and 16 GB of RAM. In implementation, we take $C_\rho = 10^5 \cdot \frac{\eta}{\sqrt{n}m}$. Additional experiments, including ablation studies on $\eta$ and the sparsification level, as well as results for the setting $\tau=\eta/2$, are reported in Appendix~\ref{app: Additional Experiment Details}.

\subsection{Synthetic Data}
\label{subsec: Synthetic Data}

We first examine the qualitative behavior of the proposed methods on the classical nested ellipses benchmark~\cite{cuturi2014fast}. The dataset consists of $K=5$ discrete probability measures represented as grayscale images on $40\times40$ and $64\times64$ pixel grids. Each image contains a pair of nested ellipses with randomly generated centers, orientations, and axis lengths. The ground cost is defined as the squared Euclidean distance between pixel locations. Unless otherwise specified, we use uniform weights $w_k=1/5$ for $k \in [5]$ and set $\eta=\tau=10^{-3}$.

Figures~\ref{fig: Nested ellipse benchmark. Top: Dataset of 5 nested ellipses with 40 pixel grids. Bottom: Barycenters computed by different methods.} and~\ref{fig: Nested ellipse benchmark. Top: Dataset of 5 nested ellipses with 64 pixel grids. Bottom: Barycenters computed by different methods} compare the barycenters computed by different methods under the two grid resolutions. The proposed NWB and SNWB  produce barycenters that preserve the nested elliptical structure with clear contours and limited background artifacts. In contrast, SmWB exhibits noticeable background noise, whereas first-order methods yield slightly more diffuse barycenters. Although the dense NWB method achieves high-quality solutions, its full Hessian computations incur additional costs. SNWB alleviates this bottleneck by leveraging Hessian sparsification, resulting in a substantial reduction in runtime while maintaining comparable barycenter quality. These results suggest that the sparse Newton strategy improves scalability without compromising the geometric fidelity of the computed barycenter. Additional synthetic numerical experiments are provided in Appendix~\ref{app: Additional Synthetic Experiments}.

\begin{figure*}[t]
\centering

\includegraphics[width=0.16\textwidth]{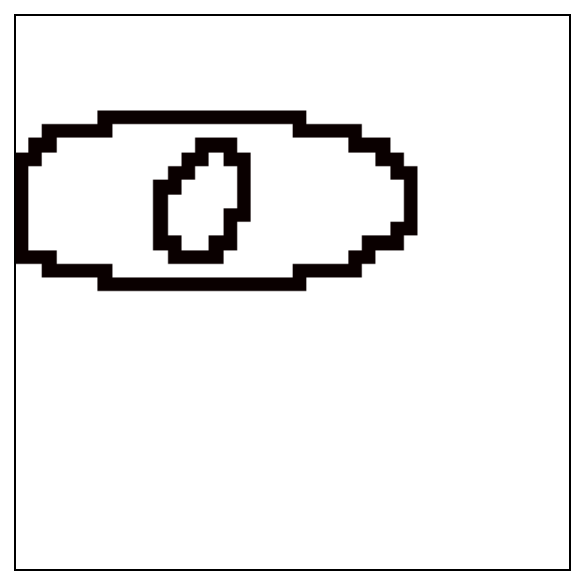}\hfill
\includegraphics[width=0.16\textwidth]{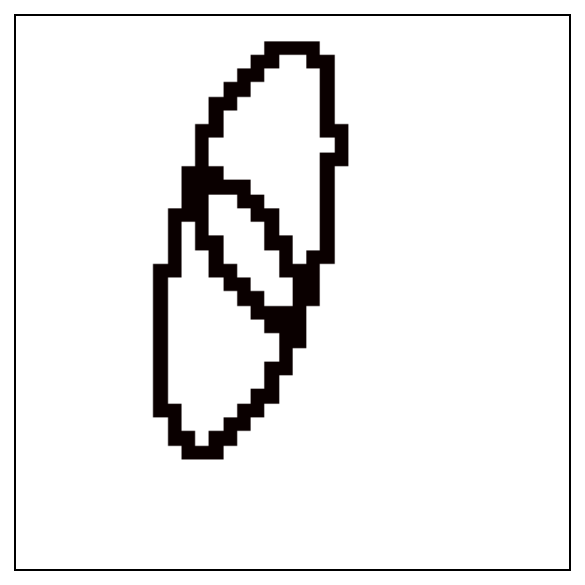}\hfill
\includegraphics[width=0.16\textwidth]{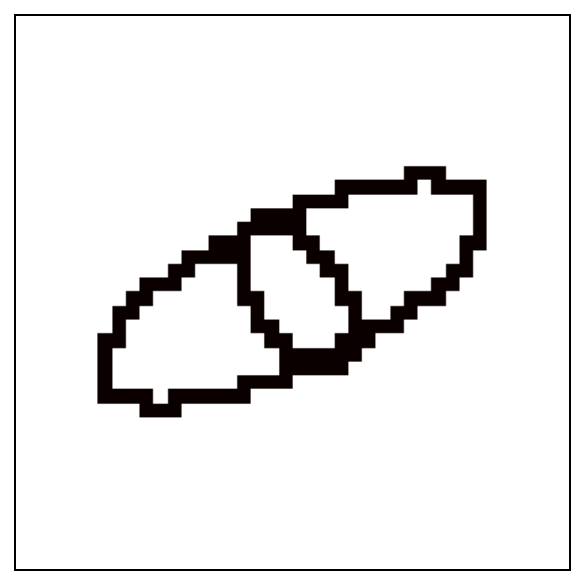}\hfill
\includegraphics[width=0.16\textwidth]{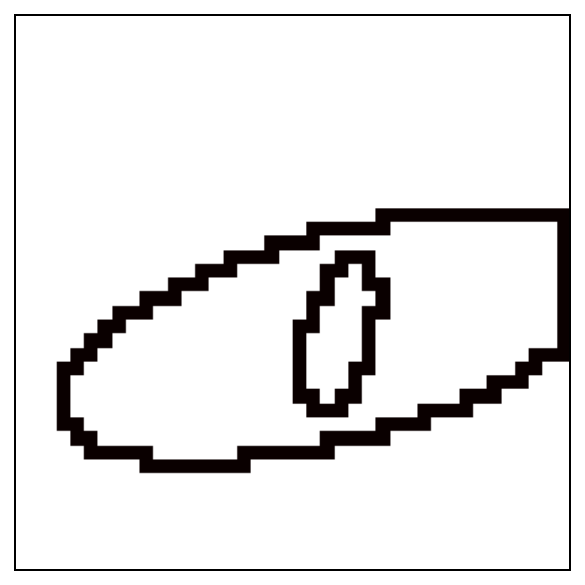}\hfill
\includegraphics[width=0.16\textwidth]{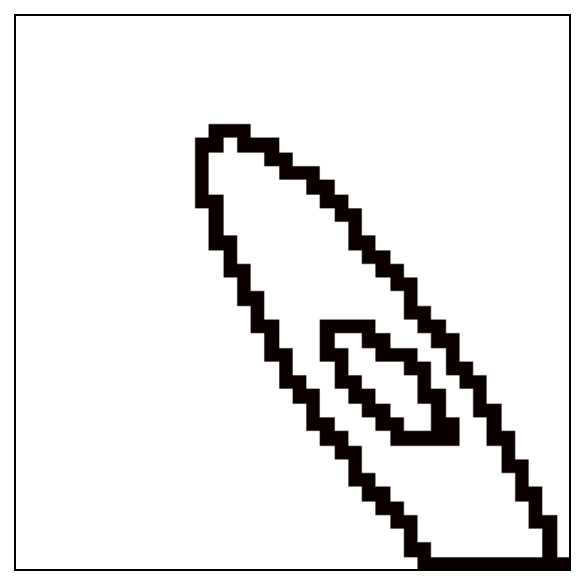}

\includegraphics[width=0.16\textwidth]{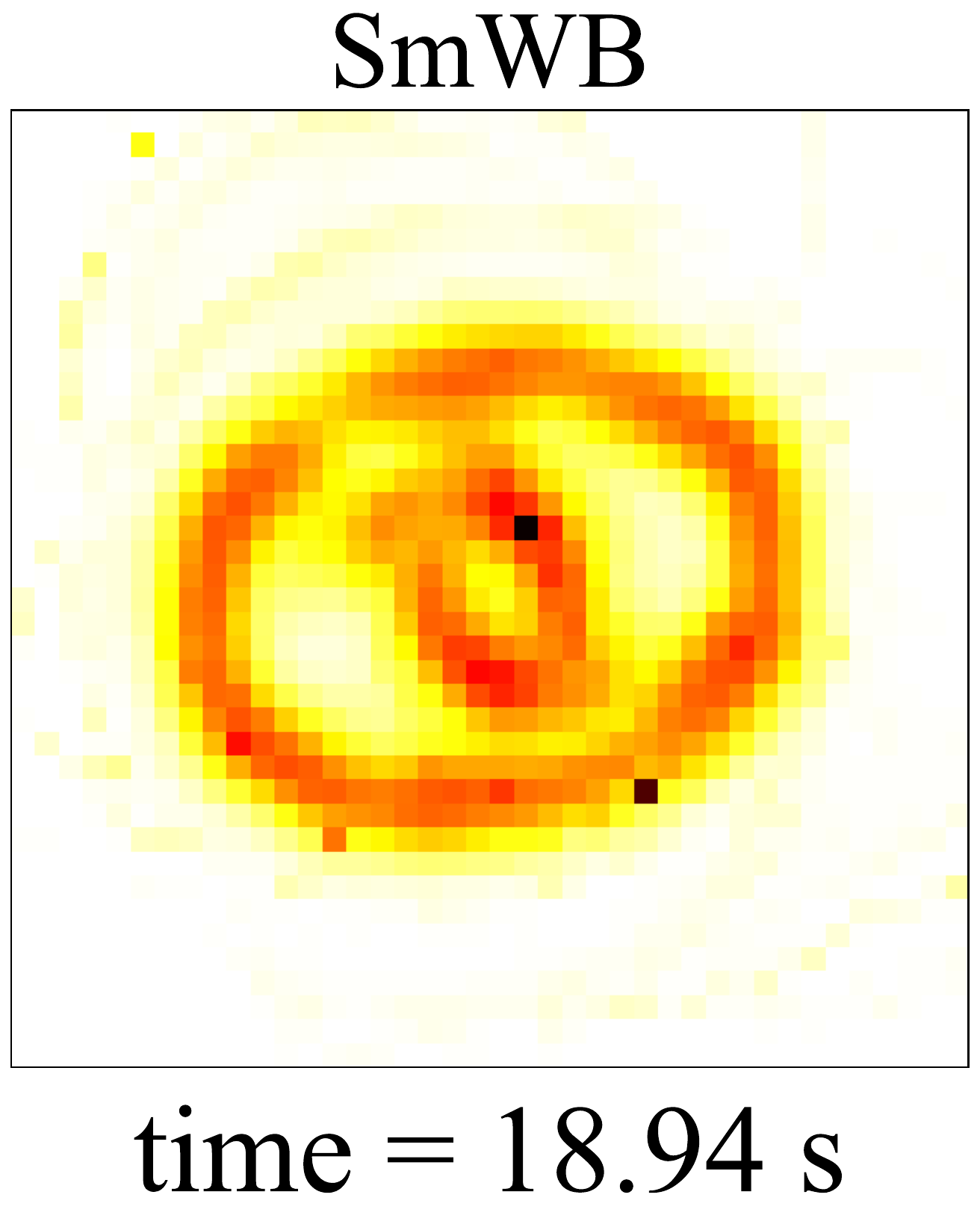}\hfill
\includegraphics[width=0.16\textwidth]{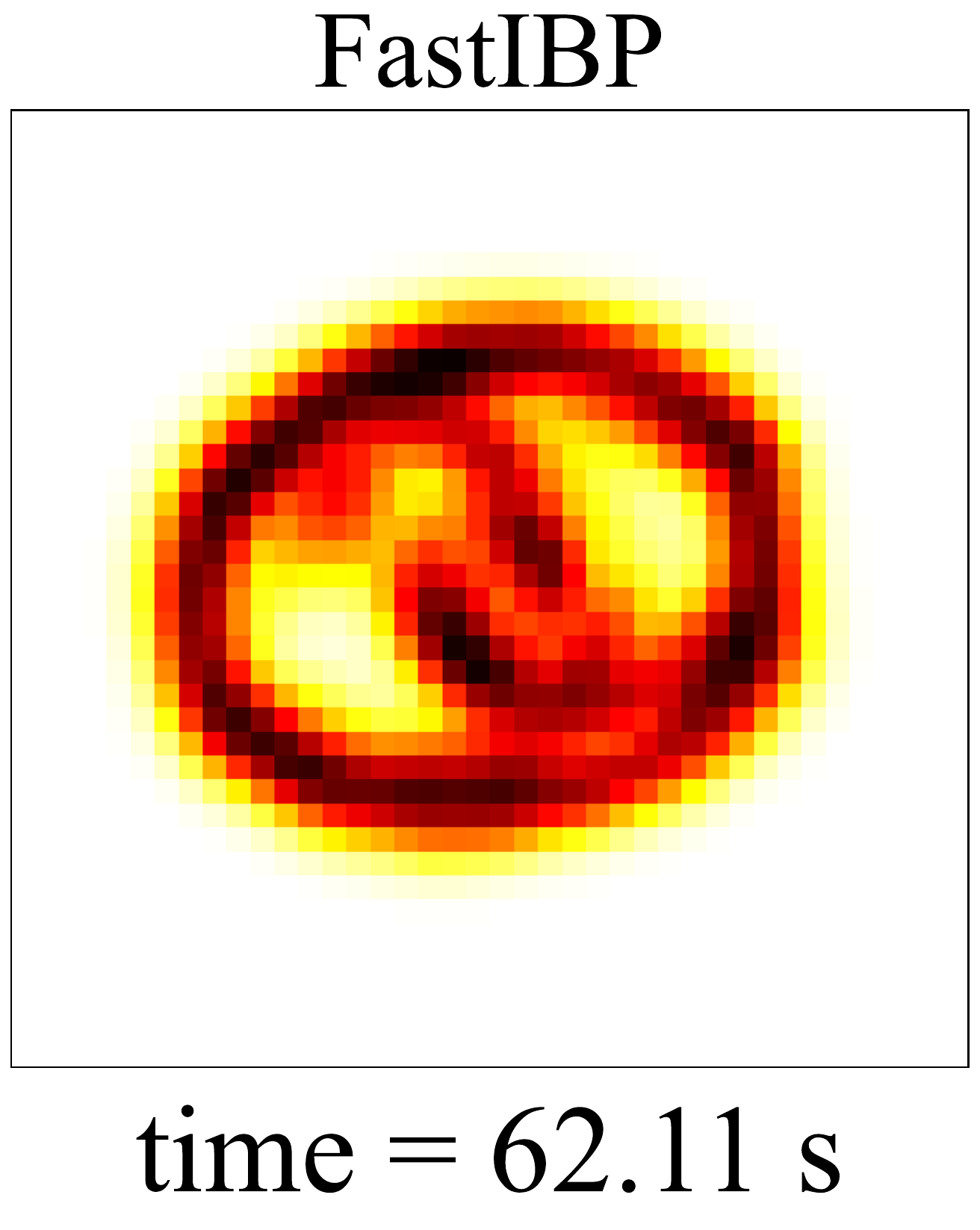}\hfill
\includegraphics[width=0.16\textwidth]{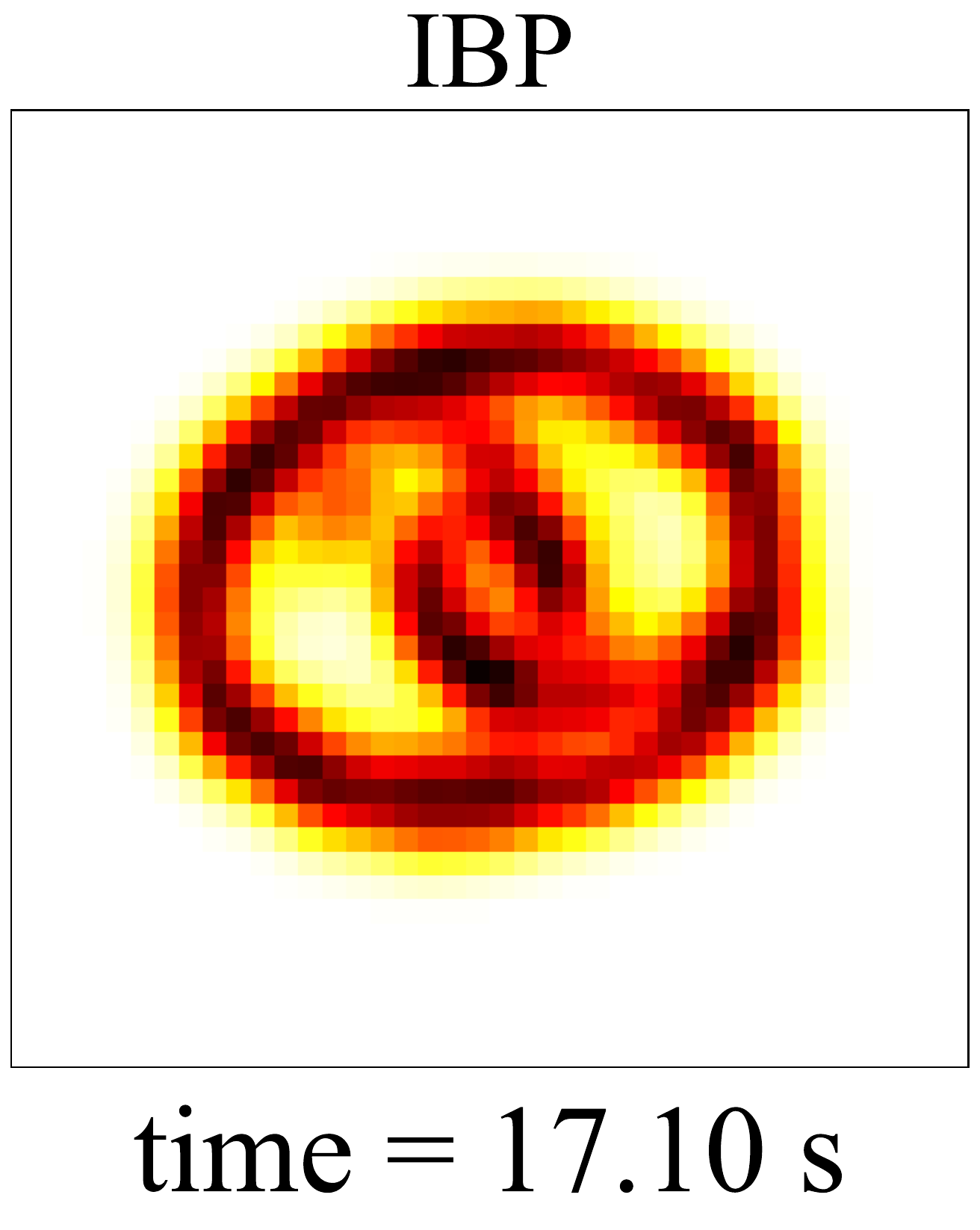}\hfill
\includegraphics[width=0.16\textwidth]{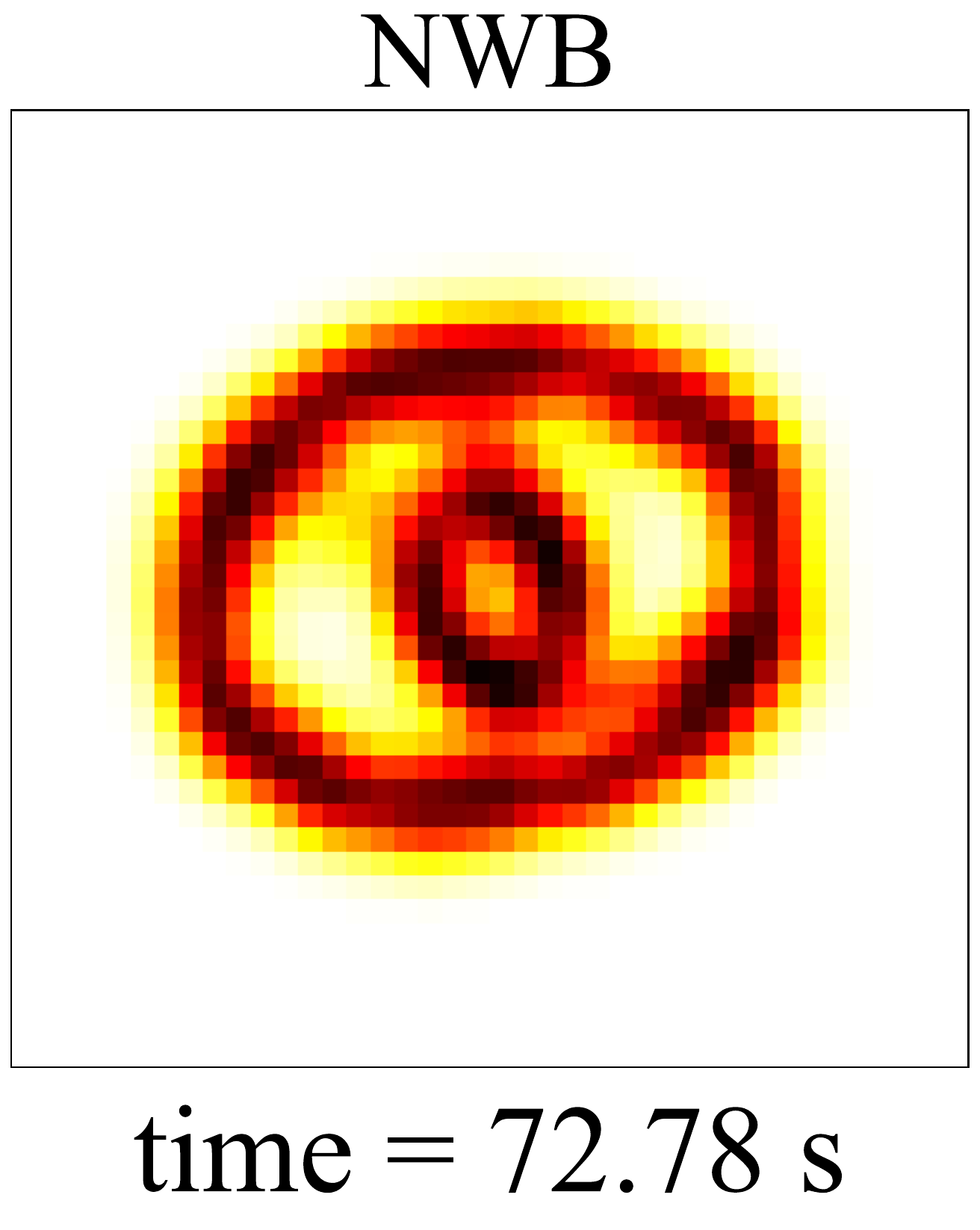}\hfill
\includegraphics[width=0.16\textwidth]{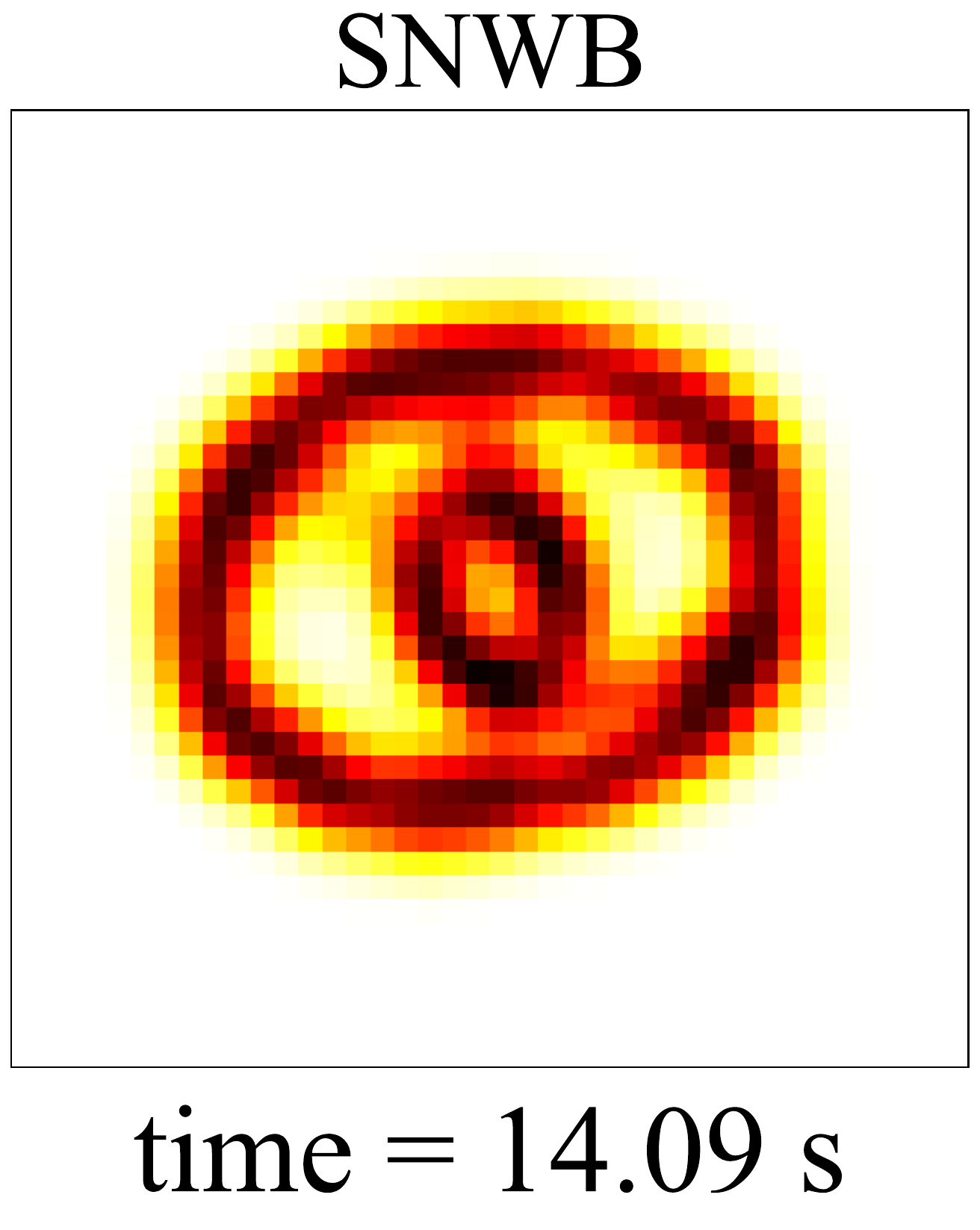}
\caption{Nested ellipses benchmark on a $40\times40$ pixel grid. Top: Dataset of 5 nested ellipses. Bottom: Barycenters computed by different methods.
}
\label{fig: Nested ellipse benchmark. Top: Dataset of 5 nested ellipses with 40 pixel grids. Bottom: Barycenters computed by different methods.}
\end{figure*}

\begin{figure*}[t]
\centering

\includegraphics[width=0.16\textwidth]{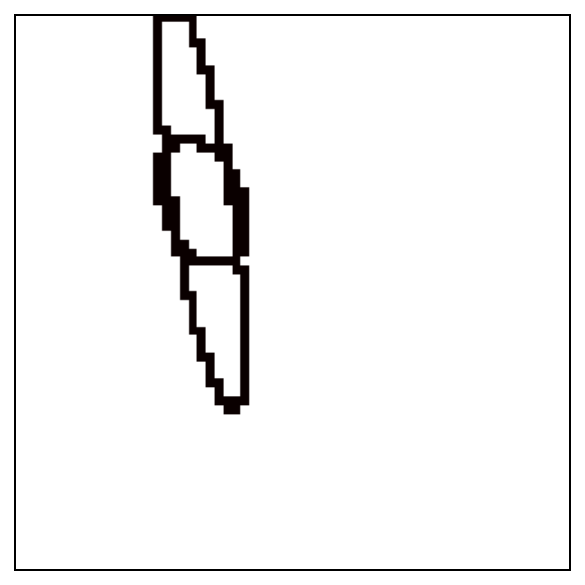}\hfill
\includegraphics[width=0.16\textwidth]{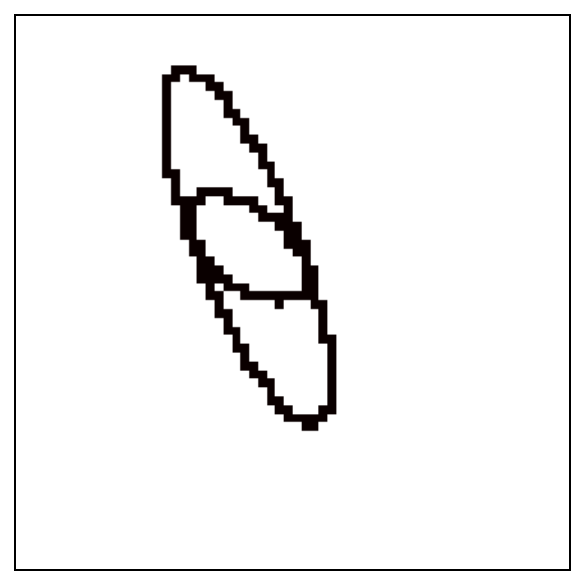}\hfill
\includegraphics[width=0.16\textwidth]{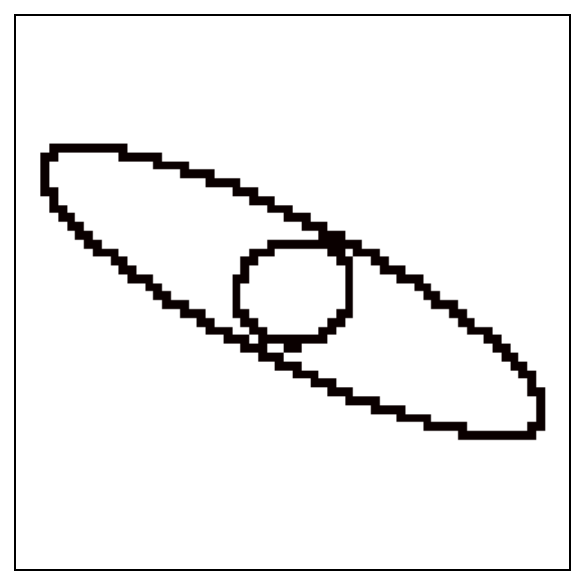}\hfill
\includegraphics[width=0.16\textwidth]{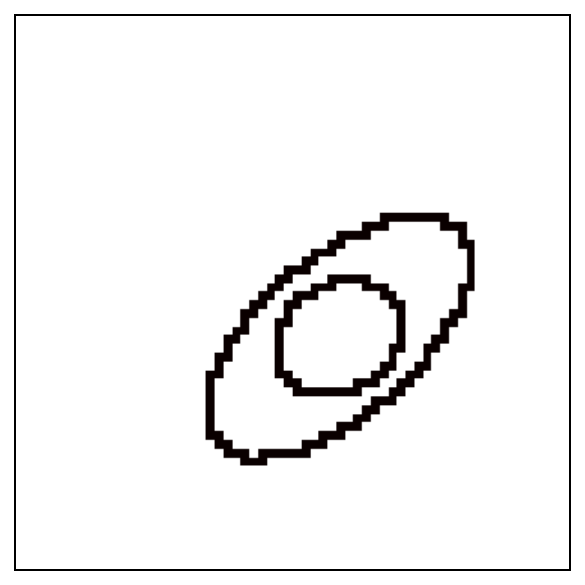}\hfill
\includegraphics[width=0.16\textwidth]{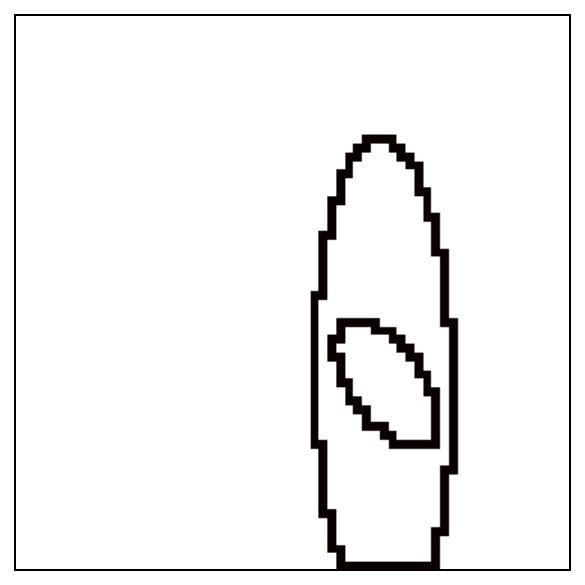}

\includegraphics[width=0.16\textwidth]{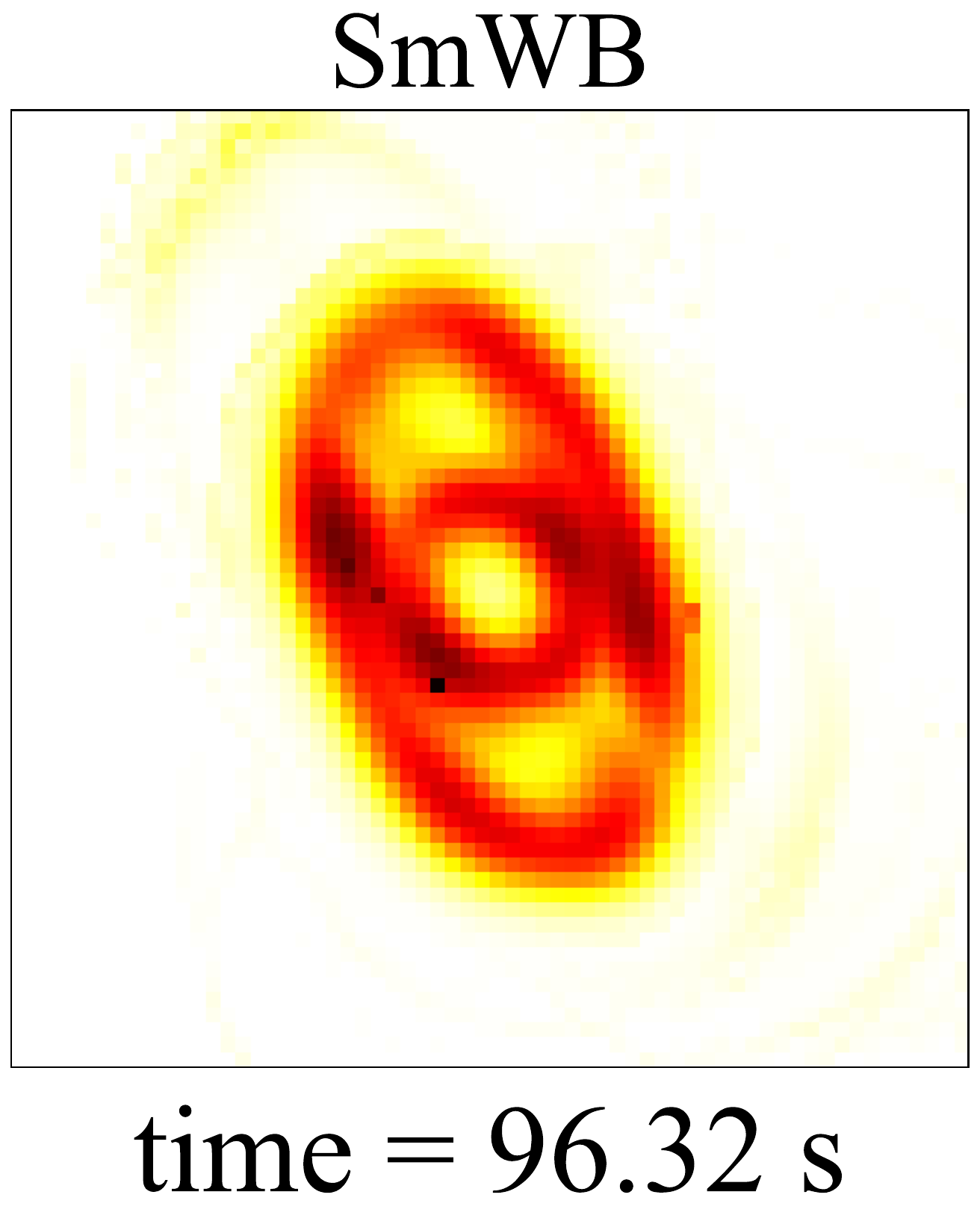}\hfill
\includegraphics[width=0.16\textwidth]{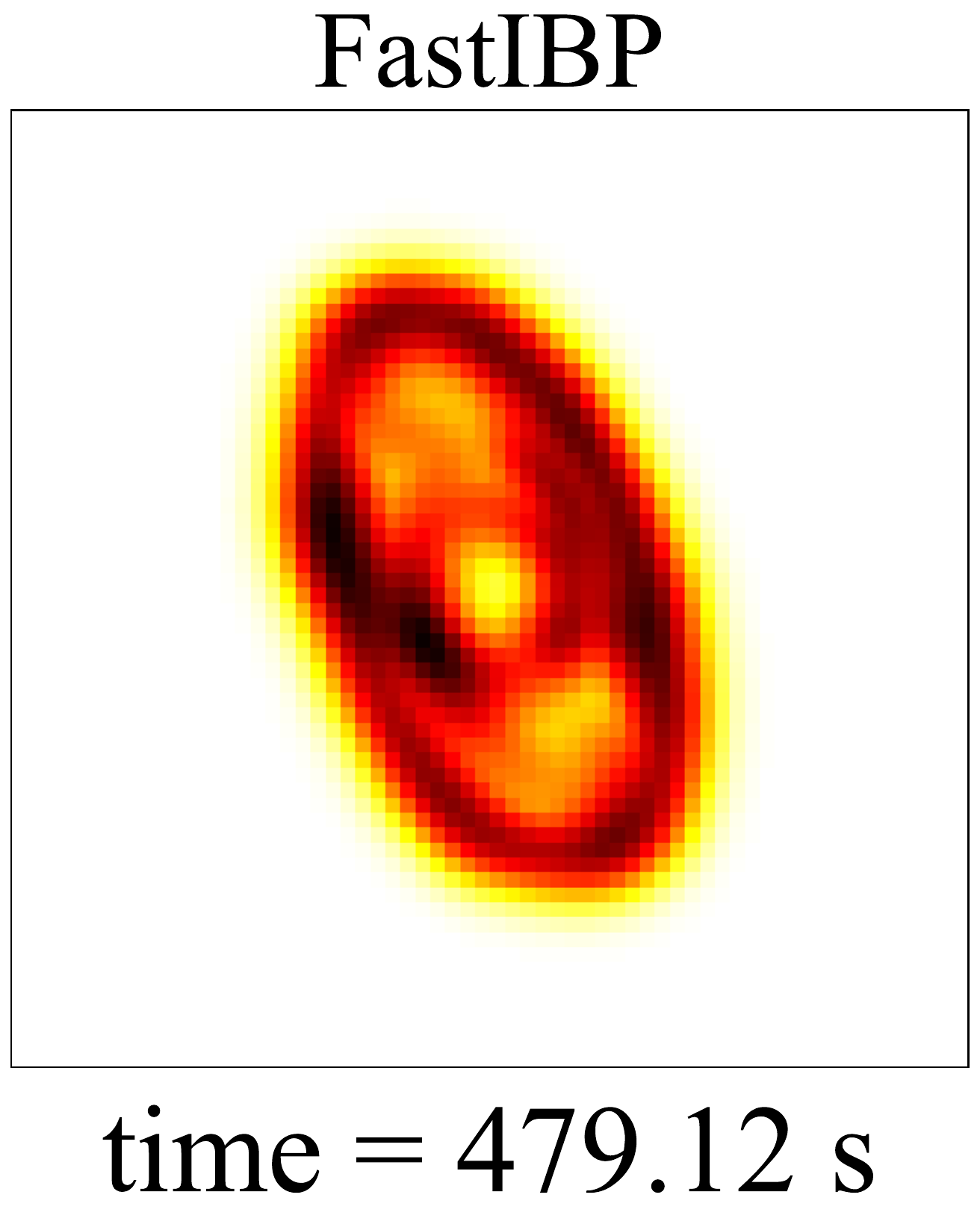}\hfill
\includegraphics[width=0.16\textwidth]{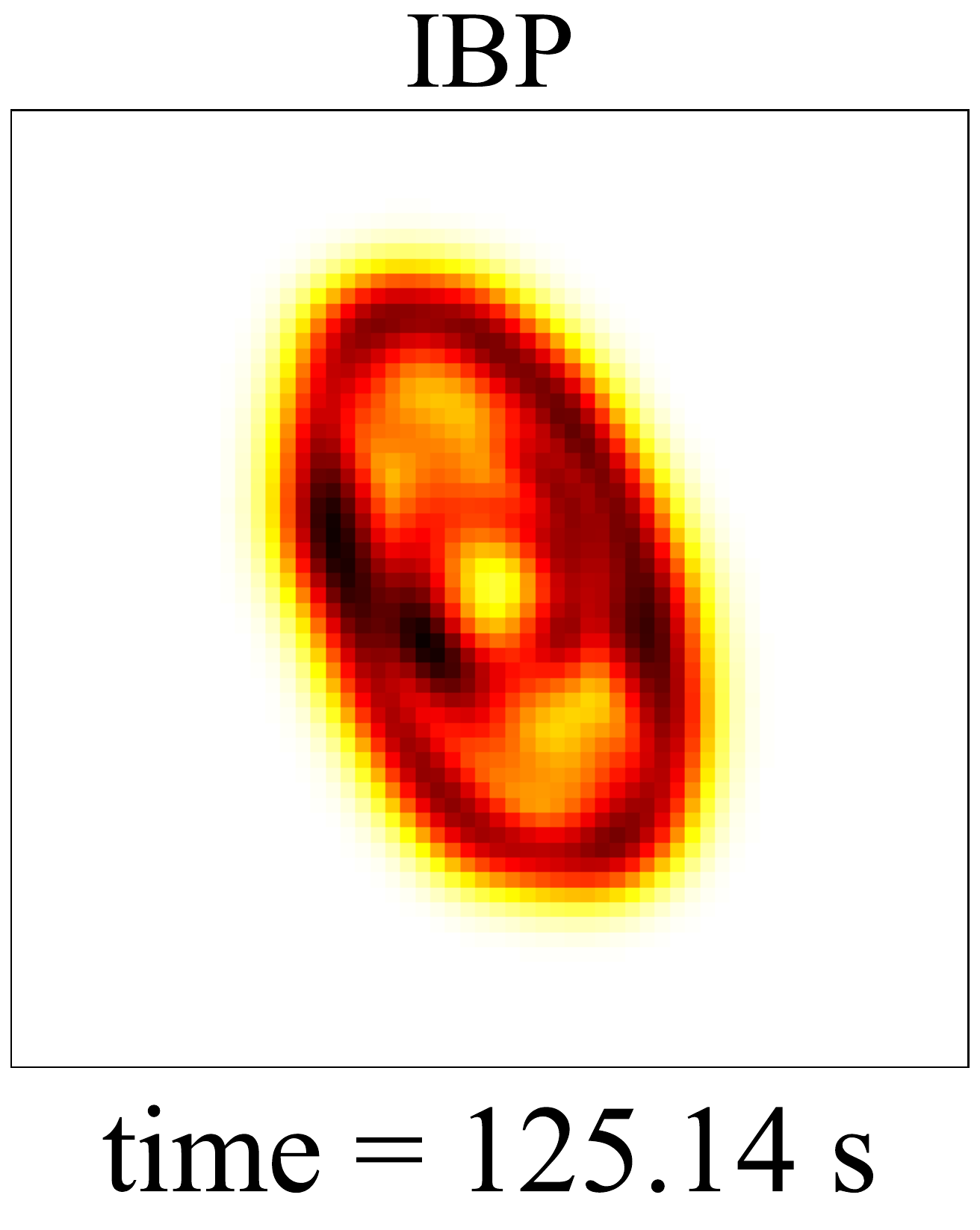}\hfill
\includegraphics[width=0.16\textwidth]{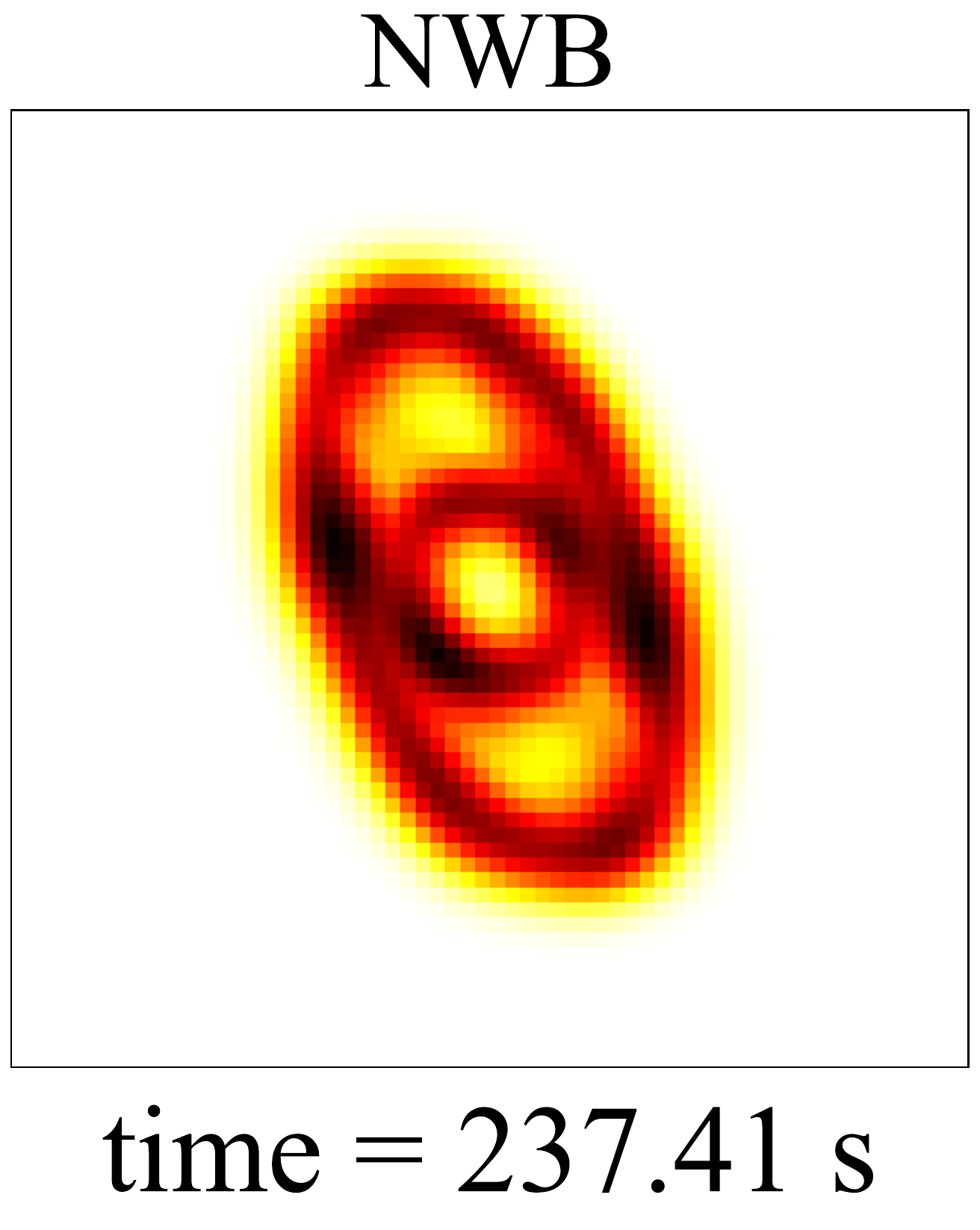}\hfill
\includegraphics[width=0.16\textwidth]{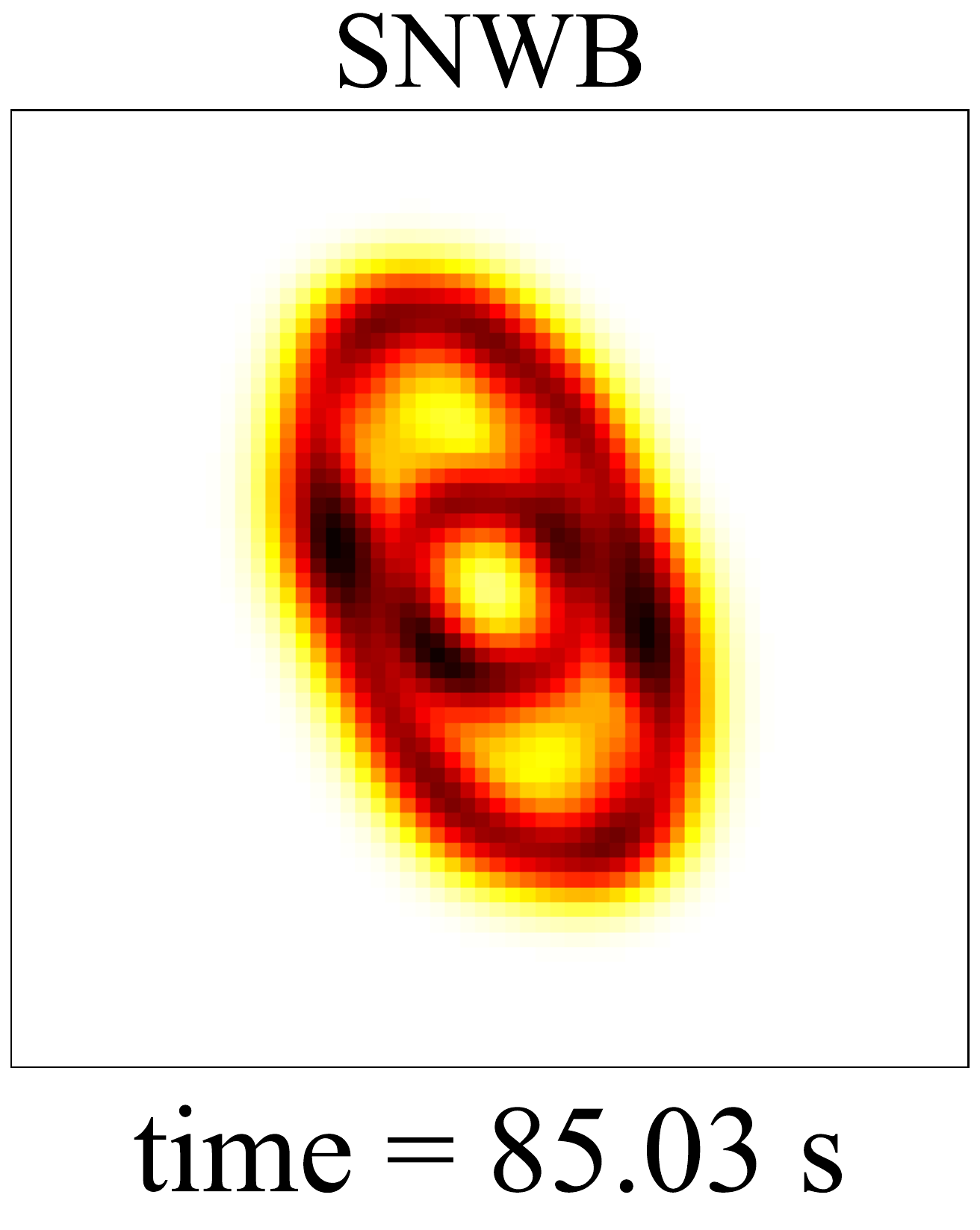}
\caption{Nested ellipses benchmark on a $64 \times 64$ pixel grid. Top: Dataset of 5 nested ellipses. Bottom: Barycenters computed by different methods.
}
\label{fig: Nested ellipse benchmark. Top: Dataset of 5 nested ellipses with 64 pixel grids. Bottom: Barycenters computed by different methods}
\end{figure*}

\subsection{Real Data}
\label{subsec: Real Data}

We next evaluate the proposed methods on real image datasets with meaningful spatial structures. We consider MNIST~\cite{lecun2002gradient} and Fashion-MNIST~\cite{xiao2017fashion}, two standard benchmarks consisting of grayscale images of handwritten digits and fashion products, respectively. 

For each dataset, we randomly select $K=10$ images from a prescribed class and resize each image to a $70\times70$ pixel grid. Each image is vectorized and normalized into a probability vector $\vect{\mu}_k$, and the cost matrix is constructed using the squared Euclidean distances between pixel locations. We use uniform weights $w_k=1/10$ for $k \in [10]$. The regularization parameters are set to $\eta=\tau=7\times10^{-4}$ for MNIST and $\eta=\tau=8\times10^{-4}$ for Fashion-MNIST. To assess convergence accuracy, we compute a high-precision reference solution $\vect{v}^{\star}$ using SNWB with a stopping criterion $\|\vect{g}\|<10^{-7}$.

Figure~\ref{fig:fashionmnist_and_mnist_performance_evaluation} reports the results on MNIST and Fashion-MNIST, respectively. The dense NWB method converges faster than FastIBP and the L-BFGS-B-based SmWB method, demonstrating the benefits of using second-order information. However, inverting the full Hessian remains expensive, making it slower than IBP in some cases. By contrast, SNWB substantially reduces the cost of the Newton updates through Hessian sparsification and achieves the fastest convergence among the tested methods. These results confirm that the proposed sparsification strategy improves computational efficiency while preserving the accuracy and stability of the Newton framework. Additional comparisons using other measures are reported in Appendix~\ref{app:real-data-residual}.

\begin{figure*}[!t]
\centering
\includegraphics[width=0.23\textwidth, trim=10 10 10 10, clip]{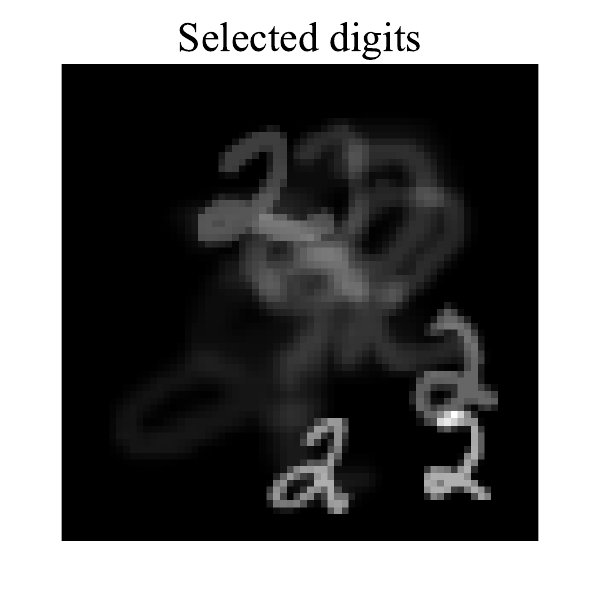}\hfill
\includegraphics[width=0.23\textwidth, trim=10 10 10 10, clip]{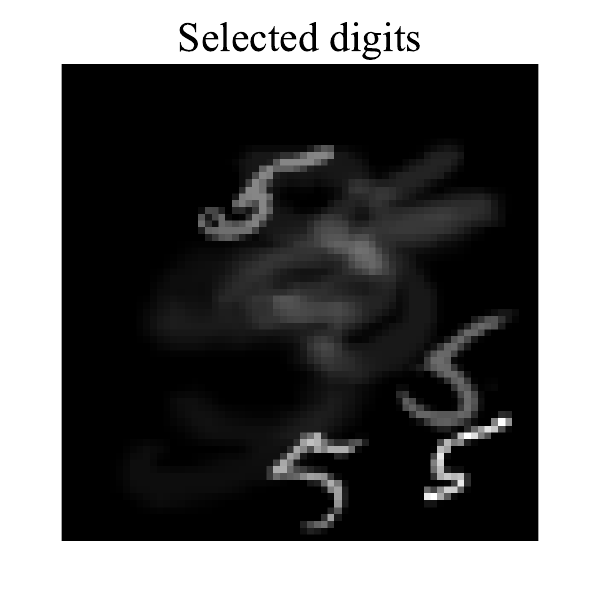}\hfill
\includegraphics[width=0.23\textwidth, trim=10 10 10 10, clip]{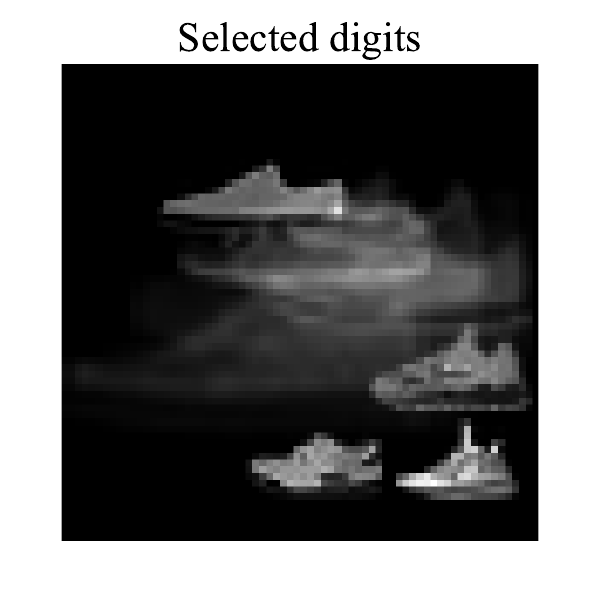}\hfill
\includegraphics[width=0.23\textwidth, trim=10 10 10 10, clip]{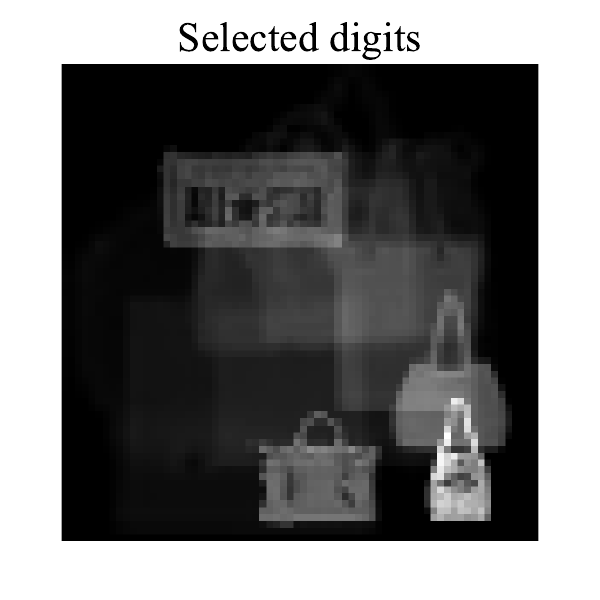}

\includegraphics[width=0.23\textwidth, trim=10 10 10 10, clip]{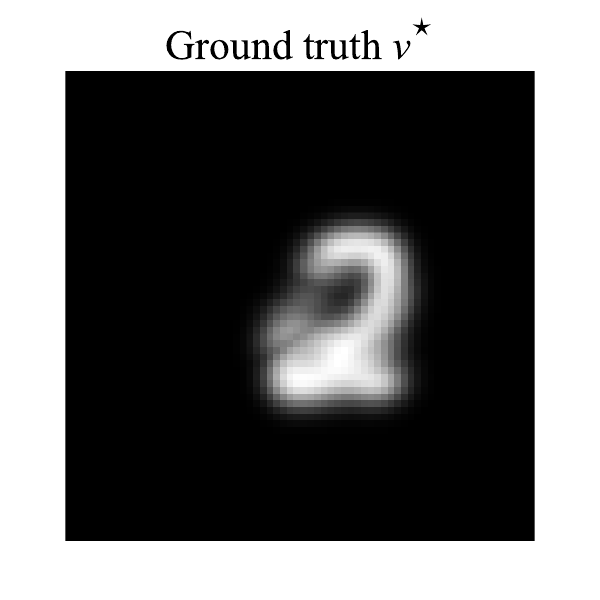}\hfill
\includegraphics[width=0.23\textwidth, trim=10 10 10 10, clip]{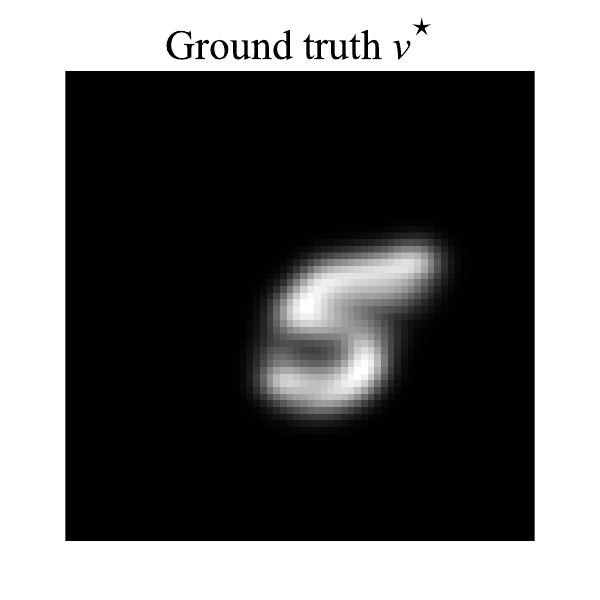}\hfill
\includegraphics[width=0.23\textwidth, trim=10 10 10 10, clip]{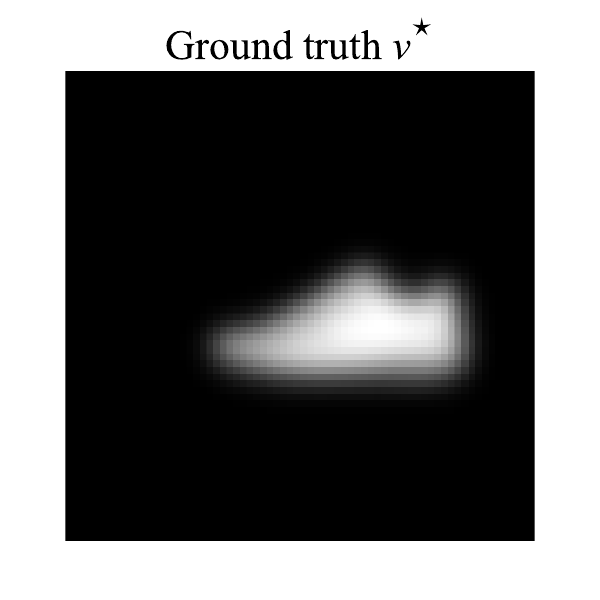}\hfill
\includegraphics[width=0.23\textwidth, trim=10 10 10 10, clip]{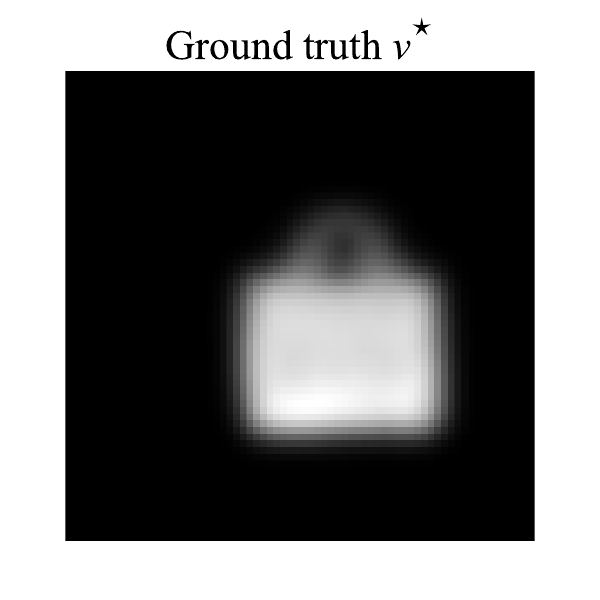}

\includegraphics[width=0.23\textwidth, trim=10 10 10 10, clip]{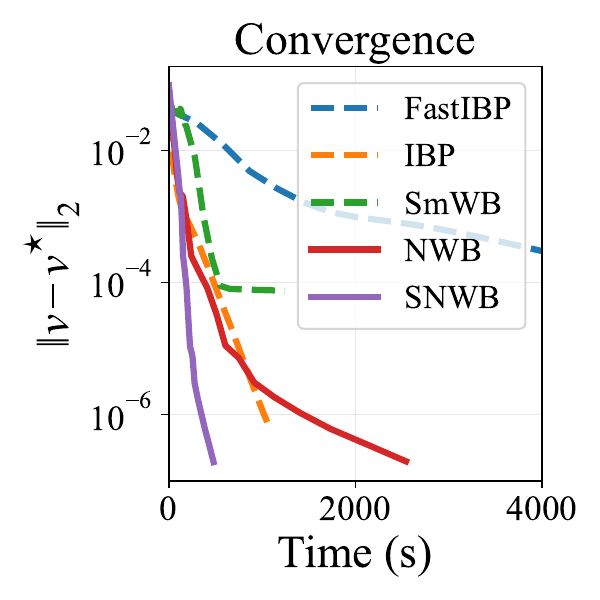}\hfill
\includegraphics[width=0.23\textwidth, trim=10 10 10 10, clip]{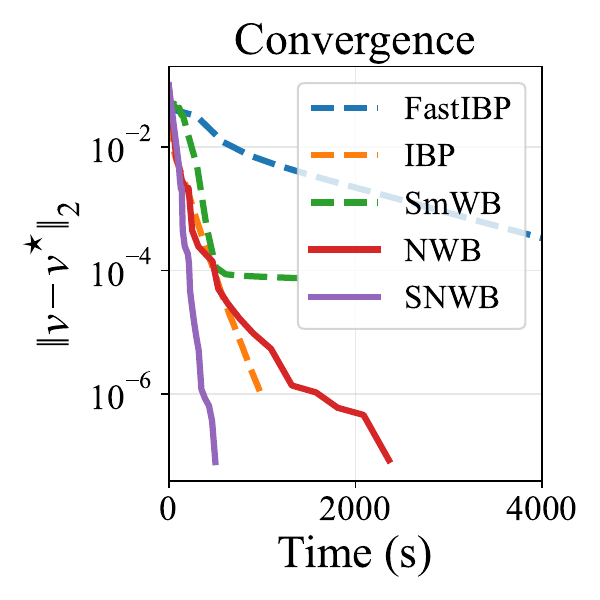}\hfill
\includegraphics[width=0.23\textwidth, trim=10 10 10 10, clip]{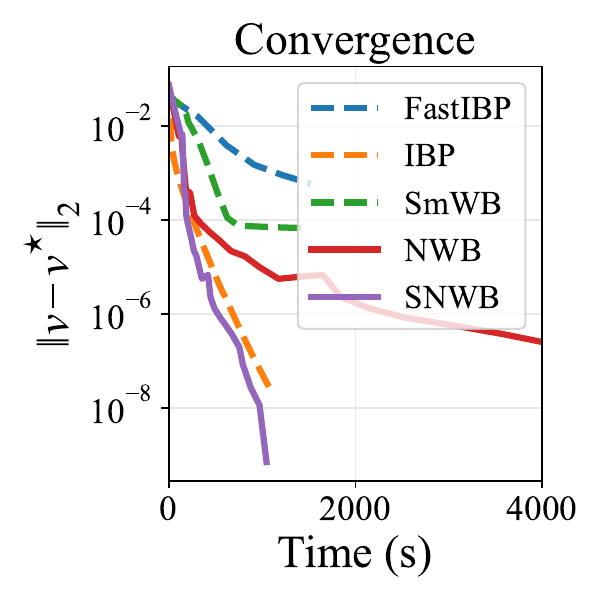}\hfill
\includegraphics[width=0.23\textwidth, trim=10 10 10 10, clip]{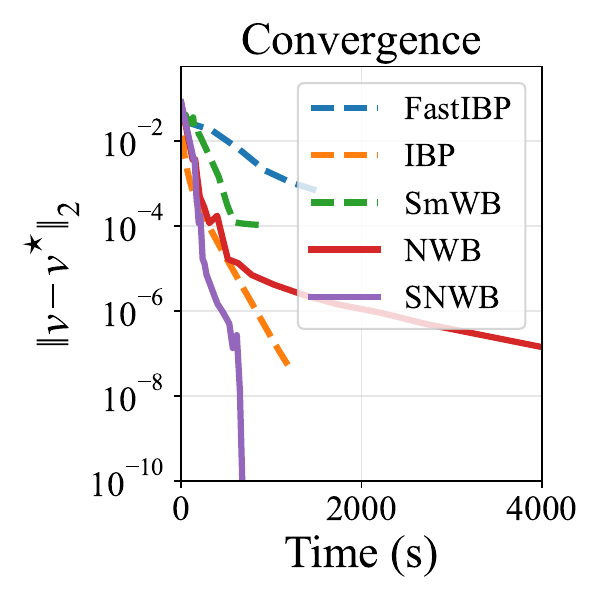}
\caption{Qualitative and quantitative performance evaluation on MNIST and Fashion-MNIST benchmarks. From top to bottom, the rows report the selected input samples, the corresponding ground-truth barycenter $v^\star$, and the convergence curves of different methods.}
\label{fig:fashionmnist_and_mnist_performance_evaluation}
\end{figure*}

\section{Conclusion}
\label{sec: conclusion}

We studied the fixed-support doubly entropic Wasserstein barycenter problem and reformulated it as a smooth, unconstrained, convex optimization problem via the semi-dual representation of entropic optimal transport. Based on explicit gradient and Hessian formulas, we developed an exact Newton method (Algorithm~\ref{alg: NWB}) and its sparse variant (Algorithm~\ref{alg: SNWB}) that sparsifies the transport probability matrices to accelerate Hessian-vector products. Theoretical analysis establishes structural properties, approximation bounds, and convergence results for the proposed methods. Numerical experiments on synthetic and real datasets show that SNWB achieves faster convergence than representative state-of-the-art solvers while maintaining high solution accuracy. Limitations of the proposed methods and possible directions for future work are discussed in Appendix~\ref{app: limitations}.

\newpage
\bibliography{reference}
\bibliographystyle{plain}

\newpage
\appendix

\section{Limitations}
\label{app: limitations}

The proposed methods have two main limitations. First, the conditioning of the Newton system depends on the transport probability matrices; when some entries become very small, the system may be ill-conditioned and require more conjugate gradient iterations. Second, the efficiency of SNWB relies on the quality of the sparse approximation. If the transport probability matrices remain dense, sparsification may provide limited computational savings. Developing effective preconditioners and adaptive sparsification rules is therefore an important direction for future work.
\section{Proofs of Theoretical Results}
\label{app: Proofs of Theoretical Results}
This section details the proofs of the theoretical results. To facilitate the subsequent analysis, we first introduce a block-wise decomposition of the `sparsified' Hessian matrix $H_\rho$. 

Given the sparsified transport probability matrices $\{P_{k,\rho}\}_{k\in [K]}$, we define the corresponding marginal vectors as:
\begin{equation}
\label{equ: gamma}
    \vect{\gamma}_{k,\rho} = (P_{k,\rho})^\top \vect{v},
\end{equation}
and introduce the auxiliary block matrices $M_{kr,\rho}$ and $N_{kk,\rho}$ such that:
\begin{equation}
    \label{equ: hessian diagonal decomposition}
    M_{kr,\rho} = (P_{k,\rho})^\top \text{diag}(\vect{v}) P_{r,\rho} - \vect{\gamma}_{k,\rho} (\vect{\gamma}_{r,\rho})^\top , \quad N_{kk,\rho} = \text{diag}(\vect{\gamma}_{k,\rho}) - (P_{k,\rho})^\top \text{diag}(\vect{v}) P_{k,\rho}.
\end{equation}

Using these components, the $(k, r)$-th block of the `sparsified' Hessian matrix $H_\rho$, originally defined in \eqref{equ: hessian_diag} and \eqref{equ: hessian_nondiag}, can be concisely rewritten as:
\begin{equation}
\label{equ: sparse hessian construction}
    H_{kr,\rho} = \frac{w_k w_r}{\tau} M_{kr,\rho} + \frac{w_k}{\eta}N_{kk,\rho}\id{k=r}.
\end{equation}

As a special case, when $\rho=0$, the sparsification step leaves the transport probability matrices unchanged, i.e., $P_{k,\rho}=P_k$ for all $k\in[K]$. Consequently,
\[
    N_{kk,\rho}=N_{kk},\qquad
    M_{kr,\rho}=M_{kr},\qquad
    H_\rho=H,
\]
and the `sparsified' Hessian reduces to the exact Hessian.

\subsection{Structural and Spectral Properties of the Sparsified Hessian Matrix}

In this subsection, we characterize the kernel space and establish bounds on the eigenvalues of the `sparsified' Hessian matrix $H_\rho$, including $\rho = 0$ and $\rho \geq 0$. Consequently, the theoretical results derived here encompass the properties of the exact Hessian as a special case, allowing us to recover the standard propositions (e.g., Proposition~\ref{prop: kernel space of hessian}, \ref{prop: the minimum and maximum eigenvalues of the hessian}, and \ref{prop: sparse_hessian_kernel}).

\begin{theorem}
\label{thm: kernel space of hessian}
    For any $\rho \geq 0$, the sparsified Hessian matrix $H_\rho$ is symmetric and positive semidefinite. Moreover, its kernel is exactly the $K$-dimensional subspace of block-wise constant vectors, namely, 
    \begin{equation}
    \label{equ: kernel space}
        \ker(H_\rho) = \left\{ \vect{d} = (\vect{d}_1^\top, \dots, \vect{d}_K^\top)^\top \in \mathbb{R}^{mK} \mid \vect{d}_k = c_k \1_m, \, c_k \in \mathbb{R}, \, \forall k \in [K] \right\}.
    \end{equation}
    Furthermore, the gradient $\vect{g}$ is strictly orthogonal to this kernel, i.e., $\nabla \lyp(\beta) \in \ker(H_\rho)^\perp$. Consequently, the associated Newton system is consistent.
\end{theorem}

\begin{proof}
By definition, each block $H_{kk,\rho}$ is symmetric and $H_{kr,\rho}^\top=H_{rk,\rho}$; therefore, the full matrix $H_\rho$ is symmetric.

For any arbitrary vector $\vect{d} = [\vect{d}_1^\top, \vect{d}_2^\top, \dots , \vect{d}_K^\top]^\top \in \mathbb{R}^{mK}$, the quadratic form evaluates to:
\begin{equation*}
\begin{aligned}
Q = \vect{d}^\top H_\rho \vect{d} = \sum_{k=1}^K \sum_{r=1}^K (\vect{d}_k)^\top H_{kr,\rho} \vect{d}_r = \underbrace{\sum_{k=1}^K \frac{w_k}{\eta} (\vect{d}_k)^\top N_{kk,\rho} \vect{d}_k}_{Q_{\eta}} + \underbrace{\sum_{k=1}^K \sum_{r=1}^K \frac{w_k w_r}{\tau}(\vect{d}_k)^\top M_{kr,\rho} \vect{d}_r}_{Q_{\tau}}.
\end{aligned}    
\end{equation*}

We first analyze the term $Q_\eta$. Define $W_{k,\rho} = P_{k,\rho}^\top \text{diag}(\vect{v}) P_{k,\rho}$. Because $P_{k,\rho}$ is row-stochastic ($P_{k,\rho} \1_m = \1_n$), we have $W_{k,\rho} \1_m = P_{k,\rho}^\top \vect{v} = \vect{\gamma}_{k,\rho}$. Therefore, $N_{kk,\rho} = \text{diag}(W_{k,\rho} \1_m) - W_{k,\rho}$. It follows that:
\begin{equation*}
    Q_{\eta} = \sum_{k=1}^K \frac{w_k}{\eta} \vect{d}_k^\top N_{kk,\rho} \vect{d}_k = \sum_{k=1}^K \frac{w_k}{2\eta} \sum_{j=1}^m \sum_{l=1}^m [W_{k,\rho}]_{jl} ([\vect{d}_k]_{j} - [\vect{d}_k]_{l})^2 \geq 0.
\end{equation*}
Since all entries of $\vect{v}$ are strictly positive, and by the sparsification rule there exists an index $i^\star$ such that $[P_{k,\rho}]_{i^{\star}j}>0$ for all $j\in [m]$, it follows from the definition of $W_{k,\rho}$ that $[W_{k,\rho}]_{ij}>0$ for all $i,j\in [m]$. Thus, $Q_{\eta} = 0$ if and only if $\vect{d}_k = c_k \1_m$ for some constants $c_k \in \mathbb{R}$.

Next, we evaluate $Q_\tau$. By the definition of $M_{kr,\rho}$, we obtain:
\begin{equation}
\label{equ: Q_tau}
\begin{aligned}
Q_{\tau} & = \frac{1}{\tau}\sum_{k=1}^K (w_kP_{k,\rho}\vect{d}_k)^\top \left[\text{diag}(\vect{v}) - \vect{v}\vect{v}^\top\right] \sum_{r=1}^K (w_rP_{r,\rho}\vect{d}_r) \\
& := \frac{1}{\tau}\left(\sum_{i=1}^n v_i y_i^2 - \left(\sum_{i=1}^n v_i y_i\right)^2\right) \geq 0,
\end{aligned}
\end{equation}
where we define $\vect{y} = \sum_{k=1}^K w_kP_{k,\rho}\vect{d}_k$ in the second equality. The final inequality holds due to the Cauchy-Schwarz inequality, noting that $\vect{v} \in \Delta_n$. Since all entries of $\vect{v}$ are strictly positive, equality holds if and only if $\vect{y}$ is a constant vector; i.e., there exists $c_0 \in \mathbb{R}$ such that $y_i = c_0$ for all $i \in [n]$. In particular, if $\vect{d}_k=c_k \1_m$ for all $k$, then:
\begin{align*}
y_i = \sum_{k=1}^K w_k \sum_{j=1}^m [P_{k,\rho}]_{ij} c_k = \sum_{k=1}^K w_k c_k \left( \sum_{j=1}^m [P_{k,\rho}]_{ij}\right) = \sum_{k=1}^K w_k c_k, \quad \forall i \in [n],
\end{align*}
which implies there exists a constant $c_0 := \sum_{k=1}^K w_k c_k$ such that $y_i = c_0$ for all $i$, yielding $Q_{\tau}=0$. Since $Q_\eta=0$ holds if and only if every block $\vect{d}_k$ is constant, and such vectors simultaneously satisfy $Q_\tau=0$, the kernel of $H_\rho$ is precisely given by~\eqref{equ: kernel space}. 

It remains to show that the gradient is orthogonal to this kernel space. Take any $\vect{d} \in \ker(H_\rho)$. Then $\vect{d}_k=c_k \1_m$ for some constants $c_k \in \mathbb{R}$. Hence,
\begin{align*}
\langle \vect{g}_k, \vect{d}_k \rangle = c_k \langle \vect{g}_k, \1_m \rangle = c_k w_k \langle (P_{k,\rho})^\top \vect{v} - \vect{\mu}_k, \1_m \rangle = 0, \quad \forall k \in [K],
\end{align*}
where the final equality holds because $\vect{v}^\top P_{k,\rho} \1_m = \vect{v}^\top \1_n = 1$ and $\vect{\mu}_k^\top \1_m = 1$. Summing over $k=1,\dots,K$ yields $\langle \vect{g},\vect{d}\rangle = 0$. This completes the proof.
\end{proof}

\begin{theorem}
\label{thm: the minimum and maximum eigenvalues of the hessian}
For any $\rho \ge 0$, the minimum strictly positive eigenvalue and the maximum eigenvalue of the sparsified Hessian matrix $H_\rho$ satisfy the following bounds: 
\begin{equation}
\lambda_{\min}^{+}(H_\rho) \ge \frac{m}{\eta n} \min_{k} (w_k p_k^2), \quad \lambda_{\max}(H_\rho) \le (\max_{k} w_k) \left( \frac{1}{2\eta} + \frac{1}{\tau} \right),
\end{equation}
where $p_k = \min_{j} [P_{k,\rho}]_{i^\star j}$ with $i^\star \in \arg\max_i [\vect{v}]_i$.
\end{theorem}
\begin{proof}
For any arbitrary vector $\vect{d} = [\vect{d}_1^\top, \vect{d}_2^\top, \dots , \vect{d}_K^\top]^\top \in \mathbb{R}^{mK}$, following the derivation in Theorem~\ref{thm: kernel space of hessian}, the quadratic form is $Q = \vect{d}^\top H_\rho \vect{d} = Q_{\eta} + Q_{\tau}$. Since $Q_\tau \ge 0$, we can safely bound $Q$ from below by $Q_\eta$:
\begin{align*}
Q \ge Q_\eta = \sum_{k=1}^K \frac{w_k}{\eta} (\vect{d}_k)^\top N_{kk,\rho} \vect{d}_k.
\end{align*}
Expanding the local variance block yields:
\begin{align*}
(\vect{d}_k)^\top N_{kk,\rho} \vect{d}_k = \frac{1}{2} \sum_{i=1}^n v_i \sum_{j=1}^m \sum_{l=1}^m [P_{k,\rho}]_{ij} [P_{k,\rho}]_{il} ([\vect{d}_k]_{j} - [\vect{d}_k]_{l})^2.
\end{align*}
For $i^\star \in \arg\max_i [\vect{v}]_i$, it is guaranteed that $v_{i^\star} \ge 1/n$. By defining $p_k = \min_j [P_{k,\rho}]_{i^\star j}$ and isolating the $i^\star$ term, we obtain:
\begin{align*}
(\vect{d}_k)^\top N_{kk,\rho} \vect{d}_k \ge \frac{1}{2n} \sum_{j=1}^m \sum_{l=1}^m [P_{k,\rho}]_{i^\star j} [P_{k,\rho}]_{i^\star l}  ([\vect{d}_k]_{j} - [\vect{d}_k]_{l})^2 \ge \frac{p_k^2}{2n} \sum_{j=1}^m \sum_{l=1}^m ([\vect{d}_k]_{j} - [\vect{d}_k]_{l})^2.
\end{align*}

For $\vect{d} \in \ker(H_\rho)^\perp$, the vector must be orthogonal to each block-wise constant basis vector, implying $\1_m^\top \vect{d}_k = 0$ for all $k \in [K]$. This zero-mean property yields the geometric identity $\sum_{j,l} ([\vect{d}_k]_{j} - [\vect{d}_k]_{l})^2 = 2m \|\vect{d}_k\|^2$ (see, e.g., \cite[Lemma B.1]{pan2026inexact}). Substituting this relation gives:
\begin{align*}
(\vect{d}_k)^\top N_{kk,\rho} \vect{d}_k \ge \frac{p_k^2}{2n} (2m \|\vect{d}_k\|^2) = \frac{m p_k^2}{n} \|\vect{d}_k\|^2.
\end{align*}

Reinserting this into the bound for $Q_\eta$ provides the global lower bound:
\begin{align*}
Q \ge \sum_{k=1}^K \frac{w_k}{\eta} \frac{m p_k^2}{n} \|\vect{d}_k\|^2 \ge \frac{m}{\eta n} \min_k (w_k p_k^2) \|\vect{d}\|^2.
\end{align*}

To establish the upper bound, we evaluate $Q_\eta$ and $Q_\tau$ individually. For $Q_\eta$, we have:
\begin{align*}
(\vect{d}_k)^\top N_{kk,\rho} \vect{d}_k & = \frac{1}{2} \sum_{i=1}^n v_i \sum_{j=1}^m \sum_{l \neq j} [P_{k,\rho}]_{ij} [P_{k,\rho}]_{il} ([\vect{d}_k]_{j} - [\vect{d}_k]_{l})^2 \\
& \le \sum_{i=1}^n v_i \sum_{j=1}^m \sum_{l \neq j} [P_{k,\rho}]_{ij} [P_{k,\rho}]_{il} ([\vect{d}_k]_{j}^2 + [\vect{d}_k]_{l}^2) \\
& = 2 \sum_{i=1}^n v_i \sum_{j=1}^m [\vect{d}_k]_{j}^2 [P_{k,\rho}]_{ij} \left( \sum_{l \neq j} [P_{k,\rho}]_{il}\right)\\
& = 2 \sum_{i=1}^n v_i \sum_{j=1}^m [\vect{d}_k]_{j}^2 [P_{k,\rho}]_{ij} (1 - [P_{k,\rho}]_{ij}) \\
& \leq 2 \sum_{i=1}^n v_i \sum_{j=1}^m [\vect{d}_k]_{j}^2 \left(\frac{1}{4}\right) = \frac{1}{2} \|\vect{d}_k\|^2,
\end{align*}
where the first inequality utilizes $(a-b)^2 \le 2(a^2 + b^2)$ and the final inequality relies on the fact that $x(1-x) \le 1/4$ for $x \in [0,1]$. Thus, $Q_\eta \le \sum_{k=1}^K \frac{w_k}{2\eta} \|\vect{d}_k\|^2$. 

For $Q_\tau$, revisiting Equation~\eqref{equ: Q_tau} and omitting the non-positive squared term yields:
\begin{align*}
Q_\tau & \le \frac{1}{\tau}\sum_{k=1}^K (w_kP_{k,\rho}\vect{d}_k)^\top \text{diag}(\vect{v}) \sum_{r=1}^K (w_rP_{r,\rho}\vect{d}_r) \\
& = \frac{1}{\tau}\sum_{i=1}^n v_i \left(\sum_{k=1}^K w_k \sum_{j=1}^m [P_{k,\rho}]_{ij} [\vect{d}_k]_{j}\right)^2 \\
& \leq \frac{1}{\tau}\sum_{i=1}^n v_i \sum_{k=1}^K w_k \sum_{j=1}^m [P_{k,\rho}]_{ij} [\vect{d}_k]_{j}^2 \\
& \leq \sum_{k=1}^K \frac{w_k}{\tau} \|\vect{d}_k\|^2.
\end{align*}
The second inequality follows from applying Jensen's inequality to the convex combination (since $\sum_{k=1}^K w_k \sum_{j=1}^m [P_{k,\rho}]_{ij} = 1$). The final inequality holds because $\sum_{i=1}^n v_i [P_{k,\rho}]_{ij} = [\vect{\gamma}_{k,\rho}]_j \le 1$ for all $j, k$. Summing these individual bounds yields the final upper bound on the spectrum:
\begin{align*}
Q \le \sum_{k=1}^K w_k \left( \frac{1}{2\eta} + \frac{1}{\tau} \right) \|\vect{d}_k\|^2 \le (\max_k w_k) \left( \frac{1}{2\eta} + \frac{1}{\tau} \right) \|\vect{d}\|^2.
\end{align*}
This successfully proves the upper bound.
\end{proof}

\subsection{Proof of Theorem~\ref{thm: hessian diff}}

To establish the overarching error bound for the sparsified Hessian, we first introduce two supporting lemmas that characterize the spectral norms of the constituent matrices.

\begin{lemma}
\label{lem: row stochastic norm}
For any row-stochastic matrix $P\in\mathbb{R}^{n\times m}_+$ (i.e., satisfying $P\1_m=\1_n$), its spectral norm satisfies $\|P\|_2 \leq \sqrt{n}$.
\end{lemma}

\begin{proof}
By the definition of matrix norms, we have: 
\begin{align*}
\|P\|_{\infty} = \max_i \sum_{j=1}^m |P_{ij}| = 1, \quad \text{and} \quad \|P\|_{1} = \max_j \sum_{i=1}^n |P_{ij}| \leq n.
\end{align*}
Applying the standard norm inequality $\|P\|_2 \leq \sqrt{\|P\|_{\infty}\|P\|_{1}}$, we immediately obtain $\|P\|_2 \leq \sqrt{n}$.
\end{proof}

\begin{lemma}
\label{lem: spectral norm upper bound}
Let $A$ be a symmetric block matrix composed of $K \times K$ blocks, where the $(k, r)$-th block is denoted as $A_{kr}$. Then its spectral norm is bounded by:
\begin{align*}
\|A\|_2 \le \max_k \sum_{r=1}^K \|A_{kr}\|_2.
\end{align*}
\end{lemma}

\begin{proof}
Consider a partitioned vector $\vect{x} = [\vect{x}_1^\top, \dots, \vect{x}_K^\top]^\top$ such that $\|\vect{x}\|_2 = 1$. The $k$-th block of the product $A\vect{x}$ satisfies $\|(A\vect{x})_k\|_2 \le \sum_{r=1}^K \|A_{kr}\|_2 \|\vect{x}_r\|_2$. We construct a $K \times K$ non-negative matrix $M$ with entries $M_{kr} = \|A_{kr}\|_2$, and a vector $\vect{y} \in \mathbb{R}^K$ with entries $y_r = \|\vect{x}_r\|_2$. By definition, $\|\vect{y}\|_2 = 1$. It follows that $\|A\vect{x}\|_2 \le \|M\vect{y}\|_2$. Consequently, we have:
\begin{align*}
\|A\|_2 = \sup_{\|\vect{x}\|_2=1}\|A\vect{x}\|_2 \leq \sup_{\|\vect{x}\|_2=1}\|M\vect{y}\|_2 = \sup_{\vect{y} \in \mathcal{S}_+} \|M\vect{y}\|_2 = \|M\|_2,
\end{align*}
where $\mathcal{S}_+ = \{ \vect{y} \in \mathbb{R}^K \mid \vect{y} \ge 0, \ \|\vect{y}\|_2 = 1 \}$, and the final equality holds by the Perron-Frobenius theorem for non-negative matrices. Furthermore, because $A$ is block symmetric, $M$ is symmetric. Therefore, $\|M\|_1 = \|M\|_{\infty}$, yielding:
\begin{align*}
\|M\|_2 \leq \sqrt{\|M\|_1 \|M\|_{\infty}} \leq \|M\|_{\infty} = \max_{k} \sum_{r=1}^K\|A_{kr}\|_2.
\end{align*}
Combining these inequalities completes the proof. 
\end{proof}

\begin{proof}[Proof of Theorem~\ref{thm: hessian diff}]
We systematically decompose the block matrices and bound the approximation error for each term. Note that we use $\|\cdot\|$ to denote the spectral norm $\|\cdot\|_2$ unless otherwise specified.

\textbf{Step 1: Bound the truncation error of $P_k$.} 
Let $E_{k} = P_k - P_{k,\rho}$ denote the sparsification error matrix. Based on the bounds derived in~\cite[Lemma B.2]{pan2026inexact}, we have:
\begin{align*}
\|E_{k}\|_{\infty} \leq 2m\rho, \quad \|E_{k}\|_1 \leq 2mn\rho, \quad \text{and}\quad \|E_{k}\|\leq 2m\sqrt{n}\rho \quad \forall k\in [K].
\end{align*}

\textbf{Step 2: Bound the block-wise errors $N_{kk}$ and $M_{kr}$.}
\begin{itemize}
    \item \textbf{Error of the local variance block $N_{kk}$:}
    Recall that $N_{kk, \rho} = \text{diag}(\vect{\gamma}_{k,\rho}) - (P_{k,\rho})^\top \text{diag}(\vect{v}) P_{k,\rho}$. By the triangle inequality, the approximation error satisfies:
    \begin{align*}
        \|N_{kk} - N_{kk, \rho}\| &= \|(\text{diag}(\vect{\gamma}_k) - (P_k)^\top \text{diag}(\vect{v}) P_k)-(\text{diag}(\vect{\gamma}_{k,\rho}) - (P_{k,\rho})^\top \text{diag}(\vect{v}) P_{k,\rho})\| \\
        &\leq \|\text{diag}(\vect{\gamma}_k)-\text{diag}(\vect{\gamma}_{k,\rho})\| + \|(P_{k, \rho})^\top \text{diag}(\vect{v}) P_{k, \rho} - (P_k)^\top \text{diag}(\vect{v}) P_k \|.
    \end{align*}
    We bound each term individually. For the diagonal matrix, we utilize the property that its spectral norm equals its infinity norm:
    \begin{align*}
        \|\text{diag}(\vect{\gamma}_k)-\text{diag}(\vect{\gamma}_{k,\rho})\| &= \|\vect{\gamma}_k-\vect{\gamma}_{k,\rho}\|_\infty \\
        &= \|(P_k)^\top \vect{v}-(P_{k,\rho})^\top \vect{v}\|_\infty \\
        &\leq \|P_k^\top-P_{k,\rho}^\top\|_\infty\|\vect{v}\|_\infty \\
        &= \|E_k\|_1 \|\vect{v}\|_\infty \\
        &\leq 2mn\rho.
    \end{align*}
    For the cross term, standard algebraic manipulation yields:
    \begin{align*}
        \|(P_{k, \rho})^\top \text{diag}(\vect{v}) P_{k, \rho} - (P_k)^\top \text{diag}(\vect{v}) P_k \| 
        & = \| - (P_{k, \rho})^\top \text{diag}(\vect{v}) E_{k} - (E_{k})^\top \text{diag}(\vect{v}) P_k \| \\
        & \leq \|(P_{k, \rho})^\top \text{diag}(\vect{v}) E_{k}\| + \|(E_{k})^\top \text{diag}(\vect{v}) P_k\| \\
        & \leq \|P_{k, \rho}\| \|\vect{v}\|_{\infty} \|E_{k}\| + \|E_{k}\| \|\vect{v}\|_{\infty} \|P_k\| \\
        & \leq \sqrt{n} \cdot 1 \cdot (2m\sqrt{n}\rho) + (2m\sqrt{n}\rho) \cdot 1 \cdot \sqrt{n} \\
        & = 4mn\rho,
    \end{align*}
    where we used the inequality $\|\vect{v}\|_{\infty} \leq 1$. Combining these bounds, the error for the $N_{kk}$ block is:
    \begin{equation*}
        \|N_{kk} - N_{kk, \rho}\| \leq 6mn\rho.
    \end{equation*}

    \item \textbf{Error of the cross-covariance block $M_{kr}$:} 
    Recall that $M_{kr, \rho} = (P_{k,\rho})^\top \text{diag}(\vect{v}) P_{r, \rho} - \vect{\gamma}_{k,\rho} (\vect{\gamma}_{r,\rho})^\top$. Applying the triangle inequality gives:
    \begin{align*}
        \|M_{kr} - M_{kr,\rho}\| & = \|((P_k)^\top \text{diag}(\vect{v}) P_r - \vect{\gamma}_k (\vect{\gamma}_r)^\top)-((P_{k,\rho})^\top \text{diag}(\vect{v}) P_{r, \rho} - \vect{\gamma}_{k,\rho} (\vect{\gamma}_{r,\rho})^\top)\| \\
        &\leq \|\vect{\gamma}_k (\vect{\gamma}_r)^\top-\vect{\gamma}_{k,\rho} (\vect{\gamma}_{r,\rho})^\top\| + \|P_k^\top \text{diag}(\vect{v}) P_r - P_{k, \rho}^\top \text{diag}(\vect{v}) P_{r, \rho} \|.
    \end{align*}
    We bound the rank-one term as follows:
    \begin{align*}
        \|\vect{\gamma}_k (\vect{\gamma}_r)^\top-\vect{\gamma}_{k,\rho} (\vect{\gamma}_{r,\rho})^\top\|
        &= \|(P_k)^\top \vect{v}\vect{v}^\top P_r-(P_{k,\rho})^\top \vect{v}\vect{v}^\top P_{r,\rho}\| \\
        &= \|(P_k)^\top \vect{v}\vect{v}^\top E_r + (E_k)^\top \vect{v}\vect{v}^\top P_{r,\rho}\| \\
        &\leq \|P_k\|\|\vect{v}\vect{v}^\top\|\|E_r\|+\|E_k\|\|\vect{v}\vect{v}^\top\|\|P_{r,\rho}\| \\
        &\leq 4mn\rho,
    \end{align*}
    where we applied Lemma~\ref{lem: row stochastic norm} and noted that $\|\vect{v}\vect{v}^\top\|_2 \leq \|\vect{v}\|_2^2 \leq 1$ (since $\vect{v} \in \Delta_n$).
    
    Similarly, the remaining matrix term is bounded by:
    \begin{align*}
        \|P_k^\top \text{diag}(\vect{v}) P_r - P_{k, \rho}^\top \text{diag}(\vect{v}) P_{r, \rho} \|
        & = \|P_k^\top \text{diag}(\vect{v}) E_r + E_k^\top \text{diag}(\vect{v}) P_{r, \rho}\| \\
        & \leq \|P_k\|\|\vect{v}\|_{\infty} \|E_r\| + \|E_{k}\|\|\vect{v}\|_{\infty} \|P_{r, \rho}\| \\
        & \leq 4mn\rho.
    \end{align*}
    Combining these inequalities, the bound for the cross-covariance block is:
    \begin{equation*}
        \|M_{kr} - M_{kr,\rho}\|\leq 8mn\rho.
    \end{equation*}
\end{itemize}

\textbf{Step 3: Establish the global Hessian error bound.} 

Based on Lemma~\ref{lem: spectral norm upper bound}, we synthesize the block-wise errors. Relaxing the coefficient $6mn\rho \leq 8mn\rho$ to facilitate a cleaner algebraic factorization, we obtain:
\begin{align*}
\|H - H_{\rho}\|_2 & \le \max_{k} \sum_{r=1}^K \|(H - H_{\rho})_{kr}\|_2 \\
& \leq \max_k \sum_{r=1}^K \left(\frac{w_kw_r}{\tau}\|M_{kr} - M_{kr, \rho}\| + \frac{w_k}{\eta}\|N_{kk} - N_{kk,\rho}\| \id{k=r} \right) \\
& \leq \max_k \sum_{r=1}^K \left(\frac{w_kw_r}{\tau}8mn\rho + \frac{w_k}{\eta}6mn\rho \id{k=r} \right) \\
& \leq 8mn\rho \max_k \sum_{r=1}^K \left(\frac{w_kw_r}{\tau} + \frac{w_k}{\eta}\id{k=r} \right) \\
& = 8mn \rho (\max_k w_k) \left( \frac{1}{\tau} + \frac{1}{\eta} \right),
\end{align*}
where the final equality follows from the fact that $\sum_{r=1}^K w_r = 1$. This completes the proof.
\end{proof}

\subsection{Proof of Lemma~\ref{lem: sequence restrict in the subspace}}

\begin{proof}
Let $\{\vect{u}_k\}_{k=1}^K$ denote an orthogonal basis for $\ker(H_\rho)$, where each basis vector $\vect{u}_k \in \R^{mK}$ is constructed such that its $k$-th block (of size $m$) is the all-ones vector $\1_m$, and all remaining $K-1$ blocks are strictly zero vectors $\0_m$. Based on the characterization of the kernel space in~\eqref{equ: kernel space}, establishing that $\vect{\beta}^t \in \ker(H_\rho)^\perp$ is mathematically equivalent to proving that $\vect{u}_k^\top \vect{\beta}^t = 0$ for all $k \in [K]$.

We proceed by mathematical induction. The base case holds trivially at $t=0$ by our initialization assumption. Assume that at a given iteration $t$, the iterate satisfies $\vect{u}_k^\top \vect{\beta}^t = 0$ for all $k \in [K]$. By Theorem~\ref{thm: kernel space of hessian}, the gradient lies in the orthogonal complement of the kernel, meaning $\vect{g}^t \in \ker(H_\rho)^\perp$. Thus, $\vect{u}_k^\top \vect{g}^t = 0$ for all $k \in [K]$. 

Recall that the regularized Newton direction $\Delta \vect{\beta}^t$ is obtained by solving the linear system:
\begin{align*}
\big(H^{t}_\rho+\|\vect{g}^{t}\|I\big)\Delta \vect{\beta}^{t} = -\vect{g}^{t}.
\end{align*}
Multiplying both sides from the left by $\vect{u}_k^\top$ yields:
\begin{align*}
\vect{u}_k^\top H^t_\rho \Delta\vect{\beta}^t + \|\vect{g}^t\| (\vect{u}_k^\top \Delta\vect{\beta}^t) = -\vect{u}_k^\top \vect{g}^t = 0.
\end{align*}
Because $\vect{u}_k \in \ker(H_\rho^t)$ and $H_\rho^t$ is symmetric, we have $\vect{u}_k^\top H^t_\rho = \0_{mK}^\top$. Consequently, the first term vanishes. Assuming $\|\vect{g}^t\| > 0$ (otherwise the algorithm has converged, yielding $\Delta \vect{\beta}^t = \0$), it strictly follows that $\vect{u}_k^\top \Delta \vect{\beta}^{t} = 0$ for all $k \in [K]$. 

Therefore, for any accepted step size $s^t$, the subsequent iterate satisfies:
\begin{align*}
\vect{u}_k^\top \vect{\beta}^{t+1} = \vect{u}_k^\top (\vect{\beta}^t + s^t \Delta \vect{\beta}^t) = 0.
\end{align*}
Since this equality holds for every $k \in [K]$, we conclude that $\vect{\beta}^{t+1} \in \ker(H_\rho)^\perp$. This completes the induction.
\end{proof}

\subsection{Proof of Theorem~\ref{thm: convergence}}

\begin{lemma}
\label{lem: bounded level set + bounded below}
The sublevel set $S(\vect{\beta}^0) := \{\vect{\beta} \mid \lyp(\vect{\beta}) \leq \lyp(\vect{\beta}^0),\ \vect{\beta} \in \ker(H_{\rho})^\perp\}$ is bounded, closed, and convex. Furthermore, the objective function $\lyp(\vect{\beta})$ is bounded below and admits a unique minimizer $\vect{\beta}^\star\in S(\vect{\beta}^0)$ on the restricted subspace $\ker(H_{\rho})^\perp$.
\end{lemma}

\begin{proof}
First, it is evident that $\lyp(\vect{\beta})$ is continuously differentiable. According to Theorem~\ref{thm: kernel space of hessian}, the sparsified Hessian matrix $H_\rho$ is strictly positive definite on the restricted subspace $\ker(H_{\rho})^\perp$. This property guarantees that $\lyp(\vect{\beta})$ is strictly convex on this subspace. Consequently, the sublevel set $S_c := \{\vect{\beta} \mid \lyp(\vect{\beta}) \leq c,\ \vect{\beta} \in \ker(H_{\rho})^\perp\}$ is closed and convex. 

We now prove that $\lyp(\vect{\beta})$ is coercive when restricted to the subspace $\ker(H_\rho)^\perp$. Specifically, we aim to show that as $\|\vect{\beta}\| \to \infty$ subject to $\vect{\beta} \in \ker(H_\rho)^\perp$, the objective value satisfies $\mathcal{L}(\vect{\beta}) \to \infty$. 

Recall the definition of the objective function:
\begin{equation*}
    \lyp(\vect{\beta}) = - \sum_{k=1}^K w_k \langle \vect{\beta}_k, \vect{\mu}_k \rangle + \tau \log \left(\sum_{i=1}^n \pi_i \exp\left(-[\Phi(\vect{\beta})]_i /\tau\right)\right),
\end{equation*}
where $\Phi(\vect{\beta}) = \sum_{k=1}^K w_k \vect{\alpha}_k (\vect{\beta}_k)$. 

\textbf{Step 1: Lower bound via Jensen's inequality.}
Since $f(x) = \log(x)$ is a concave function, we can apply Jensen's inequality to the log-sum-exp term. Noting that $\vect{\pi} \in \Delta_n$ defines a valid probability distribution, we obtain:
\begin{equation*}
\begin{aligned}
    \tau \log \left( \sum_{i=1}^n \pi_i \exp(-[\Phi(\vect{\beta})]_i/\tau) \right) & \ge \tau \sum_{i=1}^n \pi_i \left( -\frac{[\Phi(\vect{\beta})]_i}{\tau} \right) = -\sum_{i=1}^n \pi_i [\Phi(\vect{\beta})]_i\\
    & = -\sum_{k=1}^K w_k \sum_{i=1}^n \pi_i [\vect{\alpha}_k (\vect{\beta}_k)]_i.
\end{aligned}
\end{equation*}
Substituting this relation back into $\mathcal{L}(\vect{\beta})$, we establish the lower bound:
\begin{equation*}
    \mathcal{L}(\vect{\beta}) \ge \sum_{k=1}^K w_k \left[ -\langle \vect{\beta}_k, \vect{\mu}_k \rangle - \sum_{i=1}^n \pi_i [\vect{\alpha}_k (\vect{\beta}_k)]_i \right].
\end{equation*}

\textbf{Step 2: Bound the dual potentials.}
Note that $[\vect{\alpha}_k(\vect{\beta_k})]_{i} = -\eta \log\left(\sum_{j=1}^m [\vect{\mu}_k]_{j}\exp\left(\frac{[\vect{\beta}_{k}]_{j}-[C_k]_{ij}}{\eta}\right)\right)$. 
The summation inside the logarithm is strictly bounded below by its maximum term. Let $\beta_{k,\max} = \max_j [\vect{\beta}_k]_j$ and let $j^\star$ denote the index where this maximum is achieved. We then have:
\begin{align*}
    -[\vect{\alpha}_k(\vect{\beta_k})]_{i} &= \eta \log\left(\sum_{j=1}^m [\vect{\mu}_k]_{j}\exp\left(\frac{[\vect{\beta}_{k}]_{j}-[C_k]_{ij}}{\eta}\right)\right) \nonumber \\
    &\ge \eta \log \left( [\vect{\mu}_{k}]_{j^\star} \exp\left( \frac{\beta_{k,\max} - [C_k]_{ij^\star}}{\eta} \right) \right) \nonumber \\
    &= \beta_{k,\max} - [C_k]_{ij^\star} + \eta \log [\vect{\mu}_{k}]_{j^\star} \nonumber \\
    &\ge \beta_{k,\max} - C_{\max} + \eta \log \mu_{\min},
\end{align*}
where $C_{\max} = \max_{k,i,j} [C_k]_{ij}$ and $\mu_{\min} = \min_{k,j} [\vect{\mu}_k]_{j} > 0$.

Substituting this bound back, and utilizing the property that $\sum_{i=1}^n \pi_i = 1$, we obtain:
\begin{equation*}
    \lyp(\vect{\beta}) \ge \sum_{k=1}^K w_k \left( \beta_{k,\max} - \langle \vect{\beta}_k, \vect{\mu}_k \rangle \right) - C_{\max} + \eta \log \mu_{\min}.
\end{equation*}

\textbf{Step 3: Exploit the subspace restriction.}
We now analyze the term $\beta_{k,\max} - \langle \vect{\beta}_k, \vect{\mu}_k \rangle$. Given that $\vect{\mu}_k \in \Delta_m$ (its components sum to 1), we can rewrite this difference as:
\begin{equation*}
    \beta_{k,\max} - \langle \vect{\beta}_k, \vect{\mu}_k \rangle = \sum_{j=1}^m [\vect{\mu}_{k}]_{j} (\beta_{k,\max} - [\vect{\beta}_{k}]_{j}).
\end{equation*}
Because $\beta_{k,\max} - [\vect{\beta}_{k}]_{j} \ge 0$ for all $j$, we can safely bound this summation from below using the minimum marginal weight $\mu_{\min}$:
\begin{equation*}
    \sum_{j=1}^m [\vect{\mu}_{k}]_{j} (\beta_{k,\max} - [\vect{\beta}_{k}]_{j}) \ge \mu_{\min} \sum_{j=1}^m (\beta_{k,\max} - [\vect{\beta}_{k}]_{j}) = \mu_{\min} \left( m \beta_{k,\max} - \sum_{j=1}^m [\vect{\beta}_{k}]_{j} \right).
\end{equation*}

Crucially, because $\vect{\beta}$ is restricted to the orthogonal complement of the kernel, $\vect{\beta} \in \ker(H_\rho)^\perp$, the block-wise sum must vanish: $\sum_{j=1}^m [\vect{\beta}_{k}]_{j} = 0$ for all $k \in [K]$. Thus, the inequality simplifies to:
\begin{equation*}
    \beta_{k,\max} - \langle \vect{\beta}_k, \vect{\mu}_k \rangle \ge m \mu_{\min} \beta_{k,\max}.
\end{equation*}

\textbf{Step 4: Relate $\beta_{k,\max}$ to the vector norm.}
Because the components of $\vect{\beta}_k$ sum to zero, the magnitude of the most negative component is bounded by the positive components. Specifically, letting $\beta_{k,\min} := \min_j [\vect{\beta}_{k}]_{j}$, we have $-\beta_{k,\min} \le (m-1)\beta_{k,\max}$. 
Therefore, for any index $j$, $|[\vect{\beta}_{k}]_{j}| \le (m-1)\beta_{k,\max}$. 
This relationship allows us to bound the Euclidean norm of the block:
\begin{equation*}
    \|\vect{\beta}_k\|^2 \le m \left( (m-1)\beta_{k,\max} \right)^2 \implies \beta_{k,\max} \ge \frac{1}{\sqrt{m}(m-1)} \|\vect{\beta}_k\|.
\end{equation*}

\textbf{Step 5: Final coercivity bound.}
Combining the derived bounds, we arrive at the following inequality:
\begin{equation*}
    \mathcal{L}(\vect{\beta}) \ge \left( \frac{m \mu_{\min}}{\sqrt{m}(m-1)} \min_k w_k \right) \sum_{k=1}^K \|\vect{\beta}_k\| - C_{\text{const}},
\end{equation*}
where $C_{\text{const}} = C_{\max} - \eta \log \mu_{\min}$.

Let $c = \frac{\sqrt{m} \mu_{\min}}{m-1} \min_k w_k > 0$. Since the sum of the block norms upper bounds the total Euclidean norm, $\sum_{k=1}^K \|\vect{\beta}_k\| \ge \|\vect{\beta}\|$, we deduce:
\begin{equation*}
    \mathcal{L}(\vect{\beta}) \ge c \|\vect{\beta}\| - C_{\text{const}}.
\end{equation*}

Because $c > 0$, it strictly follows that $\mathcal{L}(\vect{\beta}) \to \infty$ as $\|\vect{\beta}\| \to \infty$ within the subspace $\ker(H_\rho)^\perp$, confirming that $\mathcal{L}(\vect{\beta})$ is coercive on $\ker(H_\rho)^\perp$.

As a continuous and coercive function defined on a closed subspace, its sublevel set $S(\vect{\beta}^0)$ is guaranteed to be bounded and closed. Furthermore, because $\mathcal{L}(\vect{\beta})$ is strictly convex on $\ker(H_\rho)^\perp$, it admits a unique global minimizer $\vect{\beta}^\star \in S(\vect{\beta}^0)$. Consequently, $\mathcal{L}(\vect{\beta})$ possesses a finite lower bound on $S(\vect{\beta}^0)$, which completes the proof. 
\end{proof}

\begin{lemma}
\label{lem: uniform lower bound hessian}
For any $\rho \geq 0$, there exists a uniform constant $\underline{C} > 0$ such that $\lambda_{\min}^+\left(H_{\rho}\right) \ge \underline{C}$ for all $\vect{\beta} \in S(\vect{\beta}^0)$.
\end{lemma}

\begin{proof}
By Lemma~\ref{lem: bounded level set + bounded below}, the restricted sublevel set $S(\vect{\beta}^0)$ is bounded and closed, so it's a compact set. Recall from Theorem~\ref{thm: the minimum and maximum eigenvalues of the hessian} that the smallest strictly positive eigenvalue of $H_\rho$ is bounded below by:
\begin{equation*}
    \lambda_{\min}^{+}(H_\rho) \ge \frac{m}{\eta n} \min_{k} (w_k p_k^2),
\end{equation*}
where $p_k = \min_{j} [P_{k,\rho}]_{i^\star j}$ for some index $i^\star \in \arg\max_i [\vect{v}]_i$. 

According to the sparsification rule, the elements of the $i^\star$-th row of $P_{k,\rho}$ are preserved exactly from the unsparsified transport probability matrix $P_k$; that is, $[P_{k,\rho}]_{i^\star j} = [P_k]_{i^\star j}$ for all $j \in [m]$. By definition, the entries of $P_k$ are compositions of exponential functions, ensuring they are strictly positive and continuous with respect to $\vect{\beta}$. Consequently, $p_k$ is bounded from below by the global minimum of the matrix entries:
\begin{equation*}
    p_k \ge \min_{i,j} [P_k(\vect{\beta})]_{ij} > 0.
\end{equation*}

Because $S(\vect{\beta}^0)$ is compact and the mapping $\vect{\beta} \mapsto \min_{i,j} [P_k(\vect{\beta})]_{ij}$ is strictly positive and continuous, this mapping attains a strictly positive minimum over $S(\vect{\beta}^0)$. Therefore, the term $\min_k (w_k p_k^2)$ is uniformly bounded away from zero across the entire sublevel set. We conclude that there exists a uniform constant $\underline{C} > 0$ such that $\lambda_{\min}^+(H_{\rho}) \ge \underline{C}$ for all $\vect{\beta} \in S(\vect{\beta}^0)$.
\end{proof}

\begin{lemma}
\label{lem: backtracking}
Let $\{\vect{\beta}^t\}_{t\in [T]}$ be the sequence generated by Algorithm~\ref{alg: NWB} or~\ref{alg: SNWB} satisfying $\vect{\beta}^0 \in \ker(H_\rho)^\perp$. Let $\Delta \vect{\beta}^t$ be the update direction for $\lyp$ and suppose the step size $s^t$ is chosen via Armijo backtracking. Then the following holds:
\begin{itemize}
    \item[(i)] $\Delta\vect{\beta}^t$ is a strictly descent direction for $\mathcal{L}(\vect{\beta})$.
    \item[(ii)] The backtracking terminates finitely with a step size $s^t$ bounded strictly away from zero by:
    \begin{equation}
        s^t \ge s_{\min} := \min\left\{1, c \frac{2(1-\sigma)\underline{C}^2}{L(L + G_{\max})}\right\} > 0,
    \end{equation}
    where $\underline{C}$ is the uniform lower bound established in Lemma~\ref{lem: uniform lower bound hessian}, $L \coloneqq (\max_k w_k)(\frac{1}{2\eta} + \frac{1}{\tau})$, $c \in (0,1)$ is the contraction factor, $\sigma \in (0,1)$ is the sufficient decrease parameter, and $G_{\max} \coloneqq 2\sqrt{\sum_{k=1}^K w_k^2}$.
\end{itemize}
\end{lemma}

\begin{proof}
(i) To verify that the update provides a descent direction, we evaluate its inner product with the gradient $\vect{g}^t$:
$$
(\vect{g}^t)^\top \Delta\vect{\beta}^t = -(\vect{g}^t)^\top (H_\rho^t + \|\vect{g}^t\|I)^{-1} \vect{g}^t < 0.
$$
This strict inequality holds because Theorem~\ref{thm: kernel space of hessian} guarantees that $H_\rho^t$ is positive semidefinite. Consequently, the shifted matrix $(H_\rho^t + \|\vect{g}^t\|I)$ is strictly positive definite whenever $\|\vect{g}^t\| \ne 0$. Thus, $\Delta\vect{\beta}^t$ is a strictly descent direction.

(ii) By Theorem~\ref{thm: the minimum and maximum eigenvalues of the hessian}, the maximum eigenvalue of the Hessian is globally bounded by $\lambda_{\max}(H_\rho^t) \le (\max_k w_k)(\frac{1}{2\eta} + \frac{1}{\tau}) = L$. This implies that the gradient $\nabla \lyp (\vect{\beta})$ is uniformly $L$-Lipschitz continuous. Applying the standard quadratic upper bound for functions with Lipschitz continuous gradients, we have:
\begin{align*}
\mathcal{L}(\vect{\beta}^t + s\Delta\vect{\beta}^t) \le \mathcal{L}(\vect{\beta}^t) + s(\vect{g}^t)^\top \Delta\vect{\beta}^t + \frac{L}{2}s^2\|\Delta\vect{\beta}^t\|^2.
\end{align*}
For the Armijo sufficient decrease condition $\mathcal{L}(\vect{\beta}^t + s\Delta\vect{\beta}^t) \le \mathcal{L}(\vect{\beta}^t) + \sigma s(\vect{g}^t)^\top \Delta\vect{\beta}^t$ to hold, it is sufficient that:
\begin{align*}
s(\vect{g}^t)^\top \Delta\vect{\beta}^t + \frac{L}{2}s^2\|\Delta\vect{\beta}^t\|^2 \le \sigma s(\vect{g}^t)^\top \Delta\vect{\beta}^t.
\end{align*}
Dividing by $s > 0$ and rearranging the terms reveals that any step size $s \in (0, \bar{s}^t]$ is acceptable, where:
\begin{align*}
\bar{s}^t := \frac{2(1-\sigma)(-(\vect{g}^t)^\top \Delta\vect{\beta}^t)}{L \|\Delta\vect{\beta}^t\|^2}.
\end{align*}

We now bound the numerator and denominator of $\bar{s}^t$ separately. For the numerator, we obtain the lower bound:
\begin{align*}
-(\vect{g}^t)^\top \Delta\vect{\beta}^t = (\vect{g}^t)^\top (H_\rho^t + \|\vect{g}^t\|I)^{-1} \vect{g}^t \ge \frac{1}{\lambda_{\max}(H_\rho^t) + \|\vect{g}^t\|} \|\vect{g}^t\|^2 \ge \frac{1}{L + \|\vect{g}^t\|} \|\vect{g}^t\|^2.
\end{align*}

For the denominator, we must bound the norm of the update direction $\|\Delta\vect{\beta}^t\|$. Because $H_\rho^t$ is symmetric and positive semidefinite, the space $\mathbb{R}^{mK}$ admits an orthogonal decomposition into $\ker(H_\rho^t)$ and its orthogonal complement $\ker(H_\rho^t)^\perp$. By Theorem~\ref{thm: kernel space of hessian}, the gradient satisfies $\vect{g}^t \in \ker(H_\rho^t)^\perp$. Consequently, the action of the inverse operator $(H_\rho^t + \|\vect{g}^t\|I)^{-1}$ on $\vect{g}^t$ is entirely confined to the invariant subspace $\ker(H_\rho^t)^\perp$. 

On this restricted subspace, the eigenvalues of $H_\rho^t$ are strictly positive and bounded below by $\lambda_{\min}^+(H_\rho^t)$. Applying the Rayleigh quotient theorem~\cite{Horn_Johnson_2012} to the shifted inverse operator on this subspace yields the exact spectral norm bound:
\begin{equation*}
    \sup_{\vect{d} \in \ker(H_\rho)^\perp} \frac{\|(H_\rho^t + \|\vect{g}^t\|I)^{-1} \vect{d}\|}{\|\vect{d}\|} = \frac{1}{\lambda_{\min}^+(H_\rho^t) + \|\vect{g}^t\|}.
\end{equation*}

Since the Armijo line search ensures that all iterates $\{\vect{\beta}^t\}_{t\in [T]}$ remain within the sublevel set $S(\vect{\beta}^0)$, Lemma~\ref{lem: uniform lower bound hessian} provides the uniform lower bound $\lambda_{\min}^+(H_\rho^t) \ge \underline{C} > 0$. This yields the following upper bound for the denominator:
\begin{align*}
\|\Delta\vect{\beta}^t\| \le \frac{1}{\lambda_{\min}^+(H_\rho^t) + \|\vect{g}^t\|} \|\vect{g}^t\| \le \frac{1}{\underline{C} + \|\vect{g}^t\|} \|\vect{g}^t\| \le \frac{1}{\underline{C}} \|\vect{g}^t\|.
\end{align*}

Substituting these two bounds back into the expression for $\bar{s}^t$ produces:
\begin{align*}
\bar{s}^t \ge \frac{2(1-\sigma) \frac{\|\vect{g}^t\|^2}{L + \|\vect{g}^t\|}}{L \frac{\|\vect{g}^t\|^2}{\underline{C}^2}} = \frac{2(1-\sigma)\underline{C}^2}{L(L + \|\vect{g}^t\|)}.
\end{align*}

Finally, recall from \eqref{equ: gradient} that $\nabla_k \lyp (\vect{\beta}) = w_k ((P_k)^\top \vect{v} - \vect{\mu}_k)$. Because $(P_k)^\top \vect{v}$ and $\vect{\mu}_k$ are both valid probability vectors (non-negative and summing to $1$), the maximum Euclidean distance between them is strictly bounded by $\|(P_k)^\top \vect{v} - \vect{\mu}_k\| \le 2$. Consequently, the total gradient norm admits the global upper bound $\|\vect{g}^t\| = \sqrt{\sum_{k=1}^K w_k^2 \|(P_k)^\top \vect{v} - \vect{\mu}_k\|^2} \le 2\sqrt{\sum_{k=1}^K w_k^2} = G_{\max}$. 

Replacing $\|\vect{g}^t\|$ with its maximum bound $G_{\max}$, we obtain a uniform lower bound for the theoretical acceptable step size:
\begin{align*}
\bar{s}^t \ge \frac{2(1-\sigma)\underline{C}^2}{L(L + G_{\max})}.
\end{align*}
Because the backtracking procedure iteratively shrinks the step size by a factor of $c \in (0,1)$ starting from $1$, the final accepted step size $s^t$ must be strictly larger than $c \bar{s}^t$ (unless the full Newton step $s^t = 1$ is immediately accepted). We thus conclude:
\begin{align*}
s^t \ge \min\left\{1, c \frac{2(1-\sigma)\underline{C}^2}{L(L + G_{\max})}\right\} > 0.
\end{align*}
\end{proof}

Now we are ready to prove the convergence.

\begin{proof}[Proof of Theorem~\ref{thm: convergence}]
(i) From Lemma~\ref{lem: backtracking}, the Armijo backtracking line search guarantees that the sequence of step sizes $s^t$ is bounded below by a uniform constant $s_{\min} > 0$. Furthermore, it satisfies the sufficient decrease condition:
\begin{equation*}
    \mathcal{L}(\vect{\beta}^{t+1}) - \mathcal{L}(\vect{\beta}^t) \le \sigma s^t (\vect{g}^t)^\top \Delta\vect{\beta}^t.
\end{equation*}

From the proof of Lemma~\ref{lem: backtracking}, we established the following lower bound on the directional derivative:
\begin{equation*}
    -(\vect{g}^t)^\top \Delta\vect{\beta}^t = (\vect{g}^t)^\top (H_\rho^t + \|\vect{g}^t\|I)^{-1} \vect{g}^t \ge \frac{1}{L + \|\vect{g}^t\|} \|\vect{g}^t\|^2 \ge \frac{1}{L + G_{\max}} \|\vect{g}^t\|^2.
\end{equation*}

Substituting this bound into the Armijo condition yields:
\begin{equation*}
    \mathcal{L}(\vect{\beta}^{t+1}) - \mathcal{L}(\vect{\beta}^t) \le - \sigma s_{\min} \frac{1}{L + G_{\max}} \|\vect{g}^t\|^2.
\end{equation*}

This inequality demonstrates that the sequence of objective values $\{\mathcal{L}(\vect{\beta}^t)\}_{t \ge 0}$ is monotonically decreasing. Combined with Lemma~\ref{lem: sequence restrict in the subspace}, which ensures that $\vect{\beta}^t \in \ker(H_\rho)^\perp$ for all $t$, it strictly follows that the entire sequence of iterates $\{\vect{\beta}^t\}_{t \ge 0}$ remains confined within the initial sublevel set $S(\vect{\beta}^0)$. 

By Lemma~\ref{lem: bounded level set + bounded below}, $\mathcal{L}(\vect{\beta})$ is bounded below on $S(\vect{\beta}^0)$. Let $\underline{\mathcal{L}}$ denote this global minimum. Summing the sufficient decrease inequality over the first $T$ iterations yields a telescoping sum:
\begin{equation*}
    \sum_{t=0}^T \|\vect{g}^t\|^2 \le \frac{L + G_{\max}}{\sigma s_{\min}} (\mathcal{L}(\vect{\beta}^0) - \mathcal{L}(\vect{\beta}^{T+1})) \leq \frac{L + G_{\max}}{\sigma s_{\min}} (\mathcal{L}(\vect{\beta}^0) - \underline{\mathcal{L}}) < \infty.
\end{equation*}
Taking the limit as $T \to \infty$, we obtain the convergent infinite series $\sum_{t=0}^\infty \|\vect{g}^t\|^2 < \infty$. Therefore, we conclude that $\|\vect{g}^t\| \rightarrow 0$ as $t \rightarrow \infty$.

(ii) Let $\vect{\beta}^\star$ be the unique optimal solution in the restricted subspace $\ker(H_\rho)^\perp$, the existence of which is guaranteed by Lemma~\ref{lem: bounded level set + bounded below}. The first-order optimality condition requires that the projected gradient vanishes, meaning the full gradient must be orthogonal to the subspace: $\nabla \lyp(\vect{\beta}^\star) \in \ker(H_\rho)$. However, from the exact gradient formula~\eqref{equ: gradient}, the block-wise sum always satisfies $\1_m^\top \nabla_k \lyp(\vect{\beta}^\star) = 0$, strictly placing $\nabla \lyp(\vect{\beta}^\star) \in \ker(H_\rho)^\perp$. Then we conclude $\nabla \lyp(\vect{\beta}^\star) = \0_{mK}$. 

From Lemma~\ref{lem: uniform lower bound hessian}, there exists a uniform lower bound $\underline{C} > 0$ such that the smallest strictly positive eigenvalue satisfies $\lambda_{\min}^+(H) \ge \underline{C}$ for all $\vect{\beta} \in S(\vect{\beta}^0)$ (noting that the exact Hessian corresponds to $\rho=0$). Consequently, the objective function $\mathcal{L}(\vect{\beta})$ is strongly convex with parameter $\underline{C}$ on the restricted subspace within $S(\vect{\beta}^0)$. This strong convexity condition provides a direct bound relating the norm of the gradient to the distance to the optimum:
\begin{equation*}
    \|\vect{g}^t\| = \|\nabla \lyp(\vect{\beta}^t) - \nabla \lyp(\vect{\beta}^\star)\| \ge \underline{C} \|\vect{\beta}^t - \vect{\beta}^\star\|.
\end{equation*}

Since we established in part (i) that $\|\vect{g}^t\| \to 0$ as $t \to \infty$, it immediately follows from this strong convexity bound that $\|\vect{\beta}^t - \vect{\beta}^\star\| \to 0$. Thus, the sequence of iterates globally converges to the unique optimal solution: $\vect{\beta}^t \to \vect{\beta}^\star$. 
\end{proof}

\subsection{Proof of Theorem~\ref{thm: local_quadratic}}

We first show the globally Lipschitz continuity of the exact Hessian $H(\vect{\beta})$. 

\begin{lemma}
\label{lem: global lipschitz hessian}
The exact Hessian $H(\vect{\beta})$ of the objective function $\mathcal{L}(\vect{\beta})$ is globally Lipschitz continuous over $\mathbb{R}^{mK}$. Specifically, with $W_{\max} = \max_k w_k$, the global Lipschitz constant $L_H$ is:
\begin{equation}
    \label{equ: hessian lipschitz constant}
    L_H \coloneqq 6 \frac{W_{\max}}{\eta^2} + 10 \frac{W_{\max}}{\eta \tau} + \frac{6}{\tau^2}.
\end{equation}
\end{lemma}

\begin{proof}
To establish global Lipschitz continuity, we evaluate the spectral norm of the third derivative tensor. Recall that for a scalar-valued multivariate function $f(\vect{\beta})$, its directional derivative along a vector $\vect{u}$ is defined as $\nabla_{\vect{u}} f(\vect{\beta}) = \langle \nabla f(\vect{\beta}), \vect{u} \rangle = \lim_{\epsilon \to 0} \frac{f(\vect{\beta} + \epsilon \vect{u}) - f(\vect{\beta})}{\epsilon}$. By treating the quadratic form $Q(\vect{\beta}) \coloneqq \vect{y}^\top H(\vect{\beta}) \vect{y}$ as a scalar field parameterized by $\vect{\beta}$, establishing Lipschitz continuity of the Hessian is mathematically equivalent to showing that for any arbitrary, fixed vectors $\vect{y}, \vect{u} \in \mathbb{R}^{mK}$, the third-order directional derivative of the objective is bounded by a uniform constant $L_H$:
\begin{equation*}
    |\nabla_{\vect{u}} (\vect{y}^\top H(\vect{\beta}) \vect{y})| \le L_H \|\vect{y}\|^2 \|\vect{u}\|, \quad \forall\  \vect{\beta} \in \mathbb{R}^{mK}.
\end{equation*}

\textbf{Step 1: Notation and deterministic operators.} To streamline the algebra, we define a deterministic weighted sum operator $S_{k,i}: \mathbb{R}^m \to \mathbb{R}$, which represents the inner product with the $i$-th row of the transport probability matrix $P_k$:
\begin{equation*}
    S_{k,i}(\vect{x}) = \sum_{j=1}^m [P_k]_{ij} x_j.
\end{equation*}
Because the rows of $P_k$ are discrete probability distributions (i.e., $[P_k]_{ij} > 0$ and $\sum_{j=1}^m [P_k]_{ij} = 1$), the operator $S_{k,i}$ computes a convex combination. Therefore, its magnitude is universally bounded by the infinity norm, and consequently the Euclidean norm, of the input vector: 
\begin{equation}
\label{eq: bound_S}
    |S_{k,i}(\vect{x})| \le \|\vect{x}\|_\infty \le \|\vect{x}\|.
\end{equation}

We also define an aggregated block operator $R_i(\vect{x}) = \sum_{k=1}^K w_k S_{k,i}(\vect{x}_k)$. By applying the Cauchy-Schwarz inequality and noting that $\sum_{k=1}^K w_k = 1$, its magnitude is globally bounded by:
\begin{equation}
\label{eq: bound_R}
    |R_i(\vect{x})| \le \sum_{k=1}^K w_k \|\vect{x}_k\| \le \left(\sum_{k=1}^K w_k\right)^{1/2} \left(\sum_{k=1}^K w_k \|\vect{x}_k\|^2\right)^{1/2} \le \sqrt{W_{\max}} \|\vect{x}\| \le \|\vect{x}\|.
\end{equation}

Using these operators, the exact quadratic form $Q(\vect{\beta}) := \vect{y}^\top H(\vect{\beta}) \vect{y}$ decomposes into $Q_\eta(\vect{\beta})$ and $Q_\tau(\vect{\beta})$ components based on the block definitions of the Hessian established in Theorem~\ref{thm: kernel space of hessian}:
\begin{align}
    Q_\eta(\vect{\beta}) &= \sum_{k=1}^K \frac{w_k}{\eta} \sum_{i=1}^n v_i \Big[ S_{k,i}(\vect{y}_k^2) - (S_{k,i}(\vect{y}_k))^2 \Big], \label{eq: Q_eta} \\
    Q_\tau(\vect{\beta}) &= \frac{1}{\tau} \left[ \sum_{i=1}^n v_i (R_i(\vect{y}))^2 - \Big( \sum_{i=1}^n v_i R_i(\vect{y}) \Big)^2 \right], \label{eq: Q_tau}
\end{align}
where $\vect{y}_k^2$ denotes the Hadamard (element-wise) square of the vector $\vect{y}_k$. Here, $Q_\eta(\vect{\beta})$ and $Q_\tau(\vect{\beta})$ are explicitly written as functions of $\vect{\beta}$ to emphasize that the subsequent directional derivatives are taken with respect to $\vect{\beta}$, while $\vect{y}$ and $\vect{u}$ remain arbitrary, fixed vectors. Let $\bar{R}(\vect{y}) = \sum_{i=1}^n v_i R_i(\vect{y})$ denote the weighted average of $R_i(\vect{y})$. Because the barycenter weight vector satisfies $\vect{v} \in \Delta_n$, we trivially have $|\bar{R}(\vect{y})| \le \max_i |R_i(\vect{y})| \le \|\vect{y}\|$.

\textbf{Step 2: Derivation of fundamental directional derivatives.} Before differentiating the quadratic forms, we explicitly calculate the directional derivatives of the fundamental softmax weights. 
For the transport probability matrix, its entries are defined as:
\begin{equation*}
    [P_k]_{ij} = \frac{[\vect{\mu}_k]_j \exp(([\vect{\beta}_k]_{j} - [C_k]_{ij})/\eta)}{[\vect{z_k}]_{i}}, \quad \text{where} \quad [\vect{z_k}]_{i} = \sum_{l=1}^m [\vect{\mu}_k]_l \exp\left(\frac{[\vect{\beta}_k]_{l} - [C_k]_{il}}{\eta}\right).
\end{equation*}
Taking the logarithm yields:
\begin{equation*}
    \ln [P_k]_{ij} = \ln [\vect{\mu}_k]_j + \frac{[\vect{\beta}_k]_{j} - [C_k]_{ij}}{\eta} - \ln [\vect{z}_k]_{i}.
\end{equation*}
Differentiating both sides along the direction $\vect{u}$ eliminates the constant target marginal term $\ln [\vect{\mu}_k]_j$:
\begin{align}
    \nabla_{\vect{u}} \ln [\vect{z_k}]_{i} &= \frac{1}{[\vect{z_k}]_{i}} \sum_{l=1}^m [\vect{\mu}_k]_l \exp\left(\frac{[\vect{\beta}_k]_{l} - [C_k]_{il}}{\eta}\right) \frac{[\vect{u_k}]_{l}}{\eta} = \frac{1}{\eta} \sum_{l=1}^m [P_k]_{il} [\vect{u_k}]_{l} = \frac{1}{\eta} S_{k,i}(\vect{u}_k), \nonumber \\
    \nabla_{\vect{u}} [P_k]_{ij} &= [P_k]_{ij} \nabla_{\vect{u}} (\ln [P_k]_{ij}) = \frac{1}{\eta} [P_k]_{ij} \Big( [\vect{u_k}]_{j} - S_{k,i}(\vect{u}_k) \Big). \label{eq: deriv_Pk}
\end{align}

Similarly, the barycenter weights are defined as $v_i = \pi_i \exp(-\Phi_i/\tau) / \tilde{z}$, where $\Phi_i = -\eta \sum_{k=1}^K w_k \ln [\vect{z_k}]_{i}$ and the normalization constant is $\tilde{z} = \sum_{l=1}^n \pi_l \exp(-\Phi_l/\tau)$. 
The directional derivative of the potential is algebraically derived as $\nabla_{\vect{u}} \Phi_i = -\eta \sum_{k=1}^K w_k (\frac{1}{\eta} S_{k,i}(\vect{u}_k)) = -R_i(\vect{u})$. Differentiating the logarithm of $v_i$ yields:
\begin{align}
    \nabla_{\vect{u}} \ln \tilde{z} &= \frac{1}{\tilde{z}} \sum_{l=1}^n \pi_l \exp(-\Phi_l/\tau) \left( -\frac{\nabla_{\vect{u}} \Phi_l}{\tau} \right) = \frac{1}{\tau} \sum_{l=1}^n v_l R_l(\vect{u}) = \frac{1}{\tau} \bar{R}(\vect{u}), \nonumber \\
    \nabla_{\vect{u}} v_i &= v_i \nabla_{\vect{u}} (\ln v_i) = \frac{1}{\tau} v_i \Big( R_i(\vect{u}) - \bar{R}(\vect{u}) \Big). \label{eq: deriv_v}
\end{align}
Using the bounds from \eqref{eq: bound_R}, we establish the absolute magnitude bound: $|\nabla_{\vect{u}} v_i| \le \frac{2}{\tau} v_i \|\vect{u}\|$.

\textbf{Step 3: Bounding the local component $\nabla_{\vect{u}} Q_\eta(\vect{y})$.} Differentiating \eqref{eq: Q_eta} with respect to $\vect{\beta}$ along the direction $\vect{u}$ generates two distinct summation terms: one from differentiating $v_i$ (Part A), and one from differentiating the bracketed operation (Part B).

\textit{Part A (Differentiating $v_i$):}
The bracketed term satisfies $0 \le S_{k,i}(\vect{y}_k^2) - (S_{k,i}(\vect{y}_k))^2 \le S_{k,i}(\vect{y}_k^2) \le \|\vect{y}_k^2\|_\infty \le \|\vect{y}_k\|^2$. Multiplying this by the bound for $|\nabla_{\vect{u}} v_i|$, we obtain:
\begin{equation*}
    \sum_{k=1}^K \frac{w_k}{\eta} \sum_{i=1}^n |\nabla_{\vect{u}} v_i| \|\vect{y}_k\|^2 \le \frac{2}{\eta \tau} \|\vect{u}\| \sum_{k=1}^K w_k \|\vect{y}_k\|^2 \le \frac{2 W_{\max}}{\eta \tau} \|\vect{y}\|^2 \|\vect{u}\|.
\end{equation*}

\textit{Part B (Differentiating the Bracket):}
Let $B_{k,i} = S_{k,i}(\vect{y}_k^2) - (S_{k,i}(\vect{y}_k))^2$. By the linearity of the operator, $\nabla_{\vect{u}} S_{k,i}(\vect{x}) = \sum_j (\nabla_{\vect{u}} [P_k]_{ij}) x_j$. Applying \eqref{eq: deriv_Pk}, we have an explicit formula for the derivative of the operator:
\begin{equation*}
    \nabla_{\vect{u}} S_{k,i}(\vect{x}) = \frac{1}{\eta} \Big[ S_{k,i}(\vect{x} \odot \vect{u}_k) - S_{k,i}(\vect{x})S_{k,i}(\vect{u}_k) \Big],
\end{equation*}
where $\odot$ denotes the Hadamard product. Differentiating $B_{k,i}$ yields:
\begin{equation*}
    \nabla_{\vect{u}} B_{k,i} = \frac{1}{\eta} \Big[ S_{k,i}(\vect{y}_k^2 \odot \vect{u}_k) - S_{k,i}(\vect{y}_k^2)S_{k,i}(\vect{u}_k) - 2 S_{k,i}(\vect{y}_k) \big( S_{k,i}(\vect{y}_k \odot \vect{u}_k) - S_{k,i}(\vect{y}_k)S_{k,i}(\vect{u}_k) \big) \Big].
\end{equation*}
Because $|S_{k,i}(\vect{y}_k^2 \odot \vect{u}_k)| \le \|\vect{y}_k^2 \odot \vect{u}_k\|_\infty \le \|\vect{y}_k\|^2 \|\vect{u}_k\|$ and $|S_{k,i}(\vect{y}_k \odot \vect{u}_k)| \le \|\vect{y}_k\| \|\vect{u}_k\|$, applying the triangle inequality bounds the entire bracket strictly:
\begin{equation*}
    |\nabla_{\vect{u}} B_{k,i}| \le \frac{1}{\eta} \Big[ \|\vect{y}_k\|^2 \|\vect{u}_k\| + \|\vect{y}_k\|^2 \|\vect{u}_k\| + 2\|\vect{y}_k\| \big( \|\vect{y}_k\|\|\vect{u}_k\| + \|\vect{y}_k\|\|\vect{u}_k\| \big) \Big] = \frac{6}{\eta} \|\vect{y}_k\|^2 \|\vect{u}_k\|.
\end{equation*}
Summing this strictly positive bound over $i$ (where $\sum_{i=1}^n v_i = 1$) and $k$ gives:
\begin{equation*}
    \sum_{k=1}^K \frac{w_k}{\eta} \frac{6}{\eta} \|\vect{y}_k\|^2 \|\vect{u}_k\| \le \frac{6 W_{\max}}{\eta^2} \|\vect{u}\| \sum_{k=1}^K \|\vect{y}_k\|^2 \le \frac{6 W_{\max}}{\eta^2} \|\vect{y}\|^2 \|\vect{u}\|.
\end{equation*}
Combining Parts A and B, the absolute bound for the local component is:
\begin{equation}
\label{eq: bound_Q_eta}
    |\nabla_{\vect{u}} Q_\eta(\vect{y})| \le \left( \frac{2 W_{\max}}{\eta \tau} + \frac{6 W_{\max}}{\eta^2} \right) \|\vect{y}\|^2 \|\vect{u}\|.
\end{equation}

\textbf{Step 4: Bounding the global component $\nabla_{\vect{u}} Q_\tau(\vect{y})$.} We differentiate \eqref{eq: Q_tau}, expanding the terms as:
\begin{equation*}
    \nabla_{\vect{u}} Q_\tau(\vect{y}) = \frac{1}{\tau} \sum_{i=1}^n (\nabla_{\vect{u}} v_i) \big[ R_i(\vect{y})^2 - 2 \bar{R}(\vect{y}) R_i(\vect{y}) \big] + \frac{2}{\tau} \sum_{i=1}^n v_i \big( R_i(\vect{y}) - \bar{R}(\vect{y}) \big) \nabla_{\vect{u}} R_i(\vect{y}).
\end{equation*}

\textit{Part C (Terms with $\nabla_{\vect{u}} v_i$):}
Because $|R_i(\vect{y})| \le \|\vect{y}\|$ and $|\bar{R}(\vect{y})| \le \|\vect{y}\|$, the bracketed expression is algebraically bounded by $|R_i(\vect{y})^2 - 2 \bar{R}(\vect{y}) R_i(\vect{y})| \le 3\|\vect{y}\|^2$. Substituting the bound for $|\nabla_{\vect{u}} v_i|$ yields:
\begin{equation*}
    \frac{1}{\tau} \sum_{i=1}^n \left( \frac{2}{\tau} v_i \|\vect{u}\| \right) \left( 3 \|\vect{y}\|^2 \right) = \frac{6}{\tau^2} \|\vect{y}\|^2 \|\vect{u}\|.
\end{equation*}

\textit{Part D (Terms with $\nabla_{\vect{u}} R_i$):}
By the linearity of the operator $R_i$, we have $\nabla_{\vect{u}} R_i(\vect{y}) = \sum_k w_k \nabla_{\vect{u}} S_{k,i}(\vect{y}_k)$. Based on our previous derivation for the operator derivative, $|\nabla_{\vect{u}} S_{k,i}(\vect{y}_k)| \le \frac{2}{\eta} \|\vect{y}_k\| \|\vect{u}_k\|$. Therefore, $|\nabla_{\vect{u}} R_i(\vect{y})| \le \frac{2 W_{\max}}{\eta} \|\vect{y}\| \|\vect{u}\|$.
The leading term satisfies $|R_i(\vect{y}) - \bar{R}(\vect{y})| \le 2\|\vect{y}\|$. Multiplying these elements together bounds Part D:
\begin{equation*}
    \frac{2}{\tau} \sum_{i=1}^n v_i \Big( 2 \|\vect{y}\| \Big) \left( \frac{2 W_{\max}}{\eta} \|\vect{y}\| \|\vect{u}\| \right) = \frac{8 W_{\max}}{\eta \tau} \|\vect{y}\|^2 \|\vect{u}\|.
\end{equation*}

Combining Parts C and D, the absolute bound for the global component is:
\begin{equation}
\label{eq: bound_Q_tau}
    |\nabla_{\vect{u}} Q_\tau(\vect{y})| \le \left( \frac{6}{\tau^2} + \frac{8 W_{\max}}{\eta \tau} \right) \|\vect{y}\|^2 \|\vect{u}\|.
\end{equation}

\textbf{Step 5: Final synthesis.} By the triangle inequality, the total directional derivative is bounded by $|\nabla_{\vect{u}} Q(\vect{y})| \le |\nabla_{\vect{u}} Q_\eta(\vect{y})| + |\nabla_{\vect{u}} Q_\tau(\vect{y})|$. Summing the bounds from equations \eqref{eq: bound_Q_eta} and \eqref{eq: bound_Q_tau} establishes the global maximum spectral norm, strictly confirming the global Lipschitz constant $L_H$:
\begin{equation*}
    \sup_{\|\vect{y}\|=1, \|\vect{u}\|=1} |\nabla_{\vect{u}} (\vect{y}^\top H(\vect{\beta}) \vect{y})| \le 6 \frac{W_{\max}}{\eta^2} + 10 \frac{W_{\max}}{\eta \tau} + \frac{6}{\tau^2} \eqqcolon L_H.
\end{equation*}
This explicitly establishes the global Lipschitz continuity of the exact Hessian matrix $\nabla^2 \lyp(\vect{\beta})$.
\end{proof}

Having established globally Lipschitz continuity of the exact Hessian and the global convergence of the sequence generated by Algorithm~\ref{alg: NWB} and~\ref{alg: SNWB} in Theorem~\ref{thm: convergence}, we now turn our attention to its local convergence behavior. A critical requirement for Newton-type methods to achieve their characteristic rapid convergence rates near the optimum is that the line search mechanism must not prematurely truncate the updates. The following lemma guarantees that after a sufficient number of iterations, the Armijo backtracking procedure naturally accepts the full Newton step ($s^t = 1$) without premature truncation.

\begin{lemma}
\label{lem: full_step}
Let $\{\vect{\beta}^t\}_{t \ge 0}$ be the sequence generated by Algorithm~\ref{alg: NWB} or~\ref{alg: SNWB}. If the Armijo parameter satisfies $\sigma \in (0, 1/2)$, then for sufficiently large $t$, the Armijo line search condition is strictly satisfied with the full Newton step; that is, $s^t = 1$ is eventually accepted.
\end{lemma}

\begin{proof}
Since the exact Hessian $H(\vect{\beta})$ is globally Lipschitz continuous with the constant $L_H$ in Lemma~\ref{lem: global lipschitz hessian}, applying the standard second-order Taylor expansion bound for functions with Lipschitz continuous Hessians (e.g., \cite[Lemma 1.2.4]{nesterov2018lectures}), we have:
\begin{equation*}
    \mathcal{L}(\vect{\beta}^t + \Delta\vect{\beta}^t) - \mathcal{L}(\vect{\beta}^t) - (\vect{g}^t)^\top \Delta\vect{\beta}^t - \frac{1}{2}(\Delta\vect{\beta}^t)^\top H^t \Delta\vect{\beta}^t \le \frac{L_H}{6} \|\Delta\vect{\beta}^t\|^3.
\end{equation*}

To demonstrate that the full step $s^t = 1$ is accepted, we define the Armijo difference $D_{\text{Armijo}}$ and must show that $D_{\text{Armijo}} \le 0$:
\begin{align*}
    D_{\text{Armijo}} &:= \mathcal{L}(\vect{\beta}^t + \Delta\vect{\beta}^t) - \mathcal{L}(\vect{\beta}^t) - \sigma (\vect{g}^t)^\top \Delta\vect{\beta}^t \\
    &\le (1-\sigma)(\vect{g}^t)^\top \Delta\vect{\beta}^t + \frac{1}{2}(\Delta\vect{\beta}^t)^\top H^t \Delta\vect{\beta}^t + \frac{L_H}{6} \|\Delta\vect{\beta}^t\|^3.
\end{align*}

Recall the regularized Newton system from Algorithm~\ref{alg: NWB} or~\ref{alg: SNWB}: $(H_\rho^t + \|\vect{g}^t\|I)\Delta\vect{\beta}^t = -\vect{g}^t$. Multiplying both sides from the left by $(\Delta\vect{\beta}^t)^\top$ yields:
\begin{equation*}
    (\vect{g}^t)^\top \Delta\vect{\beta}^t = -(\Delta\vect{\beta}^t)^\top H_\rho^t \Delta\vect{\beta}^t - \|\vect{g}^t\| \|\Delta\vect{\beta}^t\|^2.
\end{equation*}

Substituting this relation into the bound for $D_{\text{Armijo}}$, we obtain:
\begin{align}
    D_{\text{Armijo}} &\le -(1-\sigma) \left( (\Delta\vect{\beta}^t)^\top H_\rho^t \Delta\vect{\beta}^t + \|\vect{g}^t\| \|\Delta\vect{\beta}^t\|^2 \right) + \frac{1}{2}(\Delta\vect{\beta}^t)^\top H^t \Delta\vect{\beta}^t + \frac{L_H}{6} \|\Delta\vect{\beta}^t\|^3 \nonumber \\
    &= \frac{1}{2}(\Delta\vect{\beta}^t)^\top (H^t - H_\rho^t) \Delta\vect{\beta}^t - \left(\frac{1}{2} - \sigma\right)(\Delta\vect{\beta}^t)^\top H_\rho^t \Delta\vect{\beta}^t \nonumber \\
    &\quad - (1-\sigma)\|\vect{g}^t\| \|\Delta\vect{\beta}^t\|^2 + \frac{L_H}{6} \|\Delta\vect{\beta}^t\|^3. \label{eq:armijo_bound_2}
\end{align}

We now systematically bound each term in this inequality. First, from Theorem~\ref{thm: hessian diff} and the algorithmic updating rule $\rho = C_\rho \|\vect{g}^t\|$, the Hessian approximation error is bounded by:
\begin{equation*}
    \|H^t - H_\rho^t\| \le 8mn\rho \left(\max_k w_k\right) \left(\frac{1}{\tau} + \frac{1}{\eta}\right) = C_E \|\vect{g}^t\|,
\end{equation*}
where $C_E \coloneqq 8mn C_\rho (\max_k w_k)(\frac{1}{\tau} + \frac{1}{\eta})$ is a strictly positive constant. Thus, the first term is bounded above by $\frac{1}{2} C_E \|\vect{g}^t\| \|\Delta\vect{\beta}^t\|^2$.

Second, from Lemma~\ref{lem: uniform lower bound hessian}, we established the existence of a uniform lower bound $\underline{C}>0$ such that $\lambda_{\min}^+(H_\rho) \ge \underline{C}$ for any $\vect{\beta} \in S(\vect{\beta}^0)$. Since the entire sequence of iterates $\{\vect{\beta}^t\}_{t \ge 0}$ remains confined to $S(\vect{\beta}^0)$ and natively resides in $\ker(H_\rho)^\perp$, we have the quadratic lower bound $(\Delta\vect{\beta}^t)^\top H_\rho^t \Delta\vect{\beta}^t \ge \underline{C} \|\Delta\vect{\beta}^t\|^2$. Furthermore, the proof of Lemma~\ref{lem: backtracking} provides the norm bound $\|\Delta\vect{\beta}^t\| \le \frac{1}{\underline{C}} \|\vect{g}^t\|$. 

Substituting these structural bounds back into the inequality for $D_{\text{Armijo}}$ yields:
\begin{align*}
    D_{\text{Armijo}} &\le \frac{1}{2} C_E \|\vect{g}^t\| \|\Delta\vect{\beta}^t\|^2 - \left(\frac{1}{2} - \sigma\right)\underline{C} \|\Delta\vect{\beta}^t\|^2 - (1-\sigma)\|\vect{g}^t\| \|\Delta\vect{\beta}^t\|^2 + \frac{L_H}{6\underline{C}} \|\vect{g}^t\| \|\Delta\vect{\beta}^t\|^2 \\
    &= \left[ \left( \frac{C_E}{2} - (1-\sigma) + \frac{L_H}{6\underline{C}} \right) \|\vect{g}^t\| - \left(\frac{1}{2} - \sigma\right)\underline{C} \right] \|\Delta\vect{\beta}^t\|^2 
\end{align*}

Because $\|\vect{g}^t\| \to 0$ as $t \to \infty$, the coefficient enclosed in the square brackets approaches $-\left(\frac{1}{2} - \sigma\right)\underline{C}$. Given that the Armijo parameter is restricted to $\sigma \in (0, 1/2)$ and the uniform lower bound satisfies $\underline{C} > 0$, this asymptotic limit is strictly negative. Therefore, for all sufficiently large $t$, the term in the brackets becomes strictly less than zero, implying $D_{\text{Armijo}} \le 0$. This guarantees that the Armijo sufficient decrease condition is strictly satisfied with $s^t = 1$, completing the proof.
\end{proof}

We now show the local quadratic convergence of the proposed algorithms. 

\begin{proof}[Proof of Theorem~\ref{thm: local_quadratic}]
From Lemma~\ref{lem: global lipschitz hessian}, the exact Hessian $H(\vect{\beta})$ is globally Lipschitz continuous with the constant $L_H$, we have the following bound in the neighborhood of the optimum:
\begin{equation*}
    \|\vect{g}^t - \vect{g}^\star - H^t(\vect{\beta}^t - \vect{\beta}^\star)\| \le \frac{L_H}{2} \|\vect{\beta}^t - \vect{\beta}^\star\|^2,
\end{equation*}
where $\vect{g}^\star = \nabla \lyp(\vect{\beta}^\star)$.

By Theorem~\ref{thm: the minimum and maximum eigenvalues of the hessian}, the maximum eigenvalue of the exact Hessian is globally bounded by $L = (\max_k w_k)(\frac{1}{2\eta} + \frac{1}{\tau})$. This implies that the gradient is $L$-Lipschitz continuous. Since the first-order optimality condition yields $\vect{g}^\star = \0_{mK}$, we obtain the error bound:
\begin{equation*}
    \|\vect{g}^t\| = \|\vect{g}^t - \vect{g}^\star\| \le L \|\vect{\beta}^t - \vect{\beta}^\star\|.
\end{equation*}

For notational convenience, let $B_\rho^t = H_\rho^t + \|\vect{g}^t\|I$ denote the regularized sparsified Hessian utilized in Algorithm~\ref{alg: NWB} or~\ref{alg: SNWB}. Assuming the iterates are sufficiently close to the optimum such that the full Newton step $s^t = 1$ is accepted, we evaluate the distance to the optimum at the subsequent step:
\begin{align*}
    \|\vect{\beta}^{t+1} - \vect{\beta}^\star\| &= \|\vect{\beta}^t + \Delta\vect{\beta}^t - \vect{\beta}^\star\| \\
    &= \|\vect{\beta}^t - (B_\rho^t)^{-1}\vect{g}^t - \vect{\beta}^\star\| \\
    &= \|(B_\rho^t)^{-1} [ B_\rho^t(\vect{\beta}^t - \vect{\beta}^\star) - \vect{g}^t ]\|. 
\end{align*}

By Lemma~\ref{lem: sequence restrict in the subspace}, we know that $\vect{\beta}^t, \vect{\beta}^\star, \vect{g}^t \in \ker(H_\rho)^\perp$ and it follows that $B_\rho^t(\vect{\beta}^t - \vect{\beta}^\star) - \vect{g}^t \in \ker(H_\rho)^\perp$. As established in the proof of Lemma~\ref{lem: backtracking}, the restricted spectral norm of $(B_\rho^t)^{-1}$ on this invariant subspace yields the following bound:
\begin{equation*}
    \|(B_\rho^t)^{-1} [ B_\rho^t(\vect{\beta}^t - \vect{\beta}^\star) - \vect{g}^t ]\| \leq \frac{1}{\lambda_{\min}^+(H_\rho^t) + \|\vect{g}^t\|}\|B_\rho^t(\vect{\beta}^t - \vect{\beta}^\star) - \vect{g}^t\| \leq \frac{1}{\underline{C}}\|B_\rho^t(\vect{\beta}^t - \vect{\beta}^\star) - \vect{g}^t\|.
\end{equation*}

By injecting the optimality condition $\vect{g}^\star = \0_{mK}$, we can algebraically rewrite the inner norm term as:
\begin{equation*}
    \|B_\rho^t(\vect{\beta}^t - \vect{\beta}^\star) - \vect{g}^t\| = \|\vect{g}^t - \vect{g}^\star - H^t(\vect{\beta}^t - \vect{\beta}^\star) + (H^t - B_\rho^t)(\vect{\beta}^t - \vect{\beta}^\star)\|.
\end{equation*}

Next, we bound the approximation error between the exact Hessian $H^t$ and the shifted sparsified Hessian $B_\rho^t$. Using the algorithmic updating rule $\rho = C_\rho \|\vect{g}^t\|$ and the structural error bound from Theorem~\ref{thm: hessian diff}, we have $\|H^t - H_\rho^t\| \le C_E \|\vect{g}^t\|$, where $C_E = 8mn C_\rho (\max_k w_k)(\frac{1}{\tau} + \frac{1}{\eta})$. Therefore:
\begin{align*}
    \|H^t - B_\rho^t\| &= \|H^t - (H_\rho^t + \|\vect{g}^t\|I)\| \\
    &\le \|H^t - H_\rho^t\| + \|\vect{g}^t\| \\
    &\le C_E \|\vect{g}^t\| + \|\vect{g}^t\| \\
    &= (C_E + 1)\|\vect{g}^t\| \\
    &\le (C_E + 1)L \|\vect{\beta}^t - \vect{\beta}^\star\|.
\end{align*}

Applying the triangle inequality to the decomposed norm and substituting all derived bounds back into the fundamental inequality for $\|\vect{\beta}^{t+1} - \vect{\beta}^\star\|$, we obtain:
\begin{align*}
    \|\vect{\beta}^{t+1} - \vect{\beta}^\star\| &\le \frac{1}{\underline{C}} \left( \|\vect{g}^t - \vect{g}^\star - H^t(\vect{\beta}^t - \vect{\beta}^\star)\| + \|H^t - B_\rho^t\| \cdot \|\vect{\beta}^t - \vect{\beta}^\star\| \right) \\
    &\le \frac{1}{\underline{C}} \left( \frac{L_H}{2} \|\vect{\beta}^t - \vect{\beta}^\star\|^2 + (C_E + 1)L \|\vect{\beta}^t - \vect{\beta}^\star\|^2 \right) \\
    &= \frac{1}{\underline{C}} \left( \frac{L_H}{2} + (C_E + 1)L \right) \|\vect{\beta}^t - \vect{\beta}^\star\|^2.
\end{align*}

By defining the uniform constant $M = \frac{1}{\underline{C}} \left( \frac{L_H}{2} + (C_E + 1)L \right)$, we obtain the final bound $\|\vect{\beta}^{t+1} - \vect{\beta}^\star\| \le M \|\vect{\beta}^t - \vect{\beta}^\star\|^2$. This successfully demonstrates that the iterates achieve local quadratic convergence.
\end{proof}

\subsection{Proof of Proposition~\ref{prop: phase_transition_complexity}}
We give a more formal theorem to prove the proposition~\ref{prop: phase_transition_complexity}.
\begin{theorem}
Let $\{\vect{\beta}^t\}_{t\geq 0}$ be generated by Algorithm~\ref{alg: NWB} or~\ref{alg: SNWB}, the Armijo parameter satisfies $\sigma \in (0, 1/2)$ and let
$\underline{\mathcal{L}}:=\mathcal{L}(\vect{\beta}^\star)$ be the global minimum. Define
$C_A \coloneqq \frac{C_E}{2}-(1-\sigma)+\frac{L_H}{6\underline{C}}$.

If $C_A\leq 0$, then the Armijo line search accepts the full Newton step, i.e., $s^t=1$, for all $t\geq 0$.

If $C_A>0$, define
\begin{equation*}
    \epsilon_{\mathrm{phase}} \coloneqq
    \min\left\{
    \epsilon_{\text{Newton}}, \ \epsilon_{\text{contract}}
    \right\}, \quad  \text{where}\ \epsilon_{\text{Newton}} \coloneqq \frac{\left(\frac{1}{2} - \sigma\right)\underline{C}}{C_A}, \qquad
    \epsilon_{\text{contract}} \coloneqq \frac{\underline{C}^2}{LM}.
\end{equation*}
Then, once $\|\vect{g}^t\|\leq \epsilon_{\mathrm{phase}}$, all subsequent iterations accept the full Newton step and enter the local quadratic convergence regime. Moreover, the number of iterations before this phase is reached is bounded by
\begin{equation*}
    T_{\mathrm{global}}
    \leq
    \left\lfloor
    \frac{
    \big(\mathcal{L}(\vect{\beta}^0)-\underline{\mathcal{L}}\big)(L+G_{\max})
    }{
    \sigma s_{\min}\epsilon_{\mathrm{phase}}^2
    }
    \right\rfloor .
\end{equation*}
\end{theorem}

\begin{proof}
To guarantee that the algorithm permanently transitions into the pure Newton phase, it is sufficient to reach a state where the full Newton step ($s^t = 1$) is unconditionally accepted and the subsequent gradient norm does not increase ($\|\vect{g}^{t+1}\| \le \|\vect{g}^t\|$). This ensures that once the iterates enter this local convergence region, they will not escape back into the global phase where the backtracking line search is required.

First, we revisit the condition for accepting the full Newton step. In the proof of Lemma~\ref{lem: full_step}, we established that the Armijo condition is naturally satisfied for $s^t = 1$ when the difference $D_{\text{Armijo}} \le 0$. We demonstrated this is algebraically guaranteed when:
\begin{equation*}
    C_A \|\vect{g}^t\| - \left(\frac{1}{2} - \sigma\right)\underline{C} \le 0.
\end{equation*}

This leads to two distinct cases. If $C_A \le 0$, the inequality holds trivially for any valid gradient norm because the subtracted term is strictly positive ($\sigma < 1/2$ and $\underline{C} > 0$). In this scenario, the Armijo condition never truncates the step, and the algorithm executes pure Newton steps globally from $t=0$. 

Conversely, if $C_A > 0$, isolating $\|\vect{g}^t\|$ yields the explicit finite threshold $\epsilon_{\text{Newton}}$. Thus, whenever $\|\vect{g}^t\| \le \epsilon_{\text{Newton}}$, the full step $s^t = 1$ is guaranteed to be accepted.

Second, we determine the non-increasing gradient threshold. From Theorem~\ref{thm: local_quadratic}, a full Newton step yields the local quadratic bound $\|\vect{\beta}^{t+1} - \vect{\beta}^\star\| \le M \|\vect{\beta}^t - \vect{\beta}^\star\|^2$. Using the $L$-Lipschitz continuity of the gradient and the strong convexity parameter $\underline{C}$ on the restricted subspace, we relate this distance bound back to the gradient norm:
\begin{equation*}
    \|\vect{g}^{t+1}\| \le L \|\vect{\beta}^{t+1} - \vect{\beta}^\star\| \le LM \|\vect{\beta}^t - \vect{\beta}^\star\|^2 \le \frac{LM}{\underline{C}^2} \|\vect{g}^t\|^2.
\end{equation*}
For the sequence of gradient norms to be non-increasing (i.e., $\|\vect{g}^{t+1}\| \le \|\vect{g}^t\|$), it is sufficient that the coefficient satisfies $\frac{LM}{\underline{C}^2} \|\vect{g}^t\| \le 1$, which algebraically reduces to $\|\vect{g}^t\| \le \frac{\underline{C}^2}{LM} \eqqcolon \epsilon_{\text{contract}}$.

Therefore, assuming $C_A > 0$, if an iterate satisfies the joint condition $\|\vect{g}^t\| \le \min(\epsilon_{\text{Newton}}, \epsilon_{\text{contract}}) \eqqcolon \epsilon_{\text{phase}}$, the algorithm accepts the full step $s^t = 1$ and ensures that $\|\vect{g}^{t+1}\| \le \|\vect{g}^t\| \le \epsilon_{\text{phase}}$. By mathematical induction, all subsequent iterations will remain bounded by this threshold, executing pure Newton steps indefinitely.

Consequently, to mathematically bound the maximum number of iterations $T_{\text{global}}$ prior to this permanent transition, we analyze the worst-case trajectory where the iterates have not yet satisfied this joint sufficient condition, meaning $\|\vect{g}^t\| > \epsilon_{\text{phase}}$. 

We now recall the uniform sufficient decrease condition established in the proof of Theorem~\ref{thm: convergence}:
\begin{equation*}
    \mathcal{L}(\vect{\beta}^t) - \mathcal{L}(\vect{\beta}^{t+1}) \ge \sigma s_{\min} \frac{1}{L + G_{\max}} \|\vect{g}^t\|^2.
\end{equation*}
Coupling this with the condition $\|\vect{g}^t\| > \epsilon_{\text{phase}}$, the objective function is guaranteed to decrease by a strictly bounded minimum amount at every step prior to the guaranteed phase transition:
\begin{equation*}
    \mathcal{L}(\vect{\beta}^t) - \mathcal{L}(\vect{\beta}^{t+1}) > \frac{\sigma s_{\min}}{L + G_{\max}} \epsilon_{\text{phase}}^2 \eqqcolon \Delta_{\min}.
\end{equation*}

Because the objective function evaluates to $\mathcal{L}(\vect{\beta}^0)$ at initialization and is bounded globally from below by $\underline{\mathcal{L}}$, the cumulative decrease over the first $T_{\text{global}}$ iterations cannot exceed $\mathcal{L}(\vect{\beta}^0) - \underline{\mathcal{L}}$. 

Summing the functional decreases over these initial global iterations yields:
\begin{equation*}
    T_{\text{global}} \cdot \Delta_{\min} \le \sum_{t=0}^{T_{\text{global}}-1} (\mathcal{L}(\vect{\beta}^t) - \mathcal{L}(\vect{\beta}^{t+1})) \le \mathcal{L}(\vect{\beta}^0) - \underline{\mathcal{L}}.
\end{equation*}

Rearranging this inequality for $T_{\text{global}}$ and substituting the definition of $\Delta_{\min}$ yields the final bound on the number of iterations required to guarantee reaching the permanent quadratic phase:
\begin{equation*}
    T_{\text{global}} \le \frac{\mathcal{L}(\vect{\beta}^0) - \underline{\mathcal{L}}}{\Delta_{\min}} = \frac{(\mathcal{L}(\vect{\beta}^0) - \underline{\mathcal{L}})(L + G_{\max})}{\sigma s_{\min} \epsilon_{\text{phase}}^2}.
\end{equation*}
Taking the floor of this bound completes the proof.
\end{proof}

We now utilize the established theorem to explicitly derive the asymptotic iteration complexity for the exact Newton method, proving Proposition \ref{prop: phase_transition_complexity}.

\begin{proof}[Proof of Proposition \ref{prop: phase_transition_complexity}]
To evaluate the theoretical bound $T_{\text{global}} \le \mathcal{O}\left( \frac{L + G_{\max}}{s_{\min}\epsilon_{\text{phase}}^2} \right)$, we establish the asymptotic order of the core parameters. 

Assume the barycentric weights satisfy $w_k = \Theta(1/K)$, implying that no single target distribution dominates the barycenter. Thus, both $W_{\max}$ and $W_{\min}$ are $\Theta(1/K)$. This strictly bounds the maximum gradient norm: $G_{\max} = 2\sqrt{\sum_k w_k^2} = \mathcal{O}(1/\sqrt{K}) = \mathcal{O}(1)$.

Setting $m = n$ and assuming the outer regularization scales proportionally to the inner regularization, $\tau = \Theta(\eta)$, the gradient and Hessian Lipschitz constants scale as:
\begin{equation*}
    L = W_{\max}\left(\frac{1}{2\eta} + \frac{1}{\tau}\right) = \mathcal{O}\left(\frac{1}{\eta K}\right).
\end{equation*}
\begin{equation*}
    L_H = 6 \frac{W_{\max}}{\eta^2} + 10 \frac{W_{\max}}{\eta \tau} + \frac{6}{\tau^2} = \mathcal{O}\left(\frac{1}{\eta^2}\right).
\end{equation*}
Notice that the $6/\tau^2$ term in $L_H$ is independent of $K$, making it the dominant term asymptotically. Because $L \to \infty$ as $\eta \to 0$, the gradient bound is dominated, yielding $(L + G_{\max}) = \mathcal{O}(L)$. 

Furthermore, since the target distributions $\vect{\mu}_k$ are fixed prescribed inputs, their minimum positive entry is an $\mathcal{O}(1)$ constant. Thus, the minimum transport probability strictly scales with the maximum exponent variation $\Delta$:
\begin{equation*}
    p_{\min} \ge \min_{i,j} [P_{k,\rho}]_{ij} = \mathcal{O}\left( \exp\left(-\frac{\Delta}{\eta}\right) \right).
\end{equation*}
This establishes the strong convexity parameter $\underline{C} = \frac{1}{\eta} \min_k (w_k p_k^2) = \mathcal{O}\left( \frac{p_{\min}^2}{\eta K} \right)$. 

For the sparse variant, the algorithmic sparsification threshold is defined as $C_\rho = \Theta\left(\frac{\eta}{n^{1.5}}\right)$. The resulting Hessian approximation error $C_E$ evaluates to:
\begin{equation*}
    C_E = 8n^2 C_\rho W_{\max} \left(\frac{1}{\tau} + \frac{1}{\eta}\right) = \mathcal{O}\left( \frac{n^2}{\eta} \left(\frac{\eta}{n^{1.5}}\right) \frac{1}{K} \right) = \mathcal{O}\left( \frac{\sqrt{n}}{K} \right).
\end{equation*}
This approximation error inflates the intermediate contraction constant $M$:
\begin{equation*}
    M = \mathcal{O}\left( \frac{L_H + C_E L}{\underline{C}} \right) = \mathcal{O}\left( \frac{\eta K}{p_{\min}^2} \left[ \frac{1}{\eta^2} + \frac{\sqrt{n}}{\eta K^2} \right] \right) = \mathcal{O}\left( \frac{K^2 + \sqrt{n}\eta}{\eta K p_{\min}^2} \right).
\end{equation*}

The binding threshold for entering the local quadratic phase is governed by $\epsilon_{\text{contract}}$ (which strictly dominates $\epsilon_{\text{Newton}}$ as $p_{\min} \to 0$):
\begin{equation*}
    \epsilon_{\text{phase}} = \epsilon_{\text{contract}} = \frac{\underline{C}^2}{L M} = \mathcal{O}\left( \frac{p_{\min}^4 / (\eta^2 K^2)}{ \frac{1}{\eta K} \cdot \frac{K^2 + \sqrt{n}\eta}{\eta K p_{\min}^2} } \right) = \mathcal{O}\left( \frac{p_{\min}^6}{K^2 + \sqrt{n}\eta} \right).
\end{equation*}
Additionally, utilizing $(L + G_{\max}) = \mathcal{O}(L)$, the minimum acceptable step size scales as:
\begin{equation*}
    s_{\min} = \mathcal{O}\left(\frac{\underline{C}^2}{L^2}\right) = \mathcal{O}\left( \frac{p_{\min}^4 / (\eta^2 K^2)}{1 / (\eta^2 K^2)} \right) = \mathcal{O}\left( p_{\min}^4 \right).
\end{equation*}

Substituting these explicit rates into the global iteration bound yields:
\begin{equation*}
    T_{\text{global}} = \mathcal{O}\left( \frac{L}{s_{\min}\epsilon_{\text{phase}}^2} \right) = \mathcal{O}\left( \frac{1/(\eta K)}{p_{\min}^4 \cdot \frac{p_{\min}^{12}}{(K^2 + \sqrt{n}\eta)^2}} \right) = \mathcal{O}\left( \frac{(K^2 + \sqrt{n}\eta)^2}{\eta K p_{\min}^{16}} \right).
\end{equation*}

To simplify the asymptotic notation, we apply the standard quadratic inequality $(a+b)^2 \le 2(a^2 + b^2)$ to decompose the squared summation, explicitly isolating the exact Newton complexity from the sparsification penalty:
\begin{equation*}
    T_{\text{global}} = \mathcal{O}\left( \frac{K^4 + n\eta^2}{\eta K p_{\min}^{16}} \right) = \mathcal{O}\left( \left( \frac{K^3}{\eta} + \frac{n\eta}{K} \right) \frac{1}{p_{\min}^{16}} \right).
\end{equation*}

Inserting the exponential decay rate of $p_{\min}$, we conclude that the number of iterations sufficient to guarantee entry into the purely quadratic convergence phase is strictly bounded by:
\begin{equation*}
    T_{\text{global}} = \mathcal{O}\left( \left( \frac{K^3}{\eta} + \frac{n\eta}{K} \right) \exp\left(\frac{\Delta}{\eta}\right) \right).
\end{equation*}
\end{proof}

\section{Additional Experiment Details}
\label{app: Additional Experiment Details}

\subsection{Additional Convergence Comparison on Real Data}
\label{app:real-data-residual}

In this subsection, we present an additional comparison on the MNIST and Fashion-MNIST datasets using the residual as the convergence metric. Specifically, we define the residual by

$$
\text{Residual} := \max_k\|X_k\1 - \vect{v}\| + \max_k \|X_k^\top \1 - \vect{\mu}_k\|. 
$$
This quantity measures the violation of the marginal constraints and therefore provides a direct indicator of how close the current iterate is to satisfying the optimality conditions.

Figure~\ref{fig:real-data-residual} shows the decay of the residual with respect to runtime on four representative real-data instances. The results indicate that SNWB consistently reduces the residual faster than the competing methods. Compared with NWB, SNWB attains a similar or faster convergence rate while requiring noticeably less computational time, which highlights the practical advantage of the proposed sparsification strategy.

\begin{figure*}[t]
\centering
\subcaptionbox{Digit ``2''}{
\resizebox*{0.235\textwidth}{!}{\includegraphics{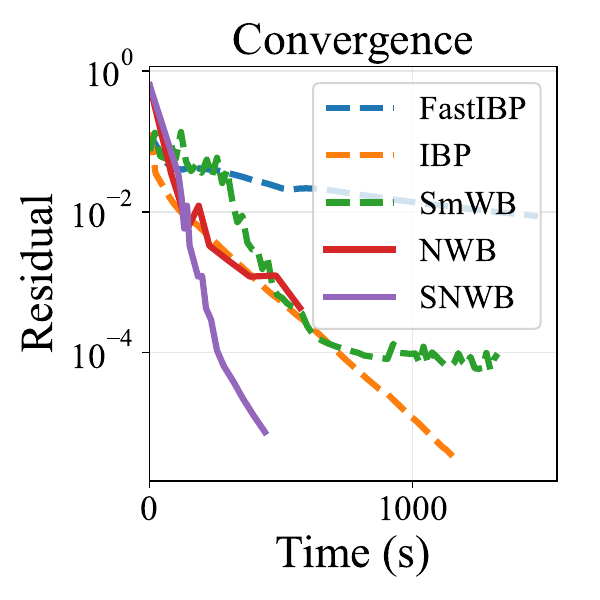}}}
\hfill
\subcaptionbox{Digit ``5''}{
\resizebox*{0.235\textwidth}{!}{\includegraphics{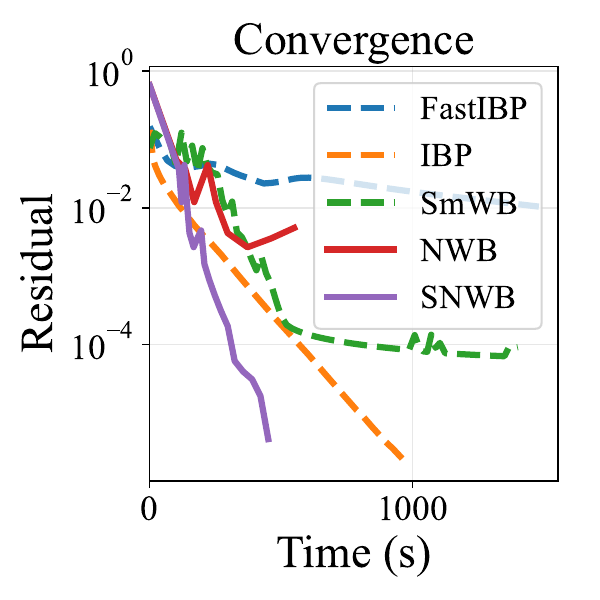}}}
\hfill
\subcaptionbox{Item ``Sneaker''}{
\resizebox*{0.235\textwidth}{!}{\includegraphics{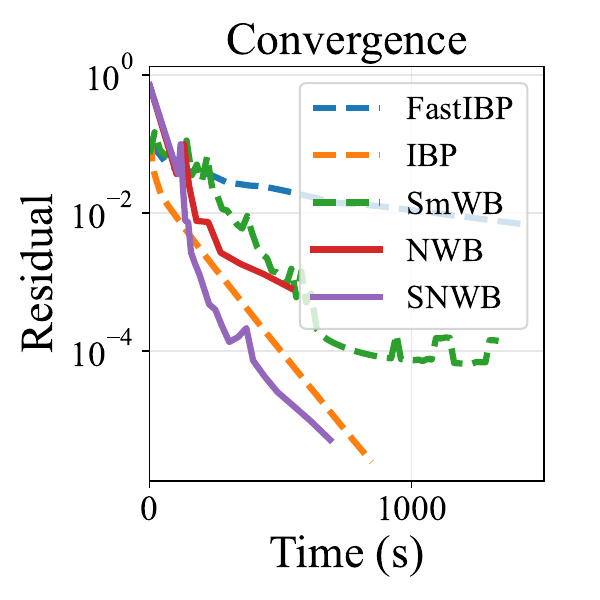}}}
\hfill
\subcaptionbox{Item ``Bag''}{
\resizebox*{0.235\textwidth}{!}{\includegraphics{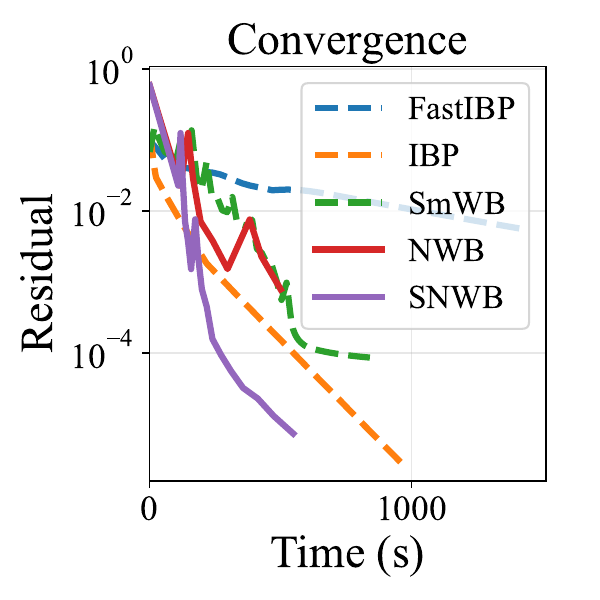}}}

\caption{Residual convergence on MNIST and Fashion-MNIST datasets. The corresponding input images are shown in Figure~\ref{fig:fashionmnist_and_mnist_performance_evaluation}.}
\label{fig:real-data-residual}
\end{figure*}

\subsection{Additional Synthetic Experiments}
\label{app: Additional Synthetic Experiments}

We further evaluate the proposed methods on synthetic Wasserstein barycenter instances following the experimental settings in~\cite{lin2020fixed,yang2021fast}. Each instance consists of $K$ discrete probability distributions in a three-dimensional space, with each distribution supported on $n$ points.

The datasets are generated as follows. For each input distribution, the support points are sampled independently from a five-component Gaussian mixture model. The corresponding probability weights $\{\vect{\mu}_k\}_{k=1}^K$ are sampled from the uniform distribution on $(0,1)$ and then normalized to have unit sum. To construct the fixed support of the barycenter, we aggregate all support points from the $K$ input distributions and apply K-means clustering to the pooled point cloud. The resulting $n$ centroids are used as the barycenter support points. For each $k$, the cost matrix $C_k$ is defined by the squared Euclidean distances between the barycenter support points and the support points of the $k$-th input distribution. The barycentric weights $\{w_k\}_{k=1}^K$ are also sampled from the uniform distribution on $(0,1)$ and normalized so that $\sum_{k=1}^K w_k=1$. In all experiments in this subsection, we set $\tau=\eta$.

The experimental results are shown in Figures~\ref{fig: Performance of different algorithms on synthetic data under 102} and~\ref{fig: Performance of different algorithms on synthetic data under 5104}. Overall, NWB and SNWB consistently achieve faster convergence than the baseline methods across different choices of the regularization parameter. When $\eta$ is relatively large, NWB can be slightly more efficient than SNWB, since the transport probability matrices are less sparse and the benefit of sparsification is less pronounced. As $\eta$ decreases, the transport plans become more concentrated, and the computational advantage of SNWB becomes increasingly evident. These results support the effectiveness of the proposed sparse Newton strategy, particularly in the small-regularization regime.

\begin{figure*}[t]
{
\resizebox*{0.32 \textwidth}{!}{\includegraphics{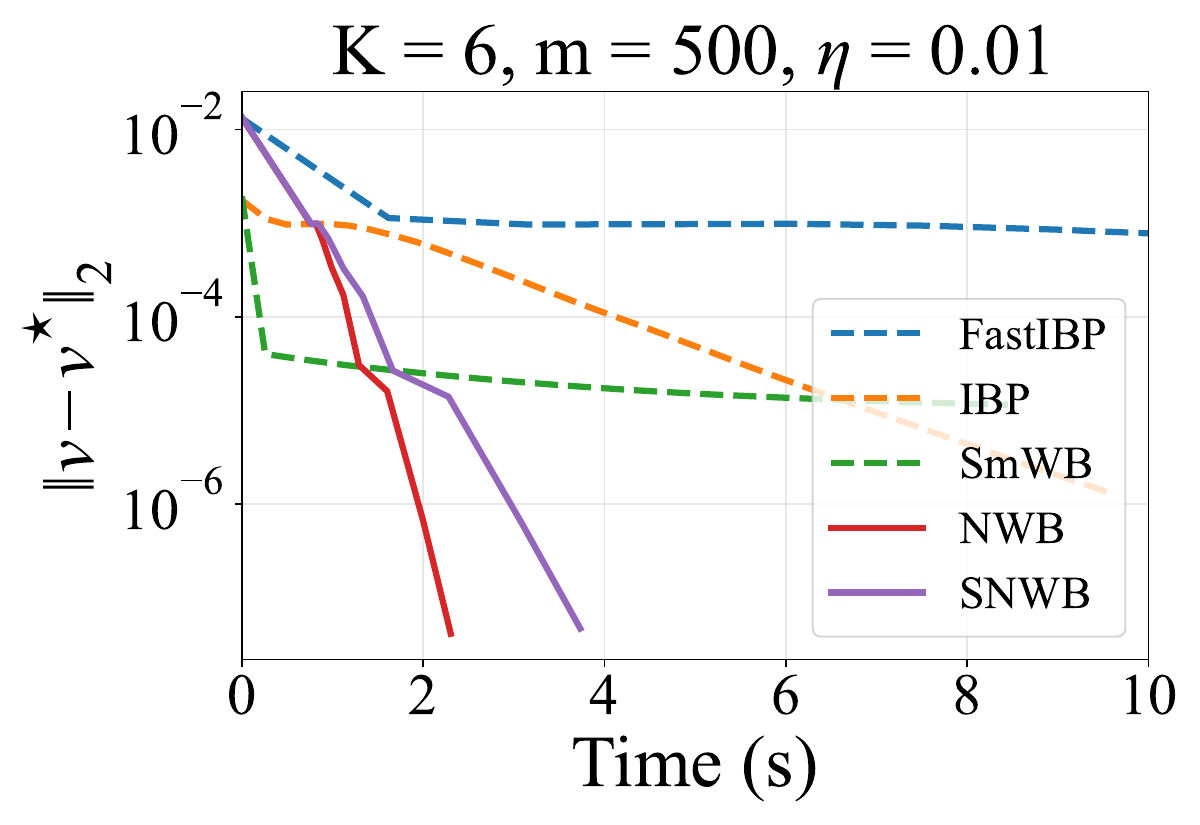}}}\hfill
{
\resizebox*{0.32 \textwidth}{!}{\includegraphics{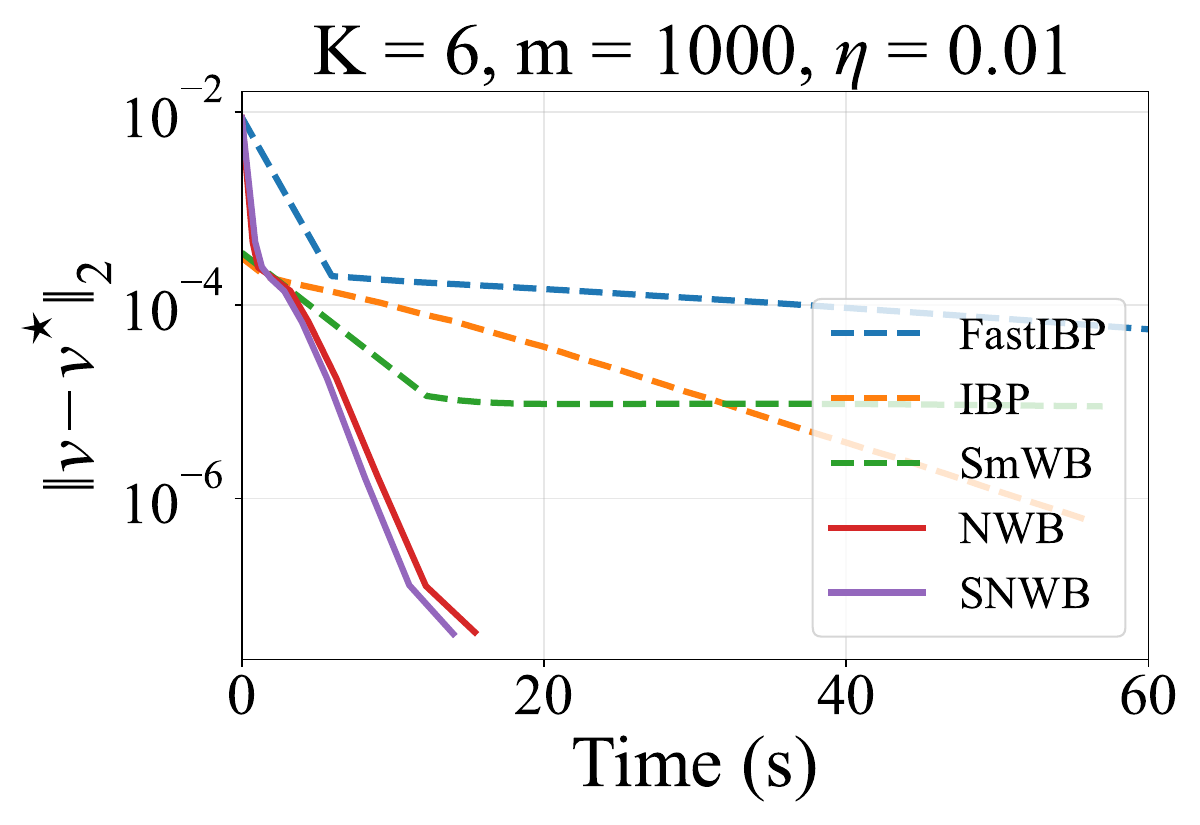}}}\hfill
{
\resizebox*{0.32 \textwidth}{!}{\includegraphics{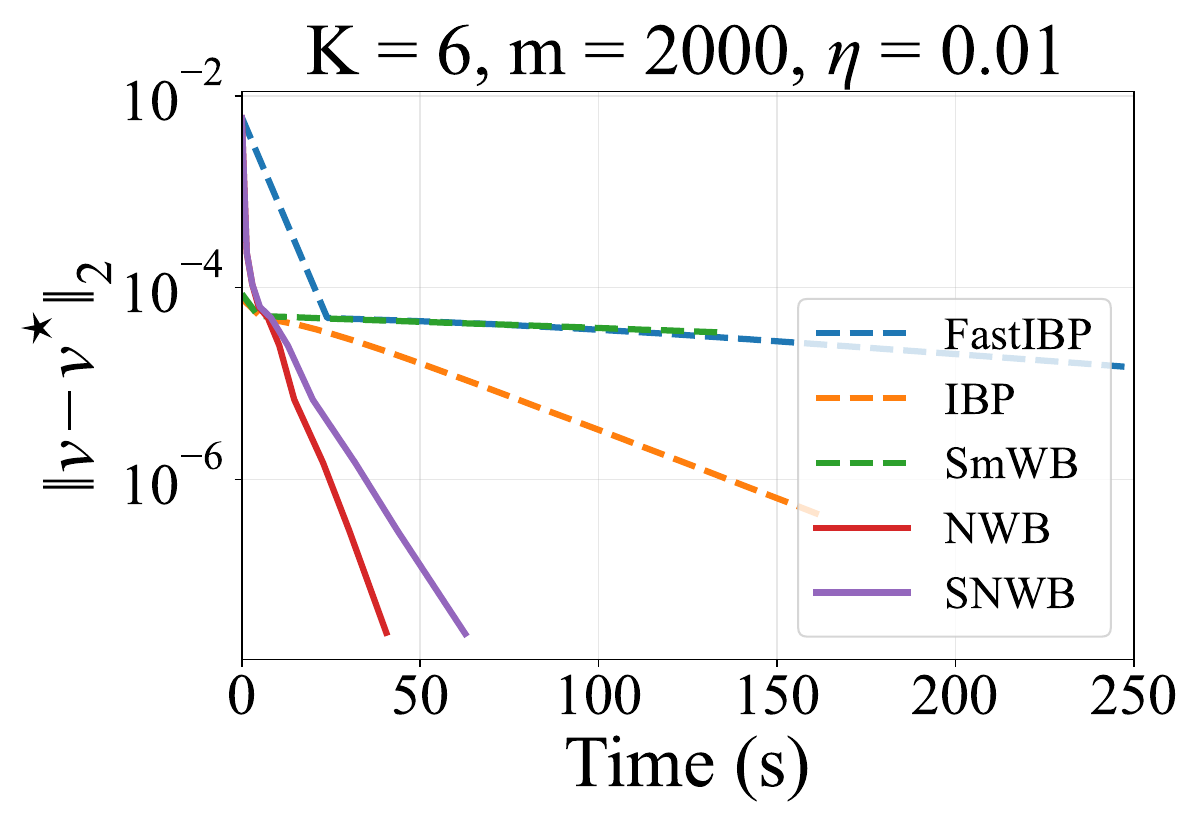}}}
{
\resizebox*{0.32 \textwidth}{!}{\includegraphics{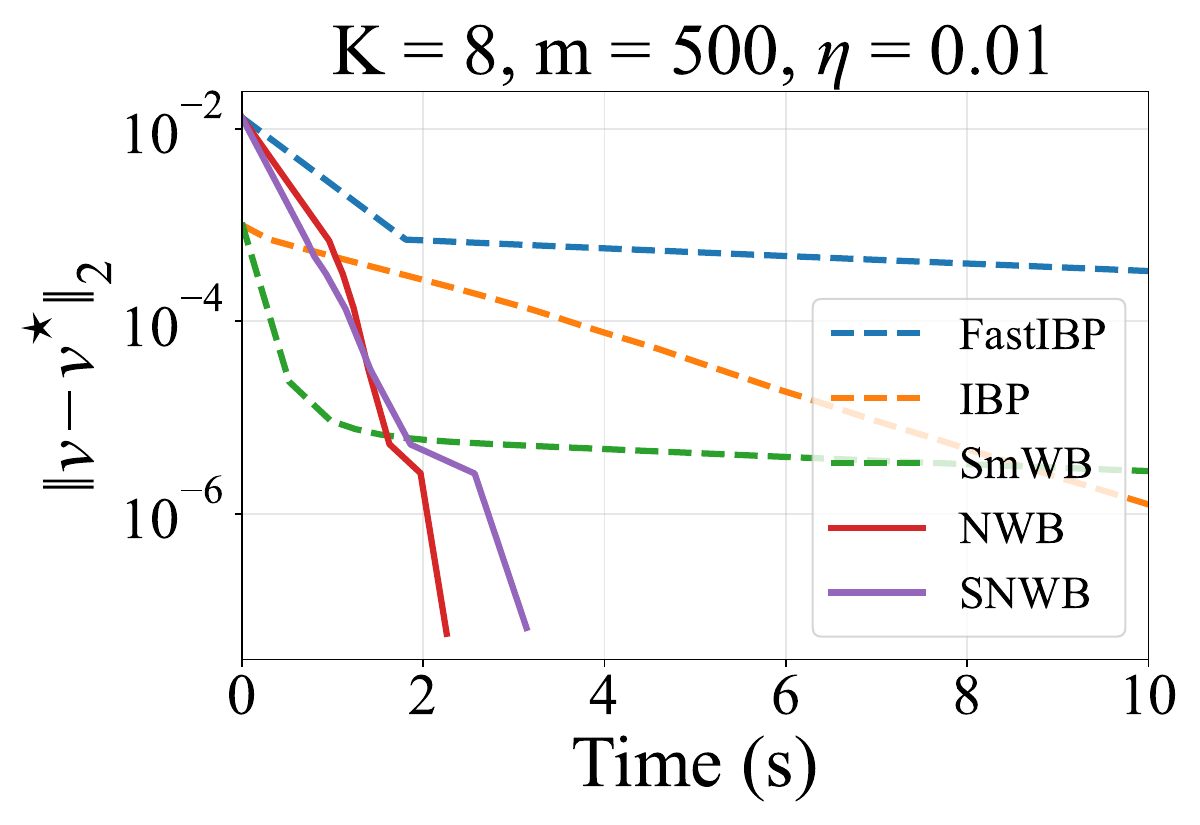}}}\hfill
{
\resizebox*{0.32 \textwidth}{!}{\includegraphics{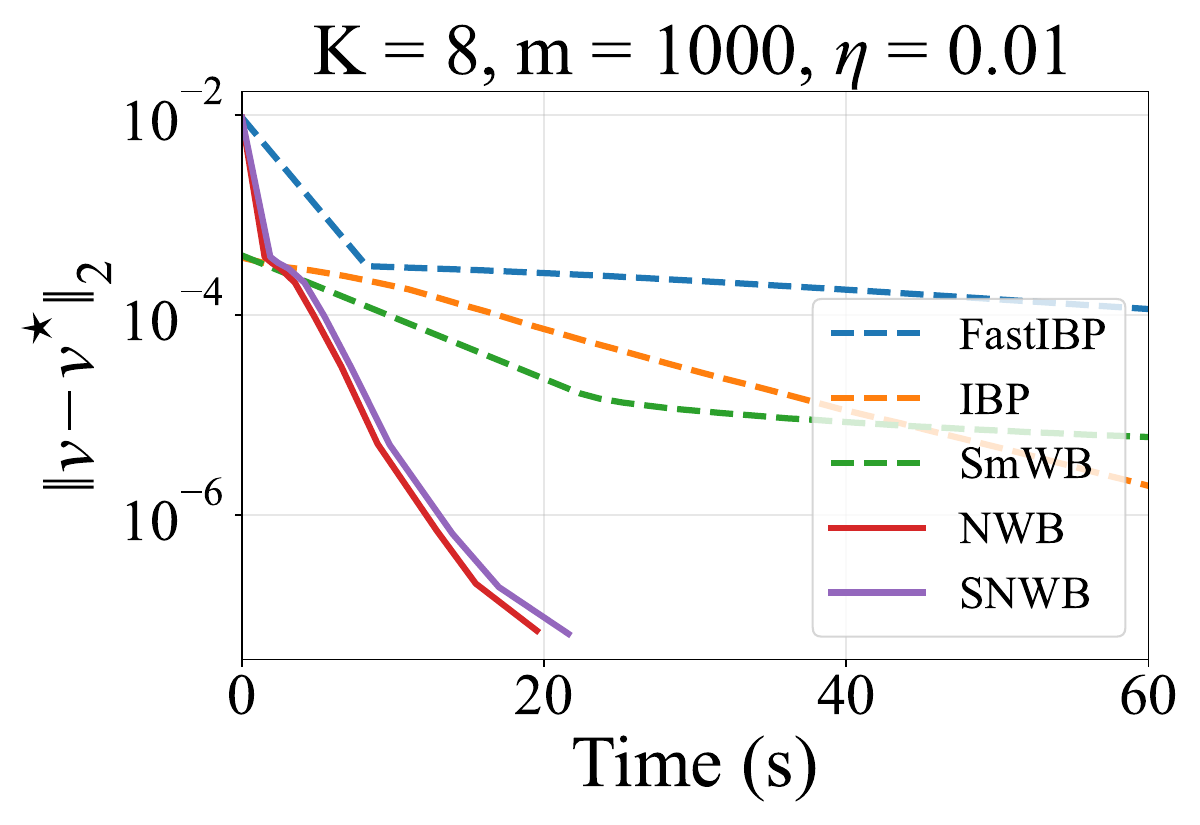}}}\hfill
{
\resizebox*{0.32 \textwidth}{!}{\includegraphics{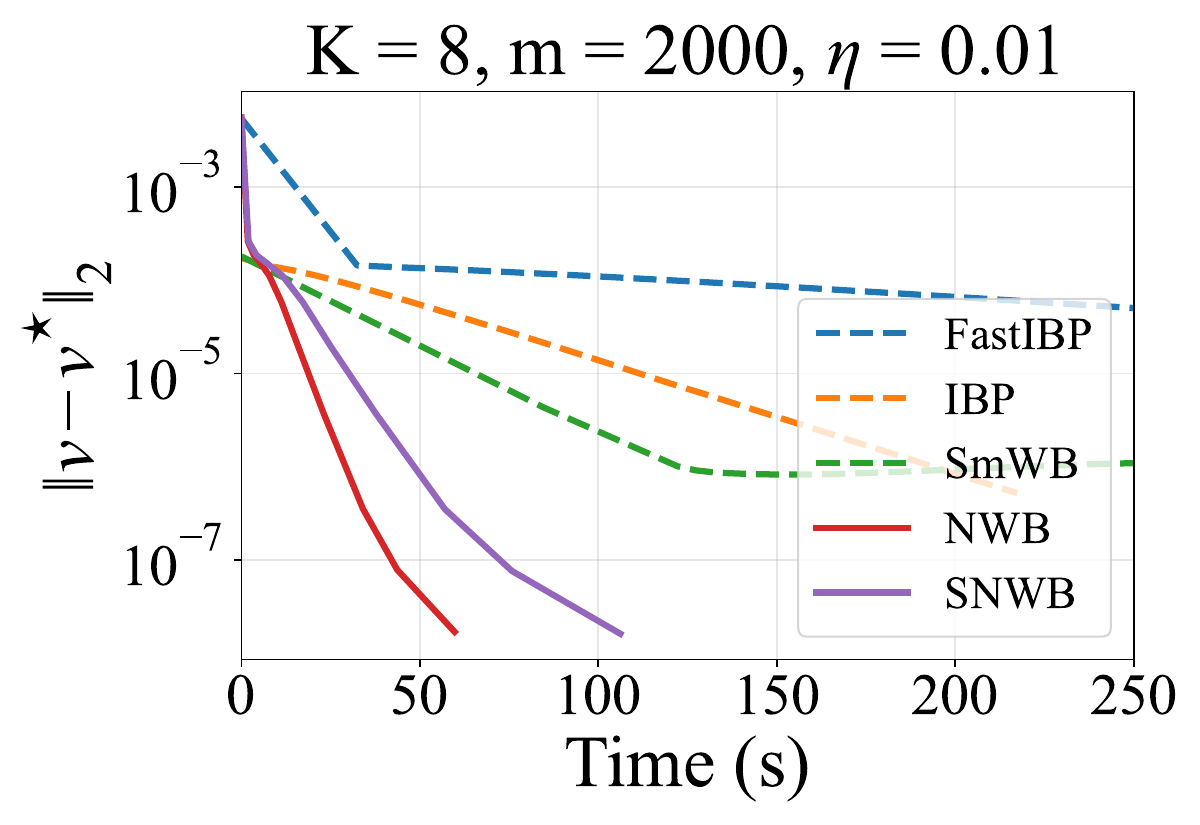}}}
\caption{Performance of different algorithms on synthetic data under $\eta = \tau = 10^{-2}$. Top: $K=6$. Bottom: $K=8$. }
\label{fig: Performance of different algorithms on synthetic data under 102}
\end{figure*}

\begin{figure*}[t]
{
\resizebox*{0.32 \textwidth}{!}{\includegraphics{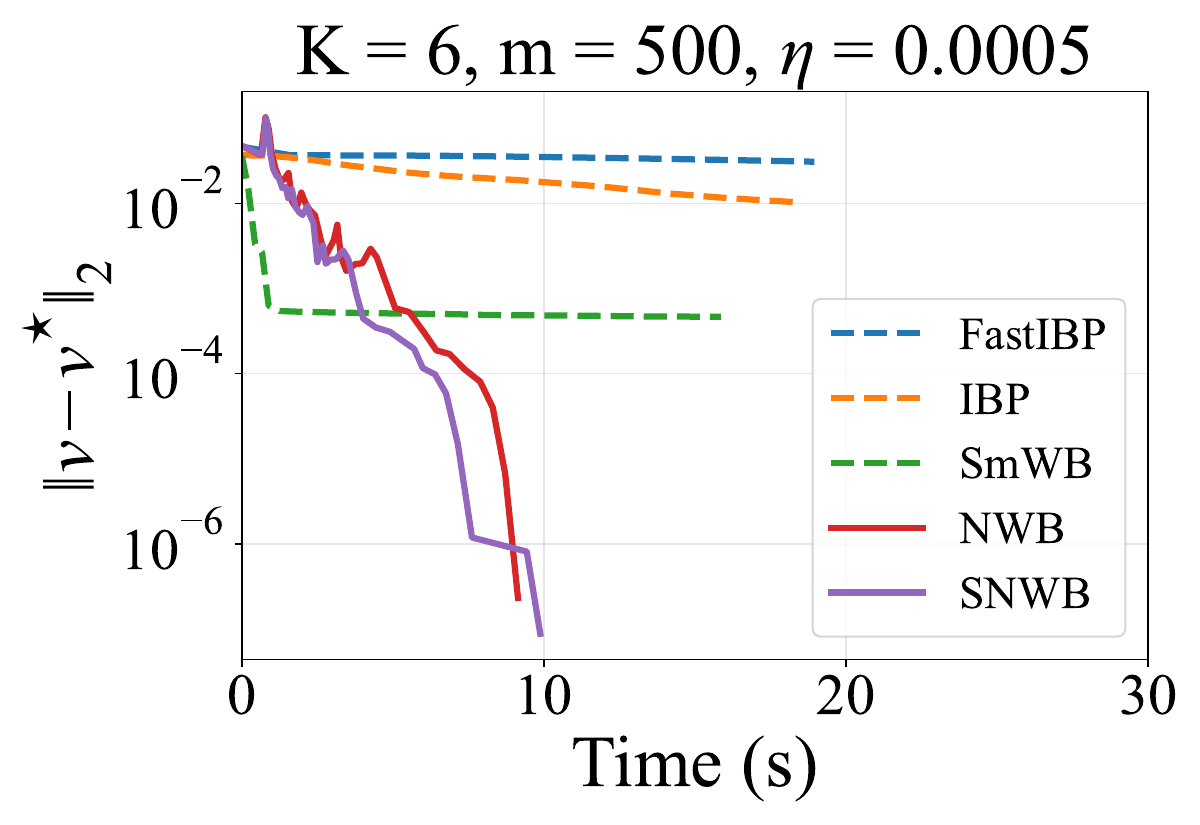}}}\hfill
{
\resizebox*{0.32 \textwidth}{!}{\includegraphics{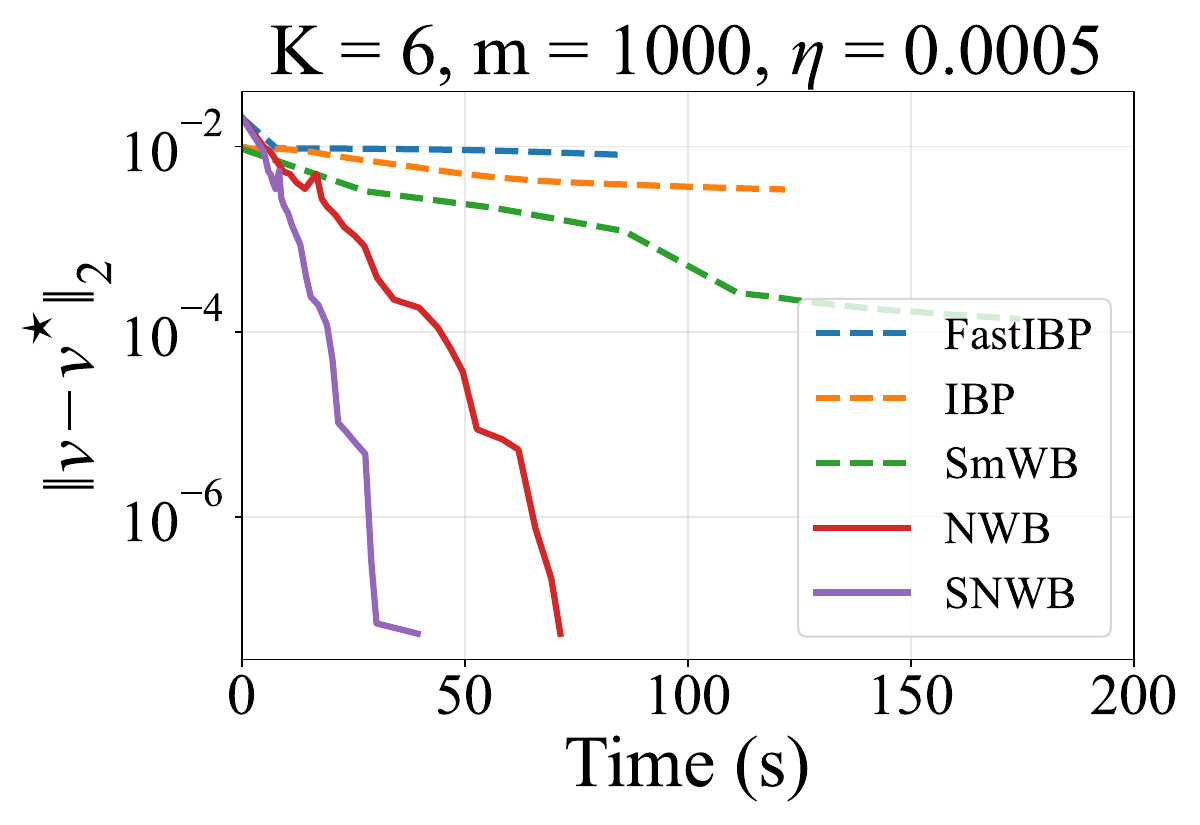}}}\hfill
{
\resizebox*{0.32 \textwidth}{!}{\includegraphics{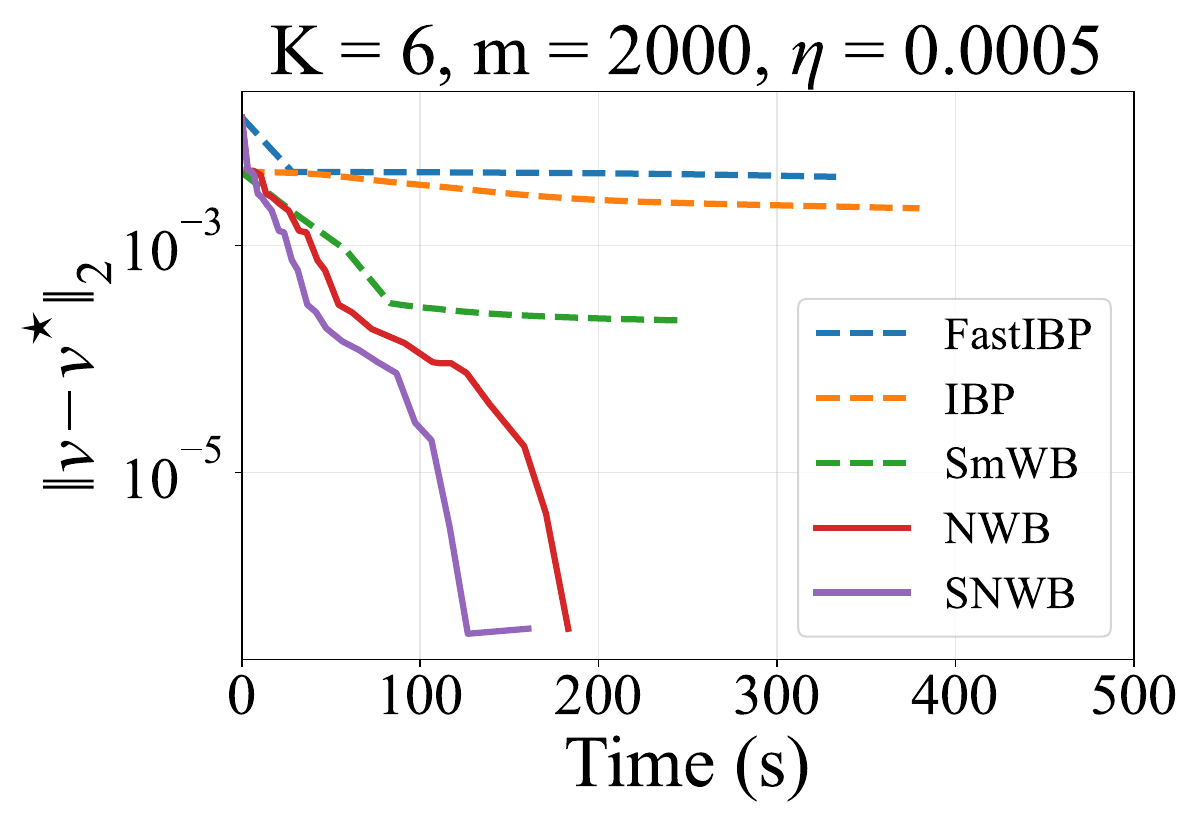}}}
{
\resizebox*{0.32 \textwidth}{!}{\includegraphics{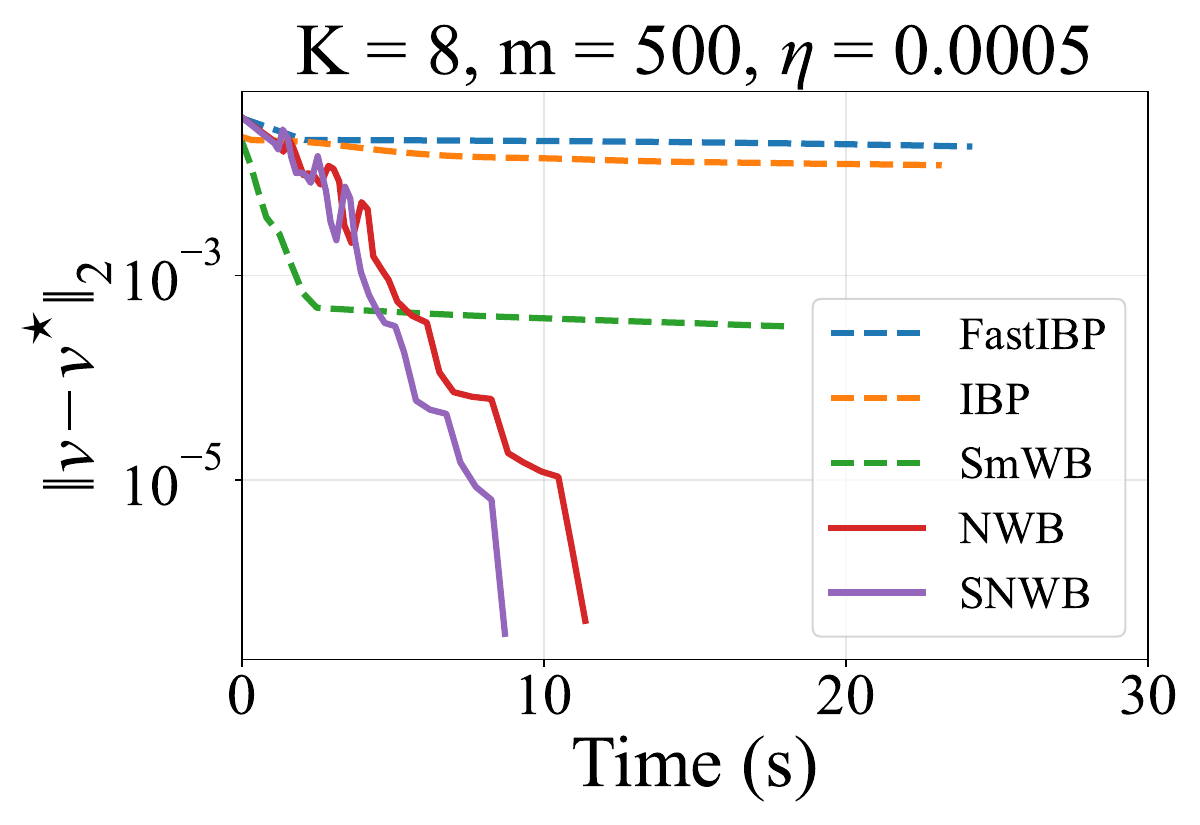}}}\hfill
{
\resizebox*{0.32 \textwidth}{!}{\includegraphics{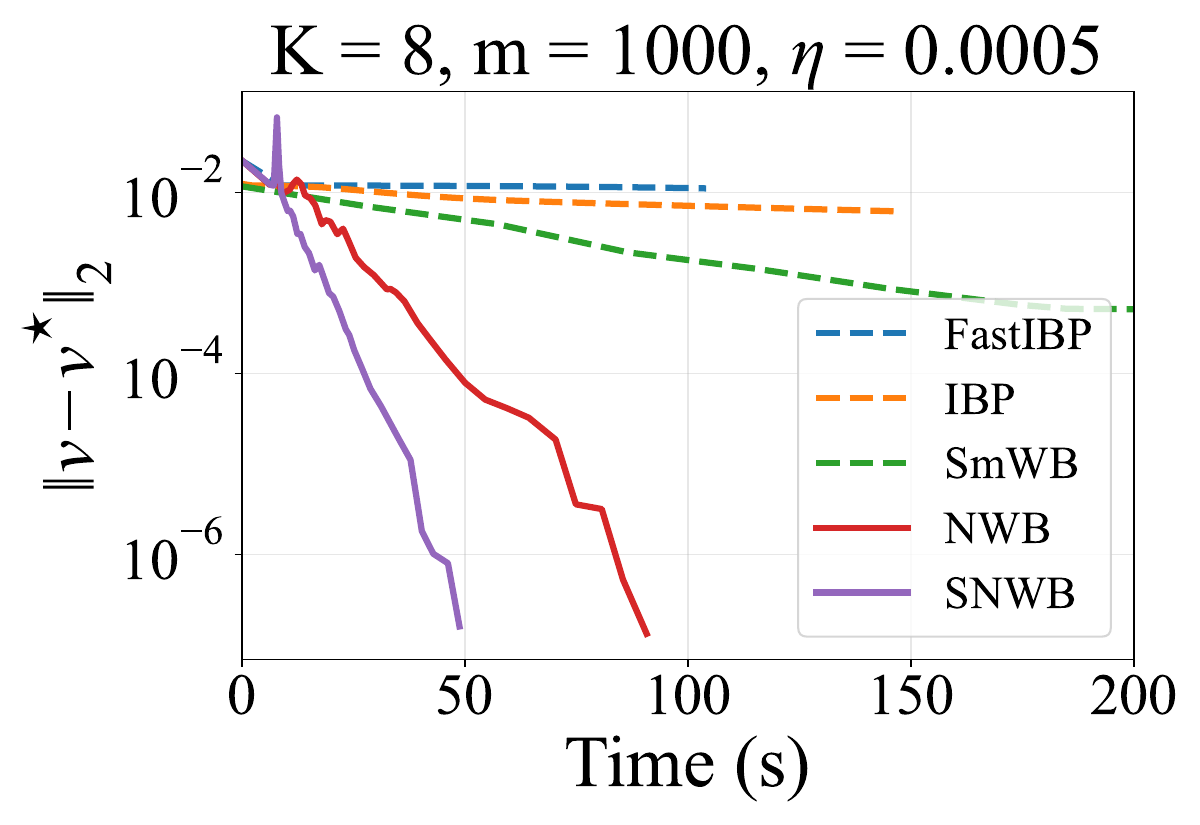}}}\hfill
{
\resizebox*{0.32 \textwidth}{!}{\includegraphics{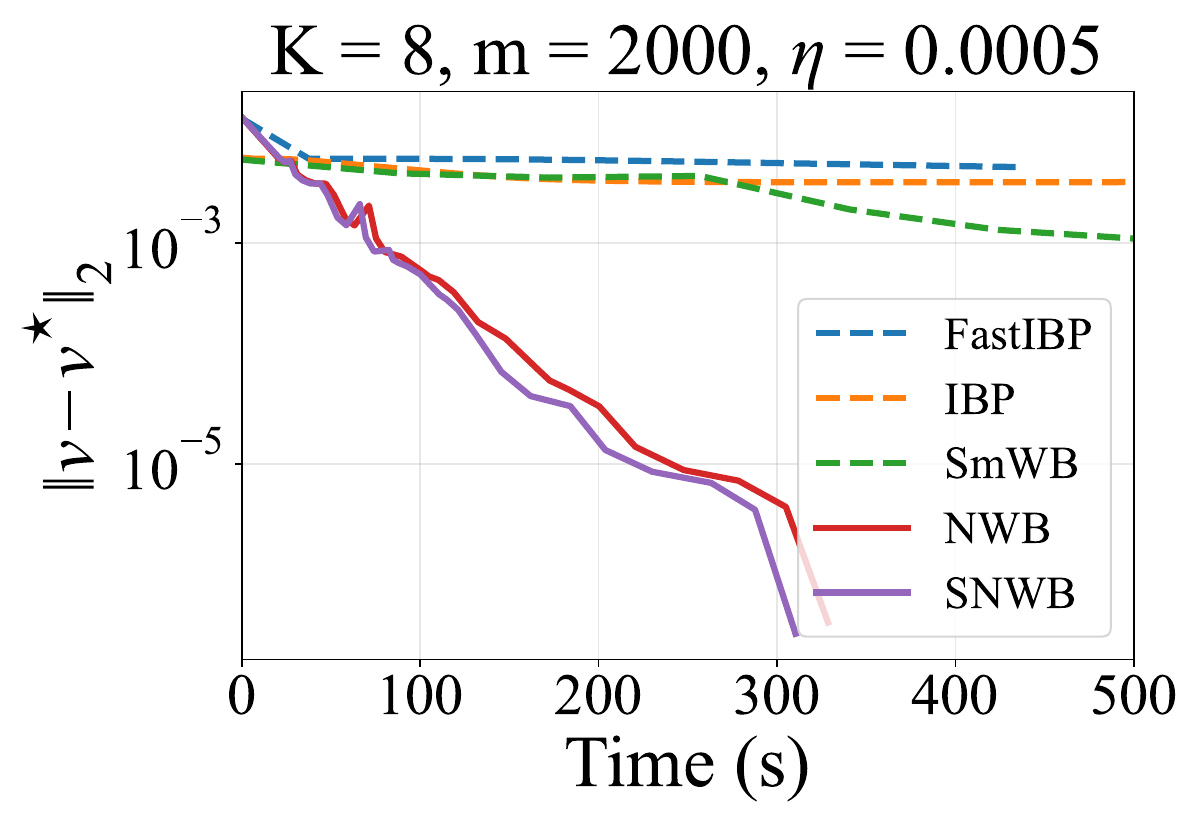}}}
\caption{Performance of different algorithms on synthetic data under $\eta = \tau = 5 \times 10^{-4}$. Top: $K=6$. Bottom: $K=8$. }
\label{fig: Performance of different algorithms on synthetic data under 5104}
\end{figure*}

\subsection{Ablation Study of $\eta$}
\label{app: Ablation study of eta}

We next examine the influence of the regularization parameter on the performance of NWB and SNWB. In this study, we set $\tau=\eta$ and vary $\eta$ on both the synthetic dataset and the MNIST dataset with $K=6$. The synthetic instances are generated following the procedure described in Appendix~\ref{app: Additional Synthetic Experiments}.

The results are reported in Figure~\ref{fig: Impact of the regularization parameter}. Overall, SNWB exhibits a clear advantage over NWB when the regularization parameter is small, particularly for $\eta=\tau=10^{-3}$ and $\eta=\tau=5\times10^{-4}$. This behavior is consistent with the fact that smaller regularization parameters typically lead to more concentrated transport probability matrices, making the sparse approximation used in SNWB more effective. In contrast, when $\eta$ is relatively large, the transport probability matrices are denser, and the benefit of sparsification becomes less pronounced. We further examine the effect of the sparsity level in Appendix~\ref{app: Ablation study of Sparsity Level}. In addition, both NWB and SNWB generally require more time to reach the prescribed tolerance as $\eta$ decreases, reflecting the increased numerical difficulty of the problem in the small-regularization regime.

\begin{figure*}[t]
{
\resizebox*{0.32 \textwidth}{!}{\includegraphics{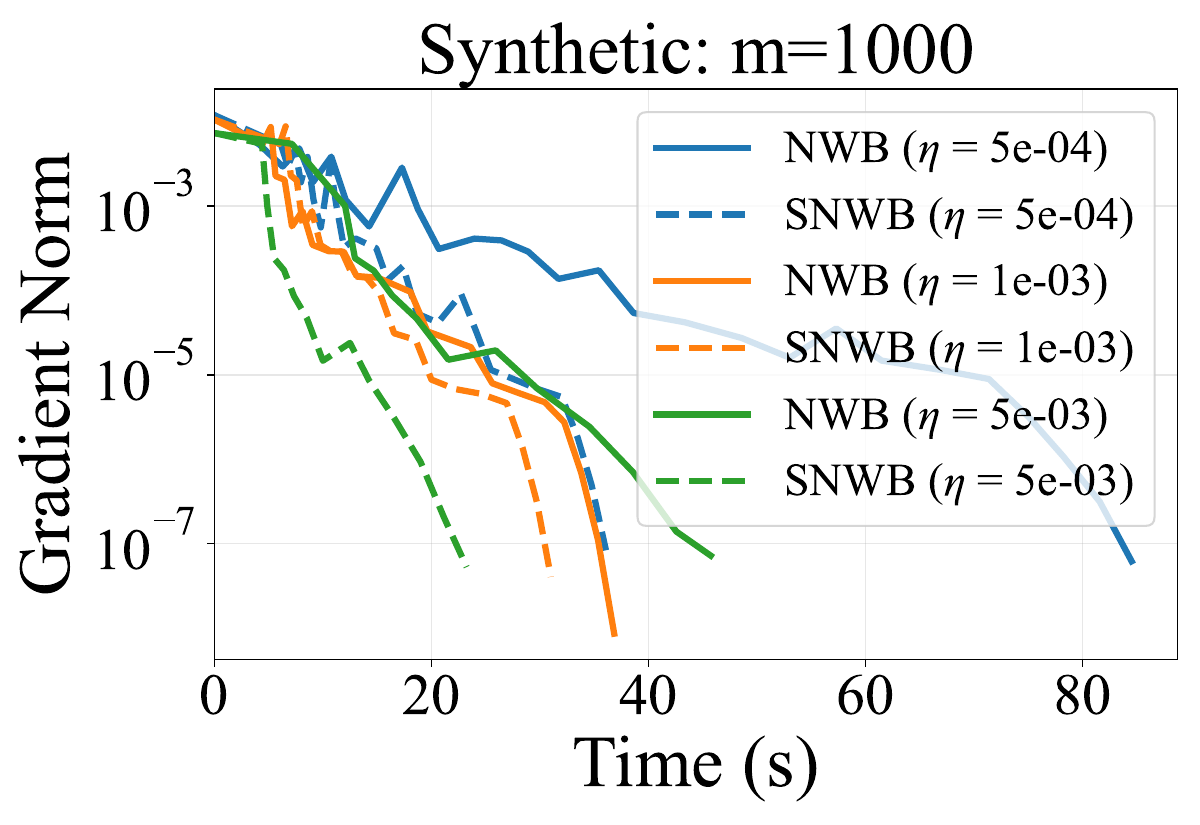}}}\hfill
{
\resizebox*{0.32 \textwidth}{!}{\includegraphics{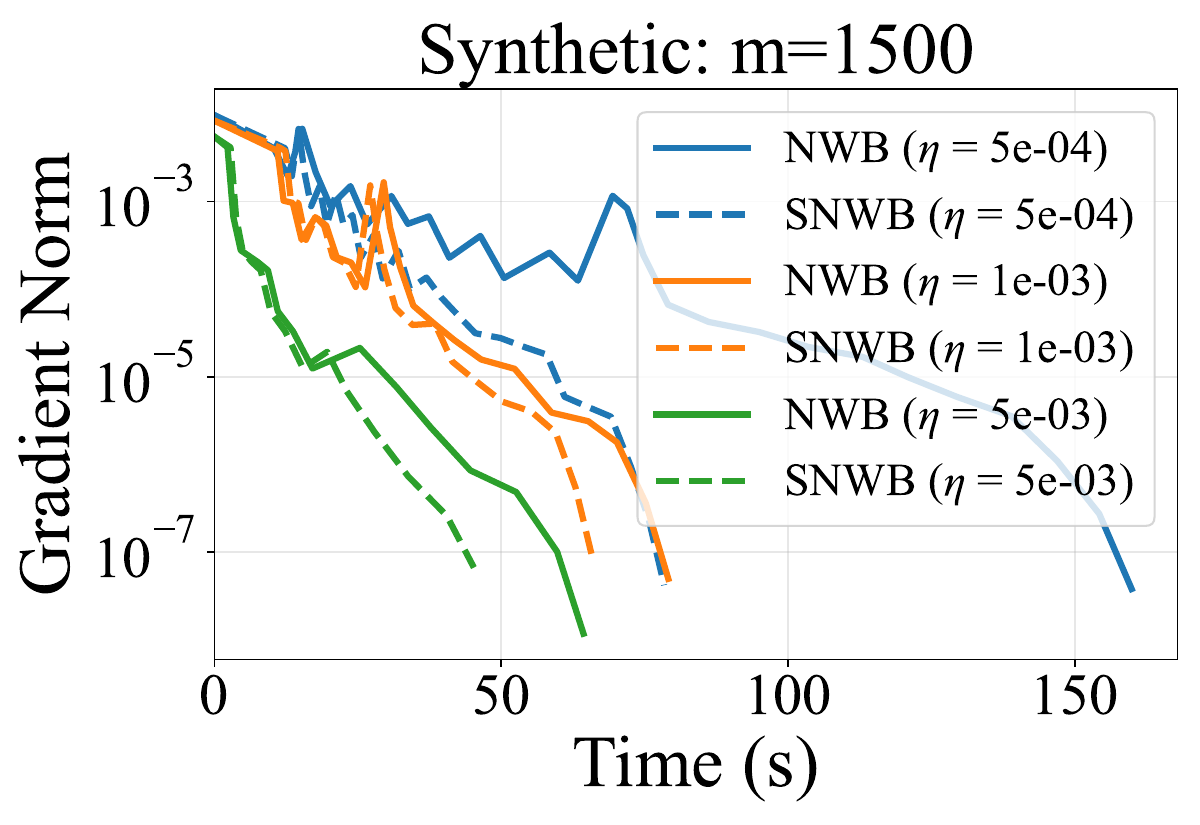}}}\hfill
{
\resizebox*{0.32 \textwidth}{!}{\includegraphics{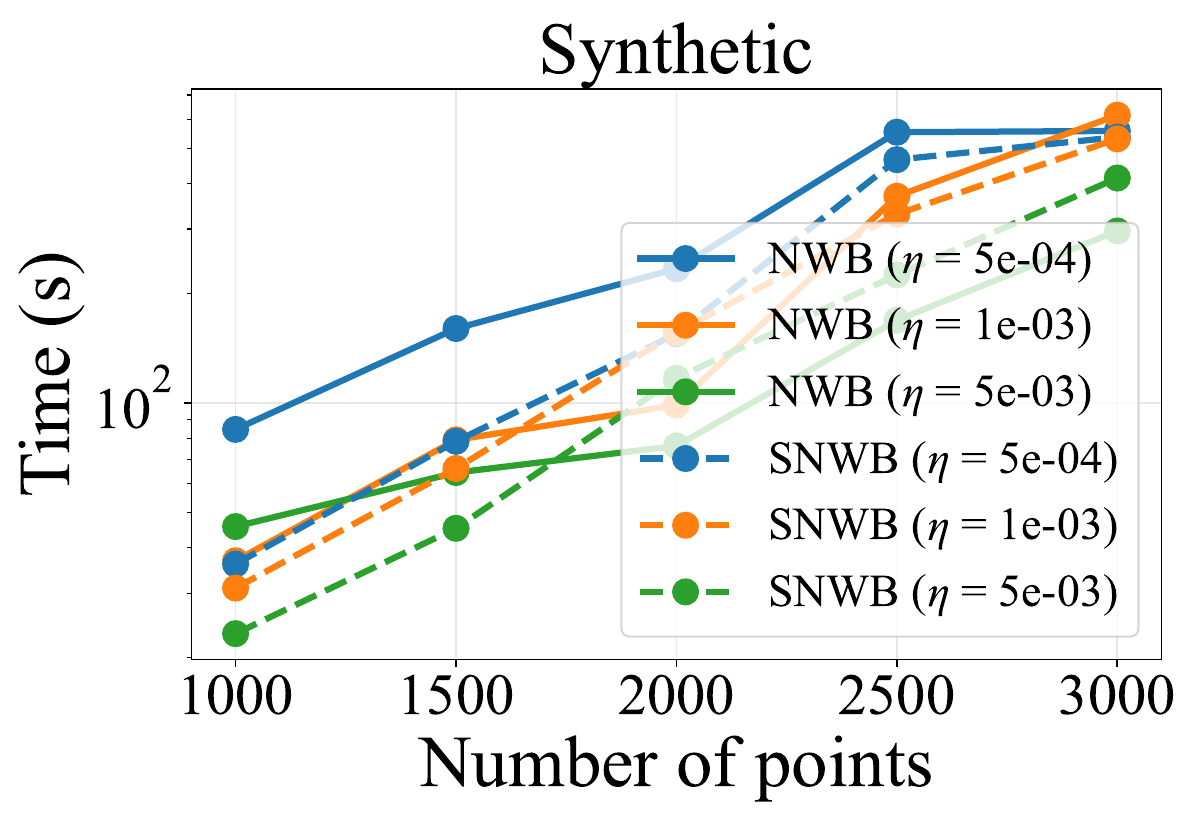}}}
{
\resizebox*{0.32 \textwidth}{!}{\includegraphics{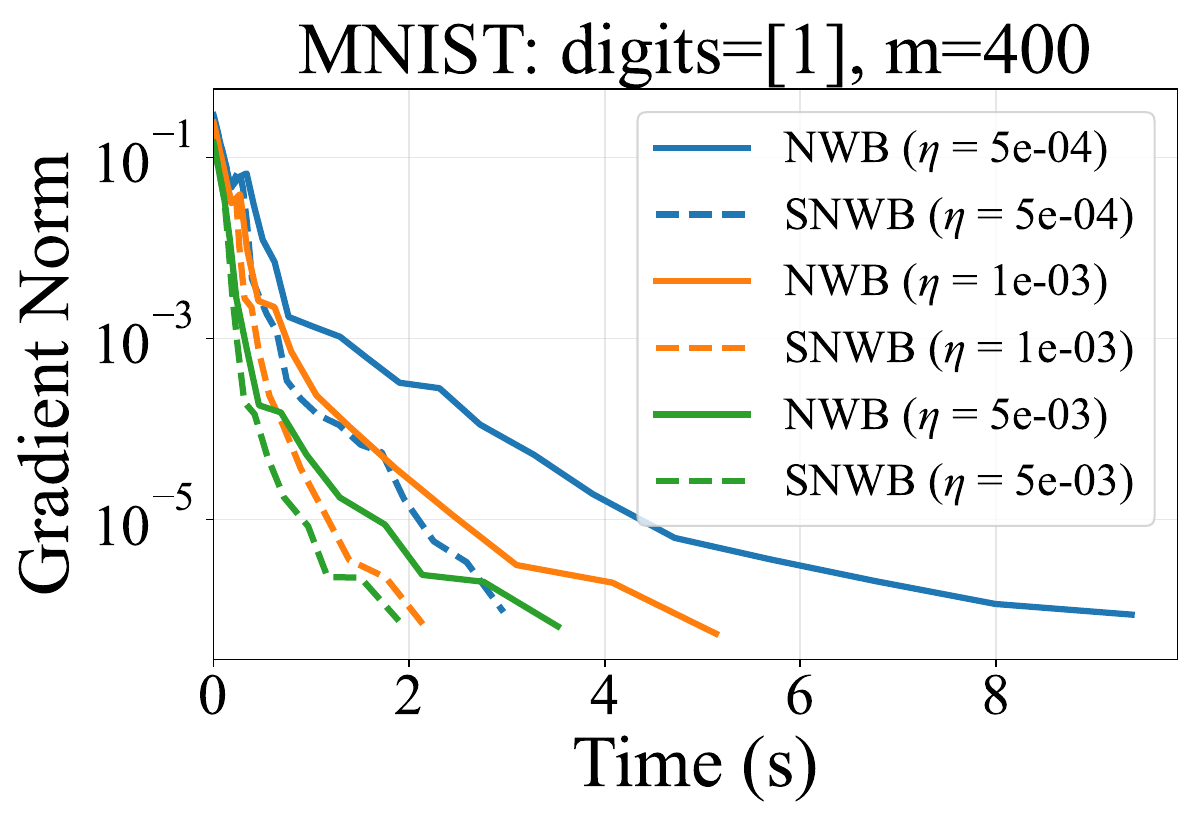}}}\hfill
{
\resizebox*{0.32 \textwidth}{!}{\includegraphics{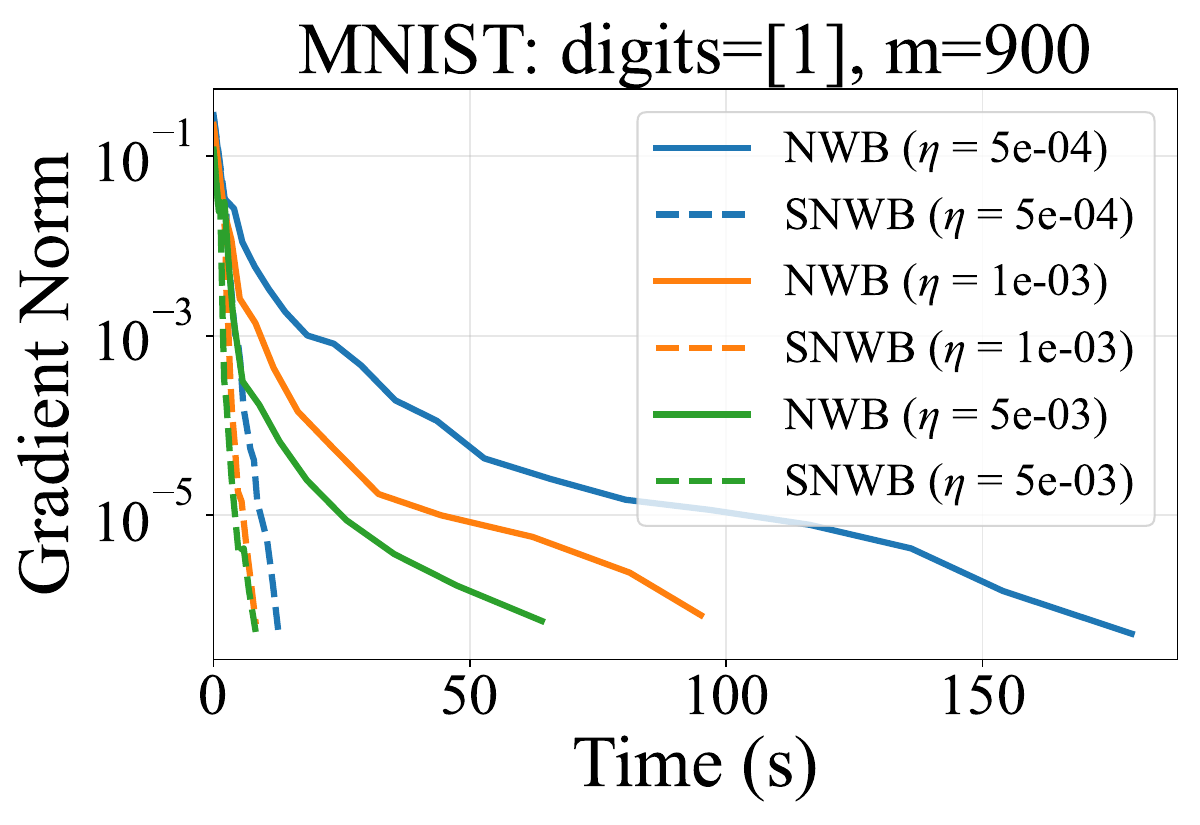}}}\hfill
{
\resizebox*{0.32 \textwidth}{!}{\includegraphics{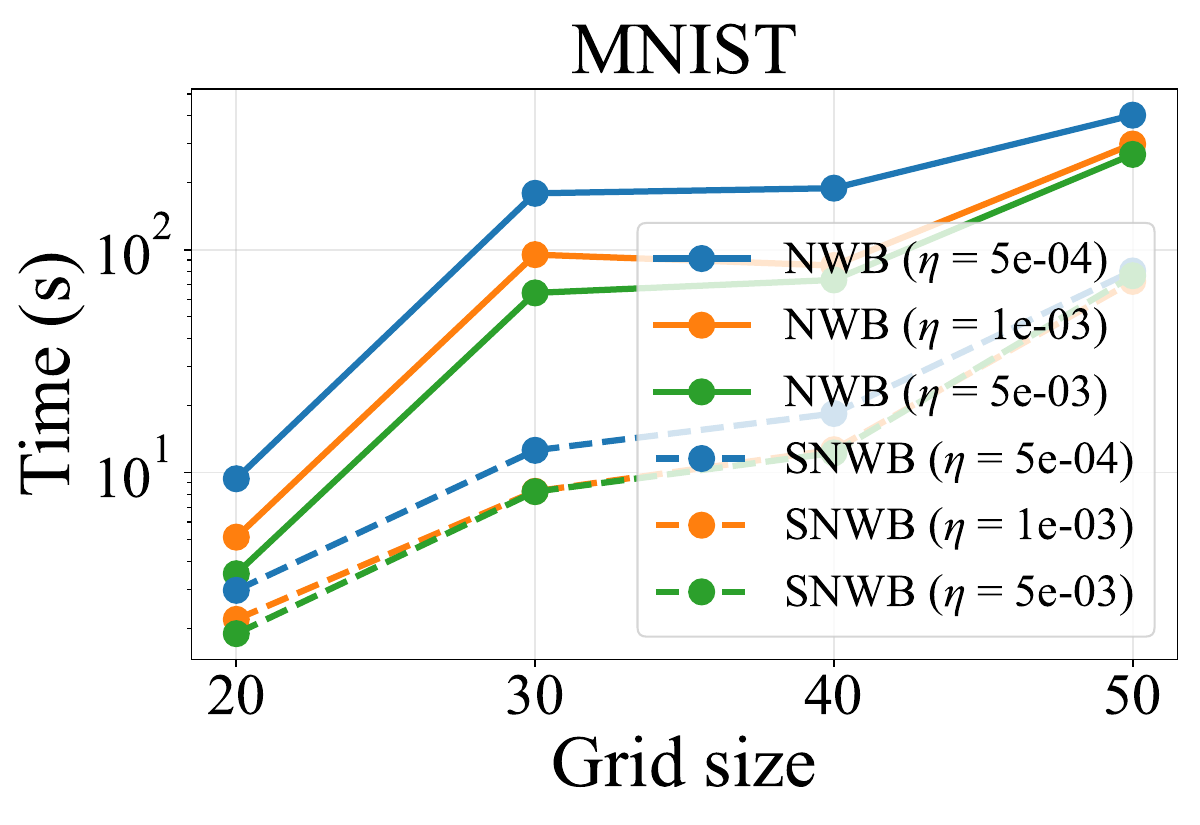}}}
\caption{Effect of the regularization parameter $\eta=\tau$ on NWB and SNWB. Top: synthetic dataset. Bottom: MNIST dataset. The left and middle panels show convergence in terms of the gradient norm over runtime for different values of $\eta$.  The right panels report the runtime required to reach $\|\vect{g}\|<10^{-7}$ as the problem size increases.  For MNIST, $m=(\text{grid size})^2$.
}
\label{fig: Impact of the regularization parameter}
\end{figure*}

\subsection{Ablation Study of Sparsity Level}
\label{app: Ablation study of Sparsity Level}

We further examine the effect of the sparsification level on the performance of SNWB. The experiments are conducted on both the synthetic instances described in Appendix~\ref{app: Additional Synthetic Experiments} and the MNIST dataset. The purpose of this study is to assess how the sparsification parameter affects the trade-off between computational efficiency and the accuracy of the sparse Hessian approximation.

Recall that in Algorithm~\ref{alg: SNWB}, the sparsification threshold is chosen as
$\rho=C_\rho\|\vect{g}\|$. We compare several choices of $C_\rho$, including
$C_\rho=0$, which corresponds to the dense NWB method, and two families of sparsification parameters:
\[
    C_\rho = S\cdot \frac{\eta}{nm},
    \qquad
    C_\rho = S\cdot \frac{\eta}{\sqrt{n}m},
    \qquad
    S\in\{10^3,10^4,10^5\}.
\]

Figure~\ref{fig: Performance comparison among different sparsification parameter} compares the convergence behavior under different sparsification levels. On MNIST, introducing sparsity leads to a substantial reduction in runtime, and the acceleration is robust across a wide range of $C_\rho$ values. This suggests that the transport probability matrices in these image-based barycenter problems are highly localized and can be accurately represented by sparse approximations. On the synthetic datasets, the advantage of sparsification is less pronounced, reflecting the denser structure of the corresponding transport probability matrices. Nevertheless, a suitable choice of the sparsification parameter, such as
$C_\rho=10^5\cdot \eta/(\sqrt{n}m)$, still achieves the fastest convergence in the reported settings.

Table~\ref{tab: Performance comparison among different sparsity parameter choices} provides a more detailed comparison of the optimization statistics. Across different sparsification levels, the average number of CG iterations, the total number of iterations, and the final gradient norms remain comparable to those of the dense baseline. This indicates that the sparse Hessian approximation preserves the essential curvature information needed for stable Newton updates. The differences in runtime are largely explained by the average proportion of nonzero entries, denoted by $\bar{p}_{nz}$. For MNIST, $\bar{p}_{nz}$ is approximately $0.01$, which leads to significant computational savings. In contrast, the synthetic datasets have a much larger proportion of nonzero entries, typically above $0.20$, and therefore yield more moderate speedups. These results show that the proposed sparsification strategy is particularly effective when the transport structure is naturally localized, while still remaining competitive in denser settings.

\begin{figure*}[t]
{
\resizebox*{0.32 \textwidth}{!}{\includegraphics{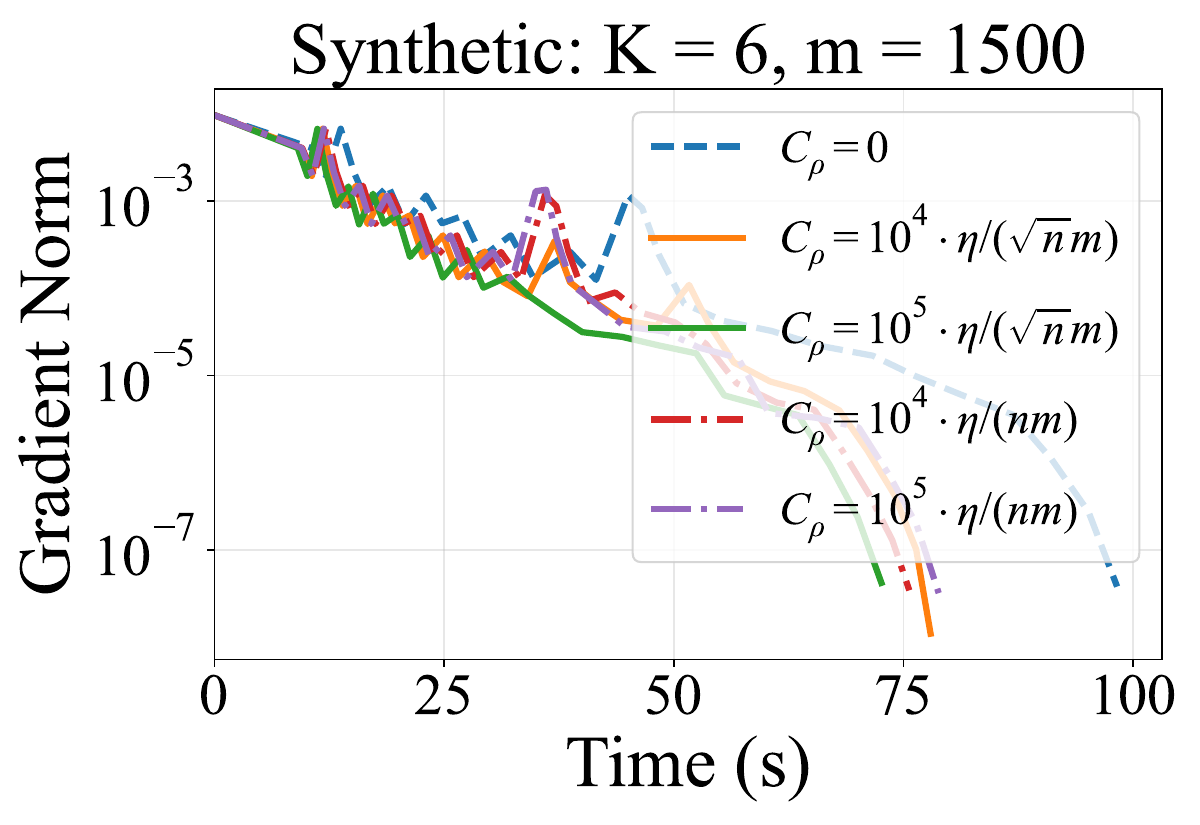}}}\hfill
{
\resizebox*{0.32 \textwidth}{!}{\includegraphics{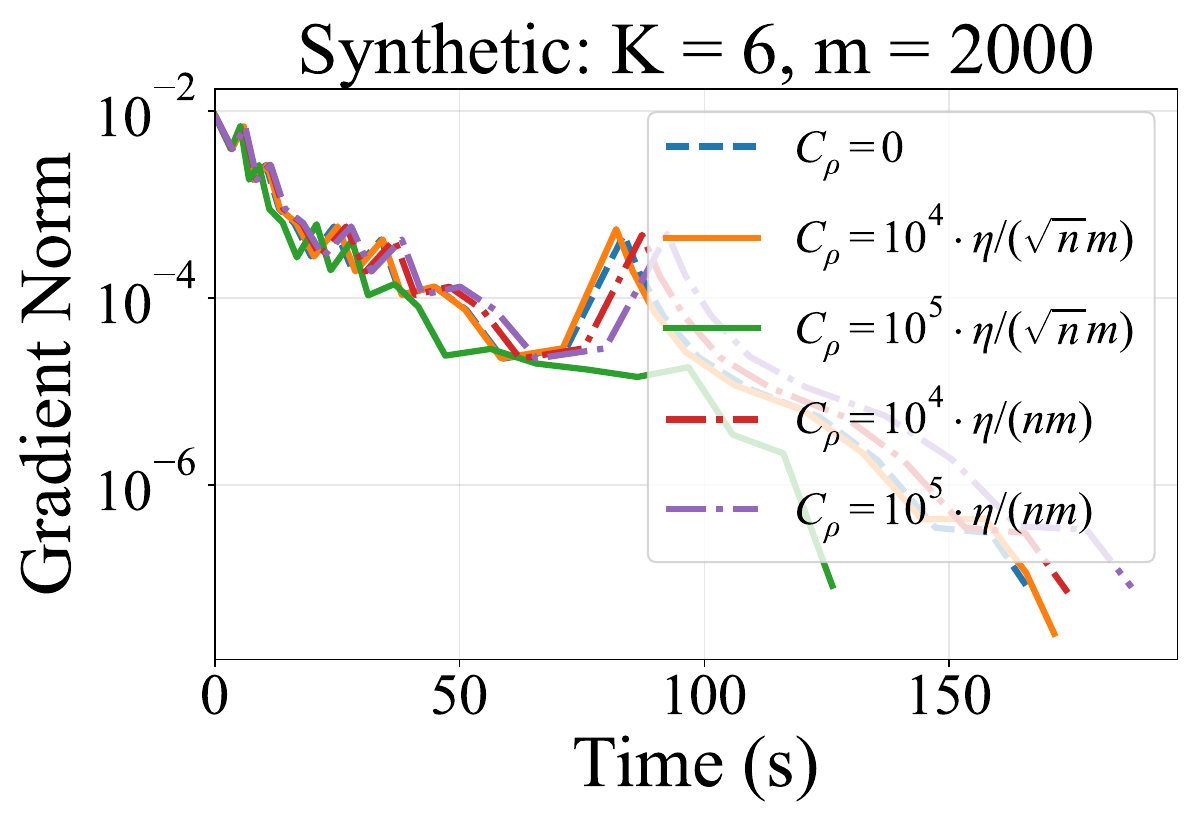}}}\hfill
{
\resizebox*{0.32 \textwidth}{!}{\includegraphics{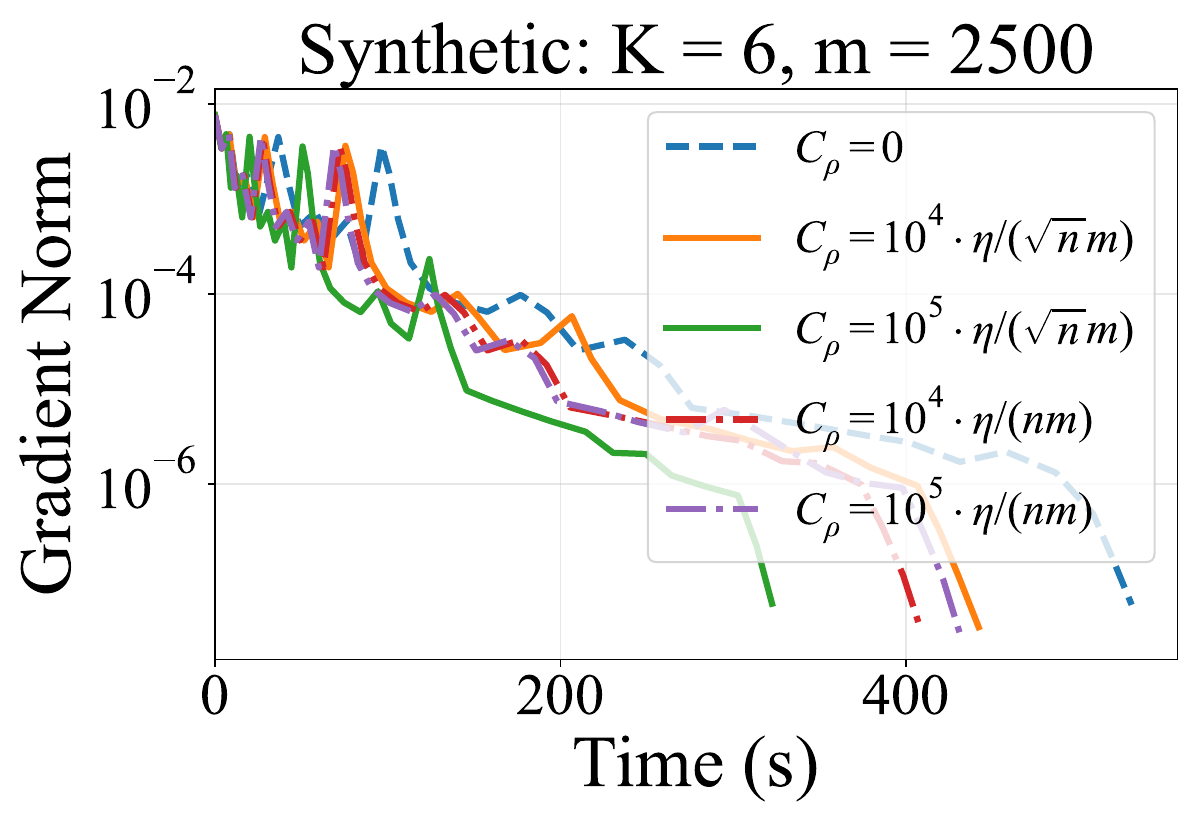}}}
{
\resizebox*{0.32 \textwidth}{!}{\includegraphics{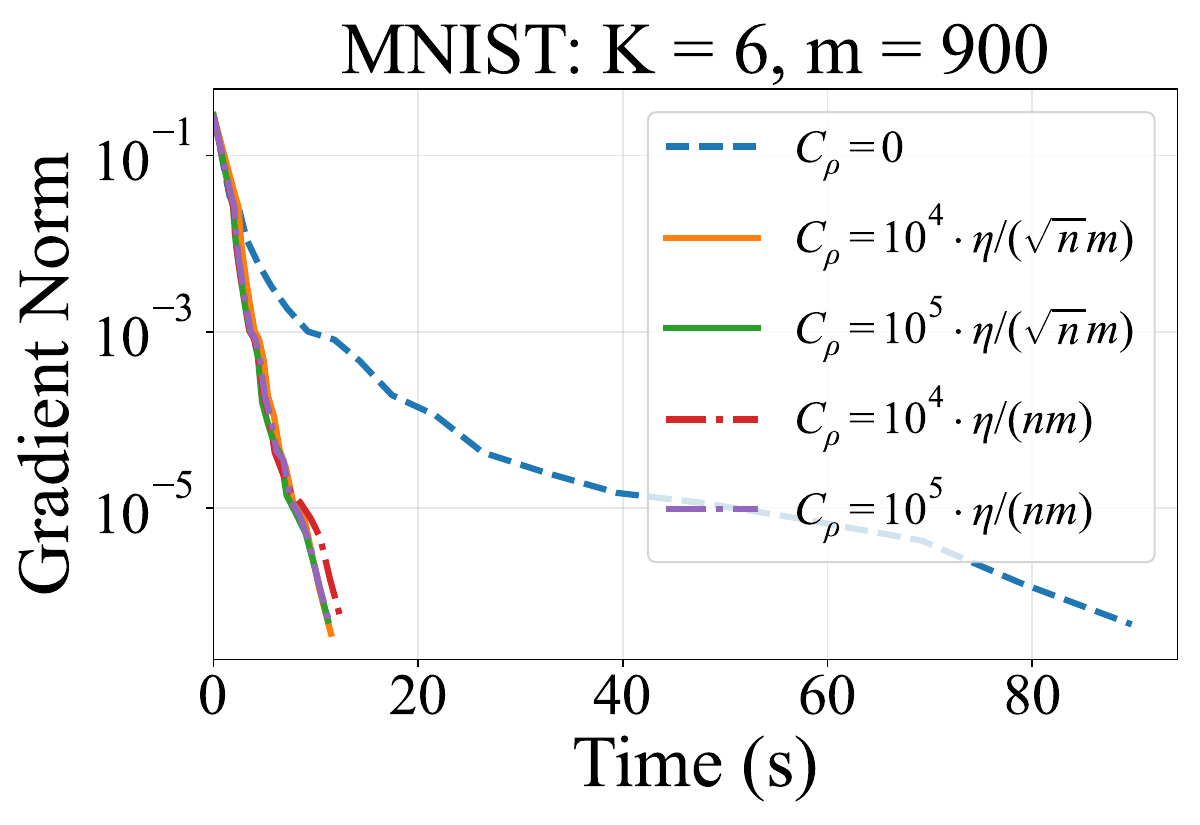}}}\hfill
{
\resizebox*{0.32 \textwidth}{!}{\includegraphics{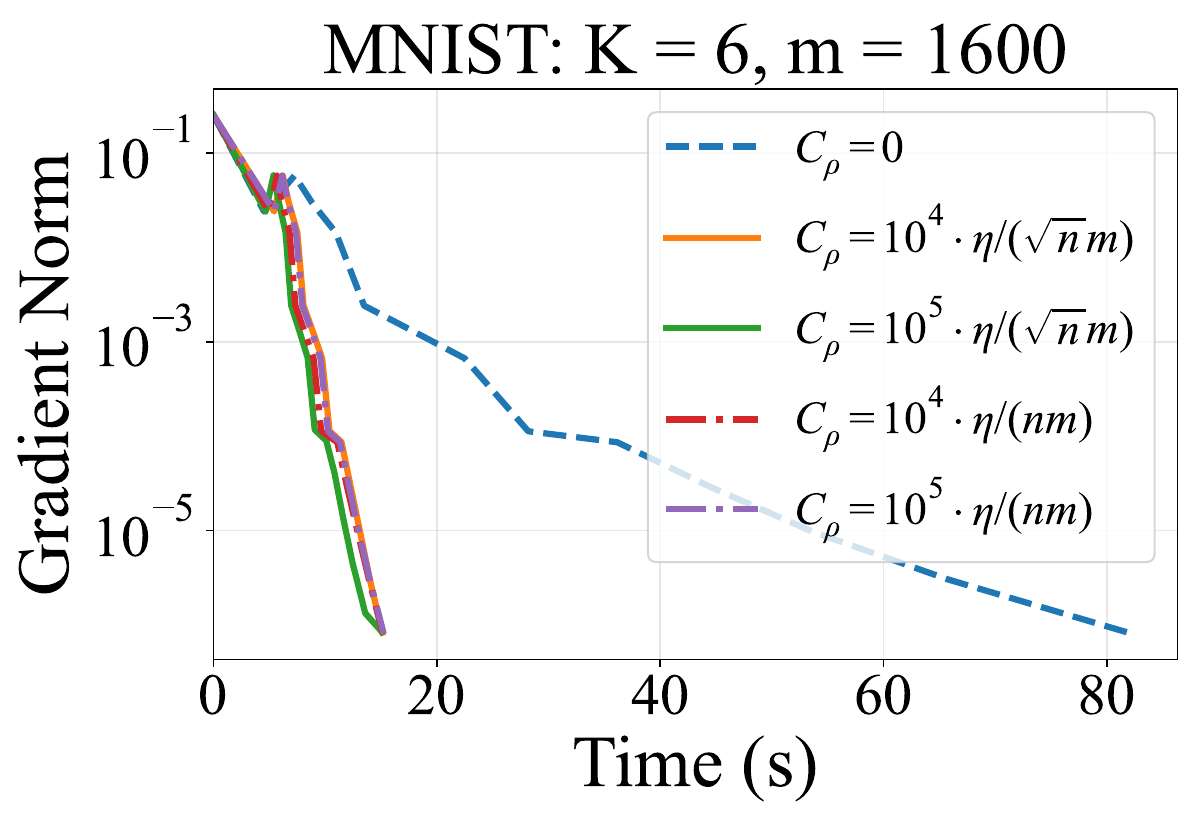}}}\hfill
{
\resizebox*{0.32 \textwidth}{!}{\includegraphics{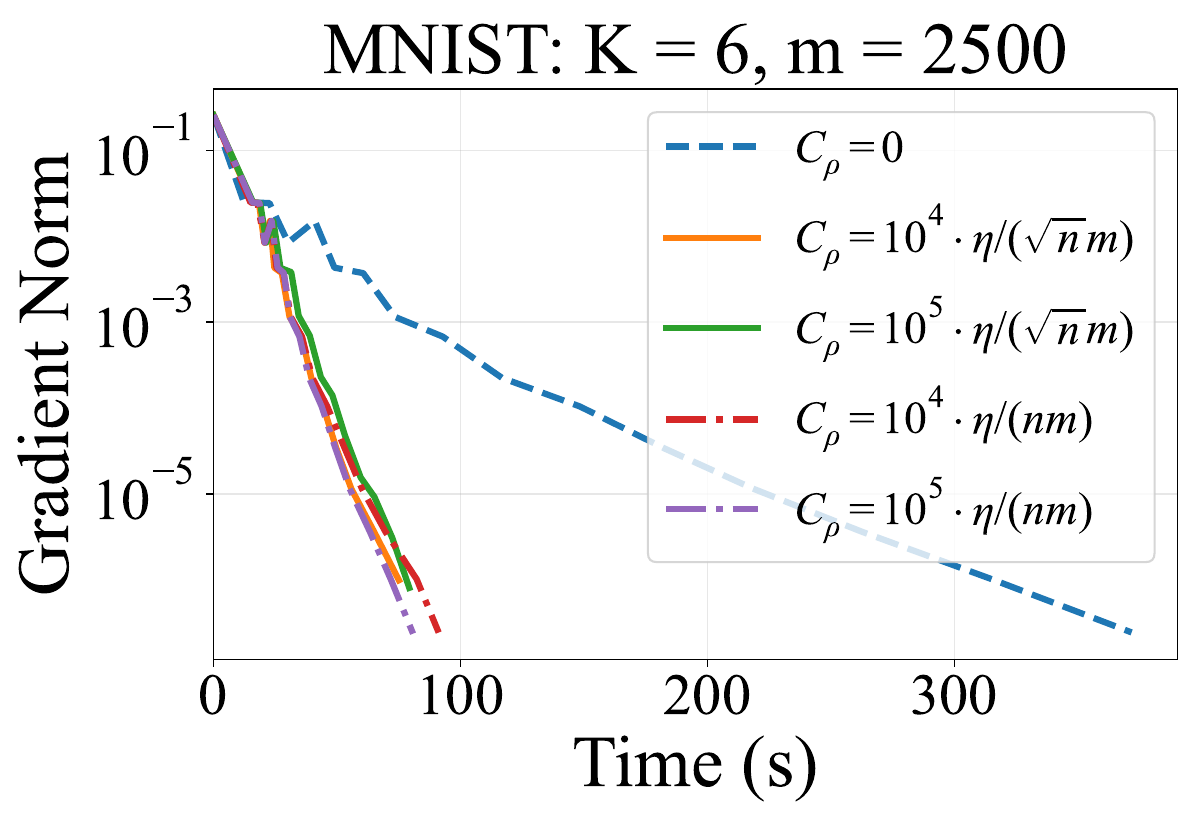}}}
\caption{Effect of the sparsification parameter $C_\rho$ on SNWB under $\eta=\tau=5\times10^{-4}$ on the synthetic dataset (Top) and the MNIST dataset (Bottom). Here, $C_\rho=0$ corresponds to the dense NWB method. }
\label{fig: Performance comparison among different sparsification parameter}
\end{figure*}

\begin{table}[!t]
\centering
\caption{Performance comparison under different sparsification parameter choices with $\eta=\tau=5\times10^{-4}$.
}
\label{tab: Performance comparison among different sparsity parameter choices}
\resizebox{\textwidth}{!}{
\begin{threeparttable}
\begin{tabular}{c|c||ccccc|ccccc}
    \toprule[2pt]
    \multicolumn{2}{c}{} & \multicolumn{10}{c}{Synthetic: $K=6$} \\ 
    \cmidrule{1-12}
    \multicolumn{2}{c||}{} & \multicolumn{5}{c|}{Size $m=1000$} & \multicolumn{5}{c}{Size $m=2000$} \\
    \cmidrule{1-12}
    Sparsification & $S$ & AveCG & $\Bar{p}_{nz}$ & Ite & GradNorm & Time (s) & AveCG & $\Bar{p}_{nz}$ & Ite & GradNorm & Time (s) \\
    \cmidrule{1-12}
    $C_\rho = 0$ & - & 422 & - & 26 & 6.18E-08 & 55.25 & 441	& -	& 24 & 7.08E-08	& 167.34 \\
    \cmidrule{1-12}
    \multirow{3}{*}{$C_\rho = S\frac{\eta}{nm}$} & $10^3$ & 421 & 0.25 & 26 & 6.19E-08 & 38.89 & 441 & 0.25 & 24 & 7.09E-08 & 172.15 \\
    & $10^4$ & 421 & 0.24 & 26 & 7.48E-08 & 37.68 & 440 & 0.25 & 24 & 7.11E-08 & 174.88 \\
    & $10^5$ & 420 & 0.23 & 26 & 5.02E-08 & 35.42 & 443 & 0.24 & 24 & 8.1E-08 & 187.89 \\
    \cmidrule{1-12}
    \multirow{3}{*}{$C_\rho = S\frac{\eta}{\sqrt{n}m}$} & $10^3$ & 419 & 0.23 & 26 & 6.37E-08 & 38.11 & 441	& 0.24 & 24	& 7.22E-08 & 164.74 \\
    & $10^4$ & 420 & 0.22 & 27& 5.67E-08 & 38.06 & 437 & 0.23 & 25 & 2.59E-08 & 172.08\\
    & $10^5$ & 432 & 0.20 & 25 & 8.34E-08& \textbf{32.77} & 476	& 0.21 & 21 & 8.43E-08 & \textbf{126.71}\\
    \cmidrule{1-12}
    \multicolumn{2}{c}{} & \multicolumn{10}{c}{MNIST: $K=6$} \\ 
    \cmidrule{1-12}
    \multicolumn{2}{c||}{} & \multicolumn{5}{c|}{Size $m=400$} & \multicolumn{5}{c}{Size $m=1600$} \\
    \cmidrule{1-12}
    Sparsification & $S$ & AveCG & $\Bar{p}_{nz}$ & Ite & GradNorm & Time (s) & AveCG & $\Bar{p}_{nz}$ & Ite & GradNorm & Time (s) \\
    \cmidrule{1-12}
    $C_\rho = 0$ & - & 1220 & - & 19 & 8.9E-07 & 9.42 & 776 & - & 12 & 8.1E-07 & 82.54\\
    \cmidrule{1-12}
    \multirow{3}{*}{$C_\rho = S\frac{\eta}{nm}$} & $10^3$ & 1257 & 0.02 & 21 & 8.4E-07 & 4.38 & 778 & 0.01 & 12 & 8.1E-07 & \textbf{14.31} \\
    & $10^4$ & 1085 & 0.01 & 19 & 6.08E-07 & 3.44 & 776 & 0.01 & 12 & 8.1E-07 & 15.51 \\
    & $10^5$ & 1099 & 0.01 & 19 & 7.67E-07 & 3.42 & 777 & 0.01 & 12 & 8.14E-07 & 15.70 \\
    \cmidrule{1-12}
    \multirow{3}{*}{$C_\rho = S\frac{\eta}{\sqrt{n}m}$} & $10^3$ & 1326 & 0.01 & 20 & 6.87E-07 & 4.30 & 778 & 0.01 & 12 & 8.11E-07 & 15.17 \\
    & $10^4$ & 1252 & 0.01 & 20 & 7.77E-07 & 4.05 & 780 & 0.01 & 12 & 8.3E-07 & 15.47 \\
    & $10^5$ & 1001 & 0.01 & 18 & 9.49E-07 & \textbf{2.95} & 916 & 0.01 & 13 & 8.47E-07 & 15.41 \\
    \bottomrule[2pt]
\end{tabular}
\begin{tablenotes}
    \footnotesize               
    \item \normalsize Bold values indicate the most favorable results. ``AveCG'' denotes the average number of CG iterations in finding a Newton direction, ``$\Bar{p}_{nz}$'' denotes the average proportion of nonzeros in $\{P_k\}_{k=1}^K$, ``Ite'' denotes the total number of iterations and ``GradNorm'' denotes the final gradient norm $\|\vect{g}\|$. 
\end{tablenotes}
\end{threeparttable}
}
\end{table}

\subsection{More Experiments under $\tau = \eta/2$}
\label{app: more_experiments_tau_eta_over_2}

We also evaluate the proposed methods under the regularization setting $\tau=\eta/2$, which has been shown to provide improved approximation properties for smooth densities~\cite{chizat2025doubly,vaskevicius2023computational}. In this setting, we compare NWB and SNWB with EDSWB~\cite{vaskevicius2023computational}, a representative first-order solver designed for this choice of regularization parameters.

Following the experimental protocols in Sections~\ref{subsec: Synthetic Data} and~\ref{subsec: Real Data}, we repeat the nested ellipses and MNIST experiments with $\tau=\eta/2$. The results are shown in Figures~\ref{fig: Nested ellipse benchmark under 1 over 2.} and~\ref{fig:fashionmnist_and_mnist_performance_evaluation_1_over_2}. Across both synthetic and real datasets, SNWB achieves the fastest convergence among the compared methods. In the nested ellipses benchmark, SNWB produces barycenters with visual quality comparable to NWB while requiring substantially less runtime. On MNIST, SNWB reaches the high-accuracy reference solution much faster than both NWB and EDSWB. These results further confirm that the proposed sparse Newton strategy remains effective beyond the setting $\tau=\eta$.

\begin{figure*}[!ht]
\centering
\setlength{\fboxrule}{0.8pt}
\setlength{\fboxsep}{4pt}

\fbox{%
\begin{minipage}{0.46\textwidth}
\centering
\includegraphics[width=0.32\textwidth]{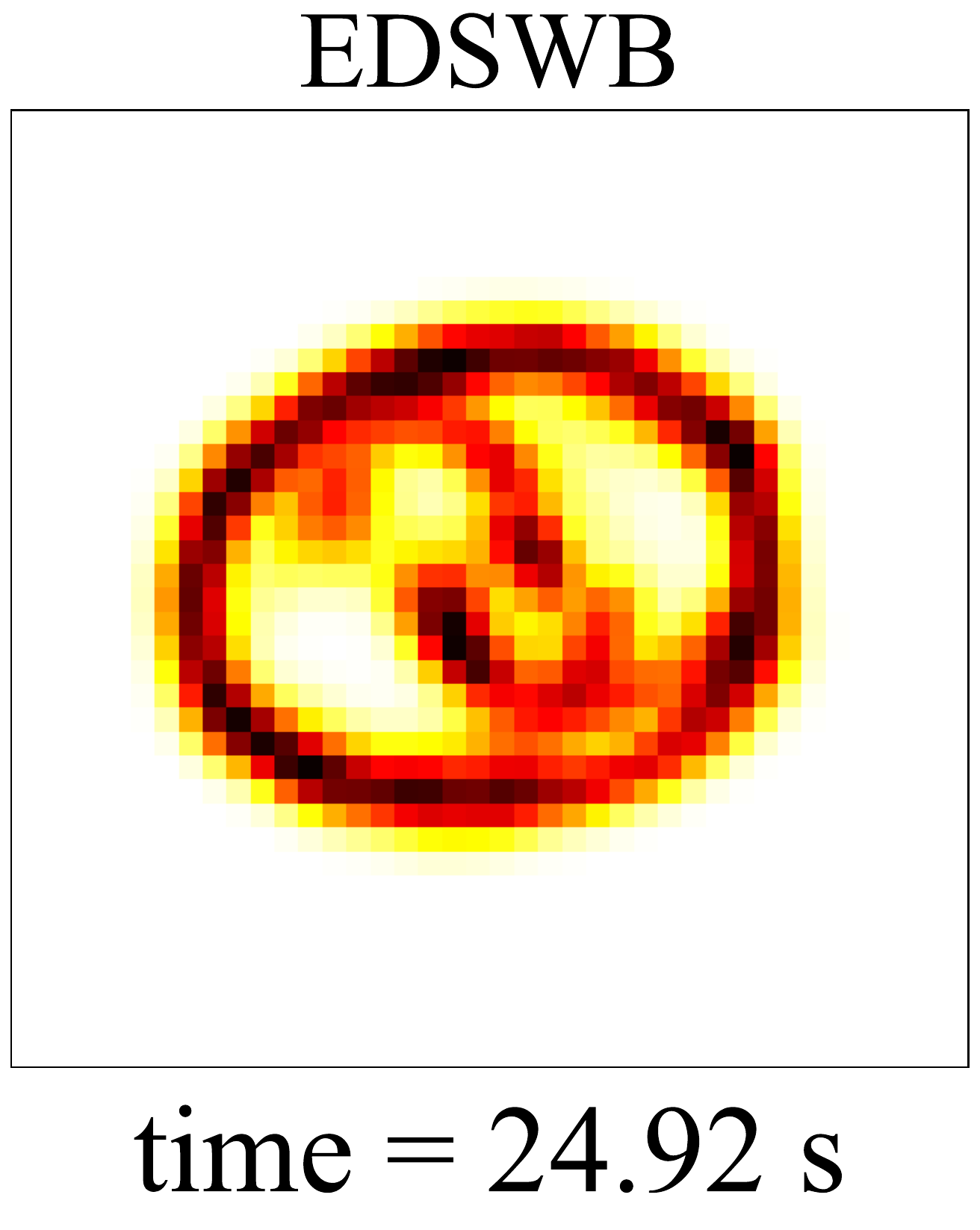}\hfill
\includegraphics[width=0.32\textwidth]{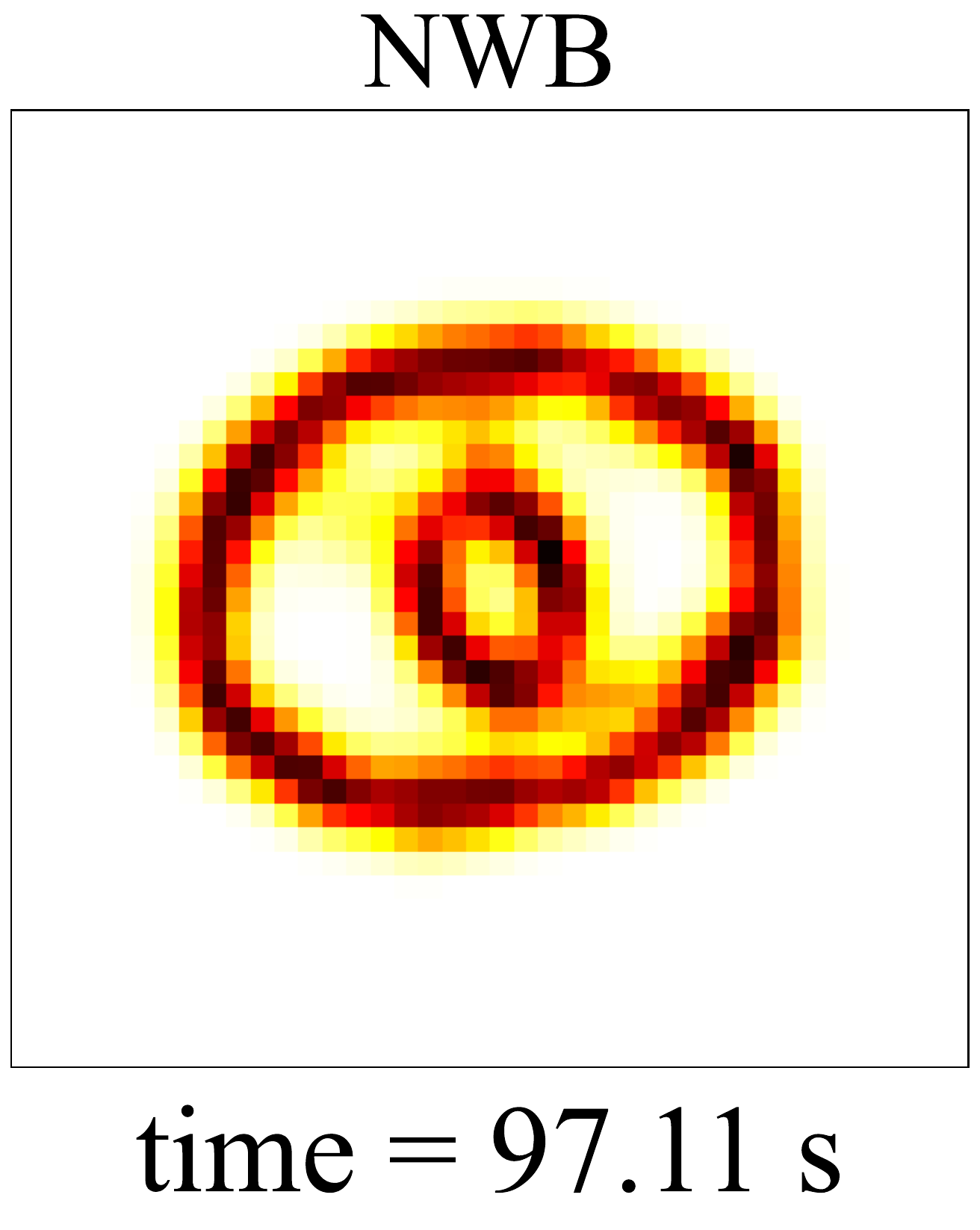}\hfill
\includegraphics[width=0.32\textwidth]{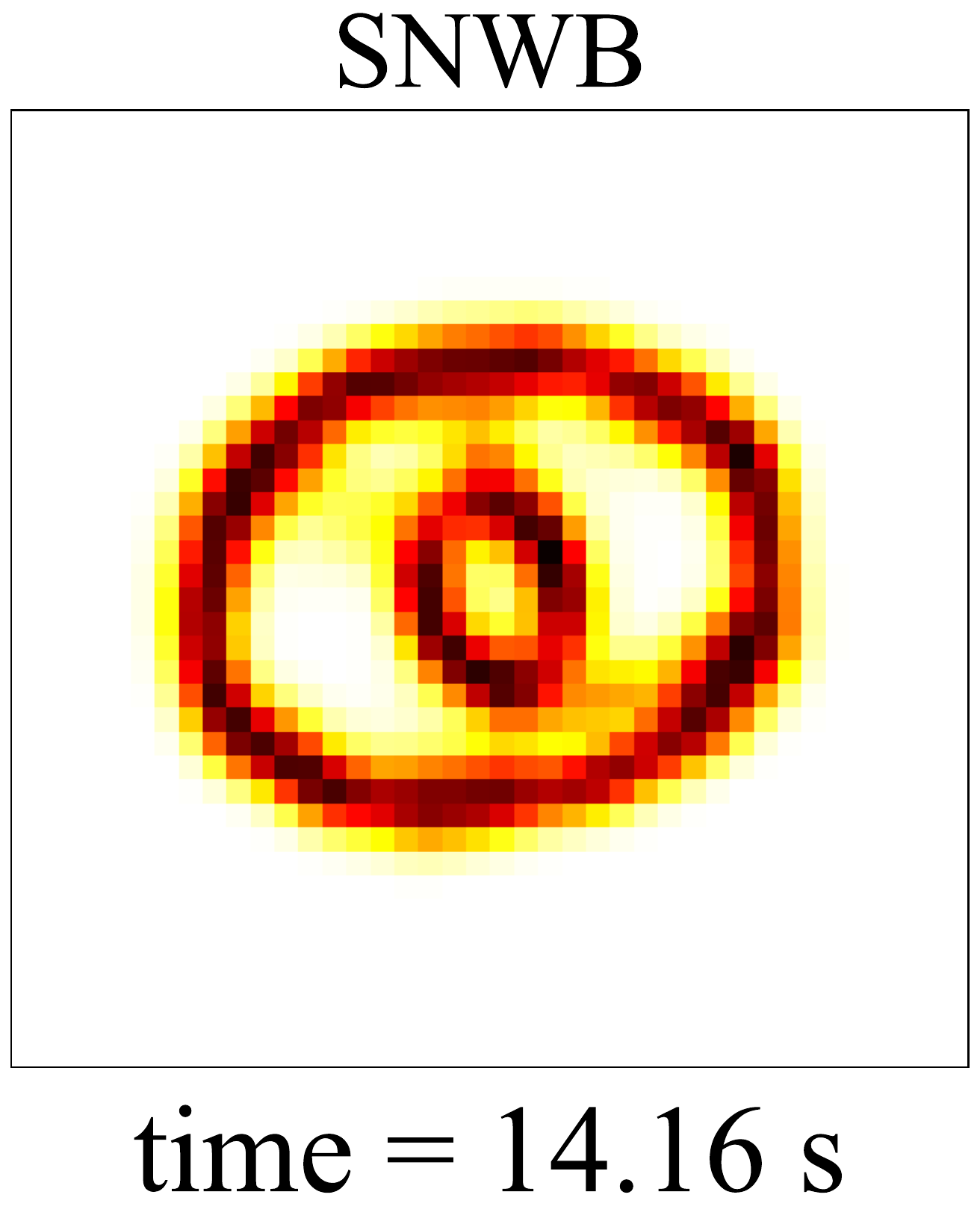}

\vspace{0.3em}
\small{(a) $40 \times 40$ pixel grids}
\end{minipage}
}
\hfill
\fbox{%
\begin{minipage}{0.46\textwidth}
\centering
\includegraphics[width=0.32\textwidth]{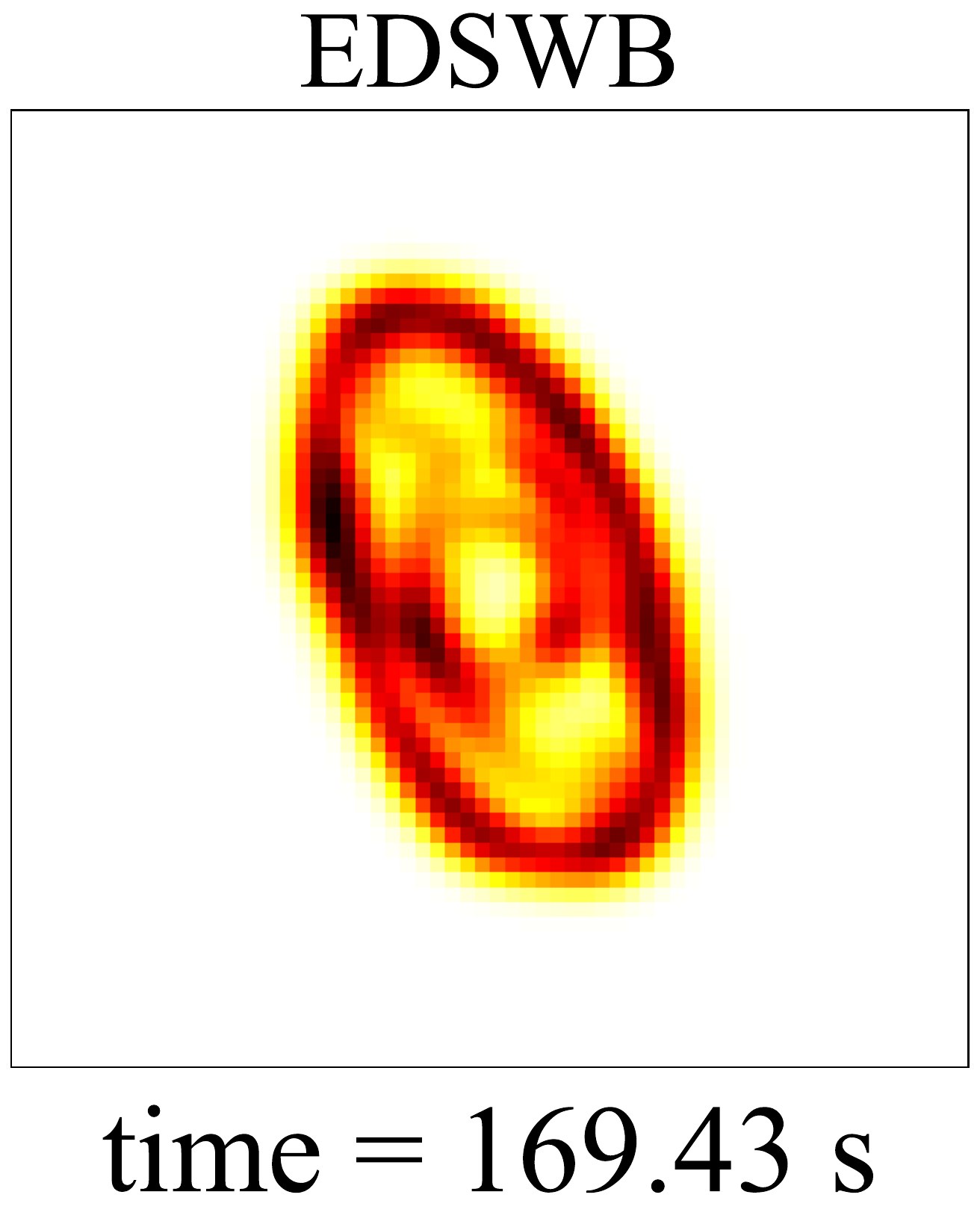}\hfill
\includegraphics[width=0.32\textwidth]{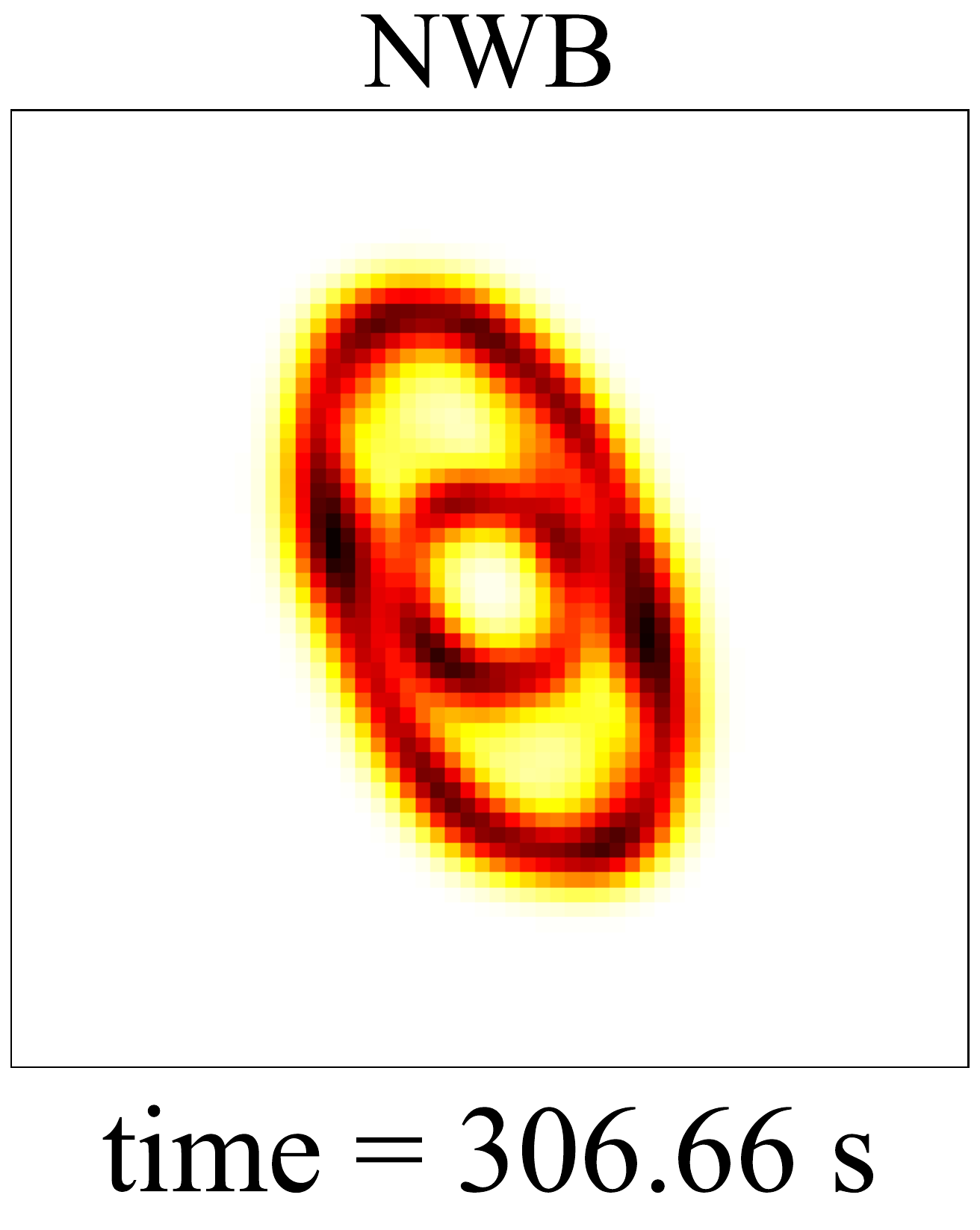}\hfill
\includegraphics[width=0.32\textwidth]{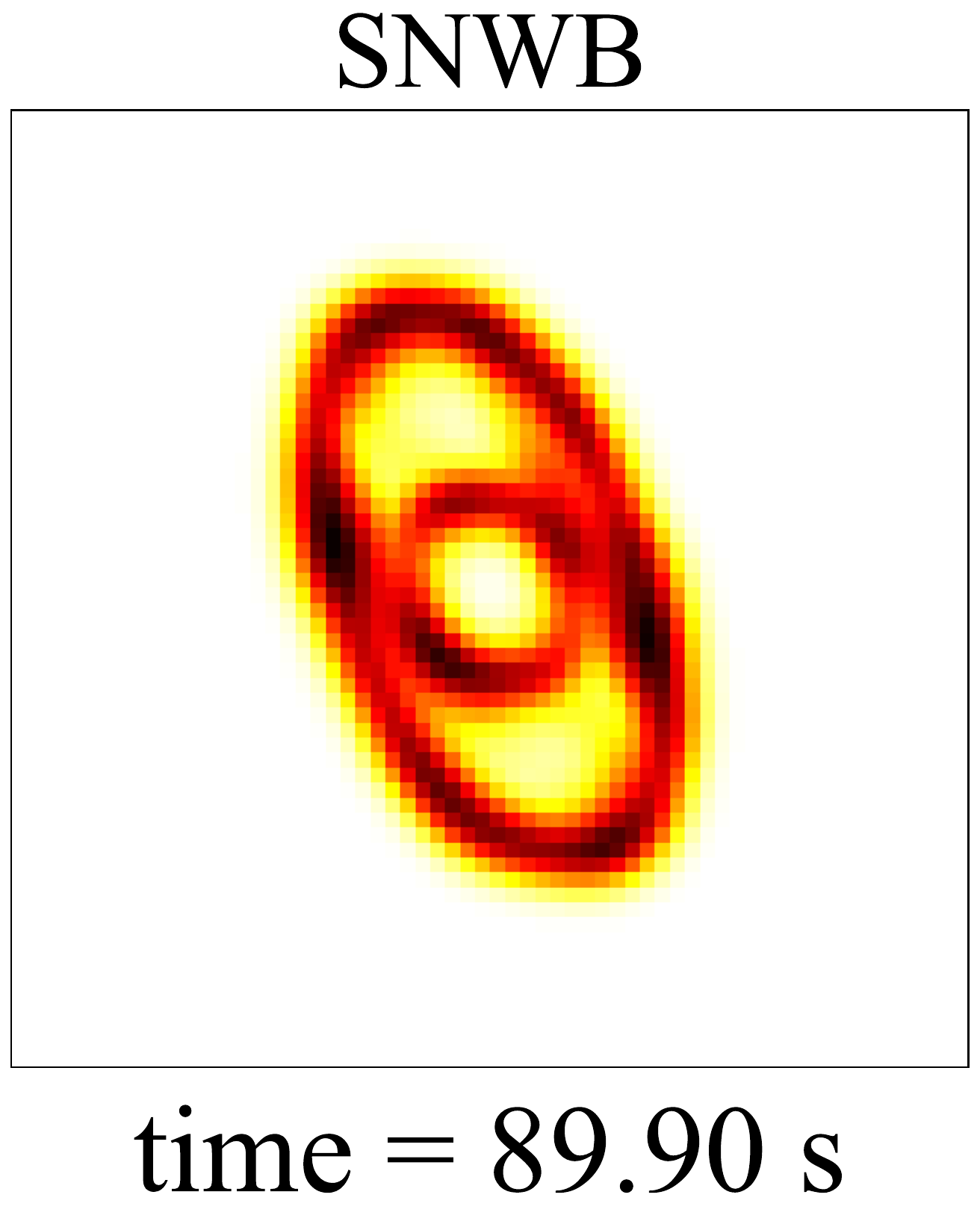}

\vspace{0.3em}
\small{(b) $64 \times 64$ pixel grids}
\end{minipage}
}
\caption{Nested ellipse benchmark under $\tau = \eta/2 = 5 \times 10^{-4}$ with $40 \times 40$ pixel grids (a) and $64 \times 64$ pixel grids (b). The input measures are the same as those used in Figures~\ref{fig: Nested ellipse benchmark. Top: Dataset of 5 nested ellipses with 40 pixel grids. Bottom: Barycenters computed by different methods.} and~\ref{fig: Nested ellipse benchmark. Top: Dataset of 5 nested ellipses with 64 pixel grids. Bottom: Barycenters computed by different methods}, respectively. }
\label{fig: Nested ellipse benchmark under 1 over 2.}
\end{figure*}

\begin{figure*}[t]
\centering

\begin{minipage}{0.23\textwidth}
\centering
\includegraphics[width=\linewidth, trim=10 10 10 10, clip]{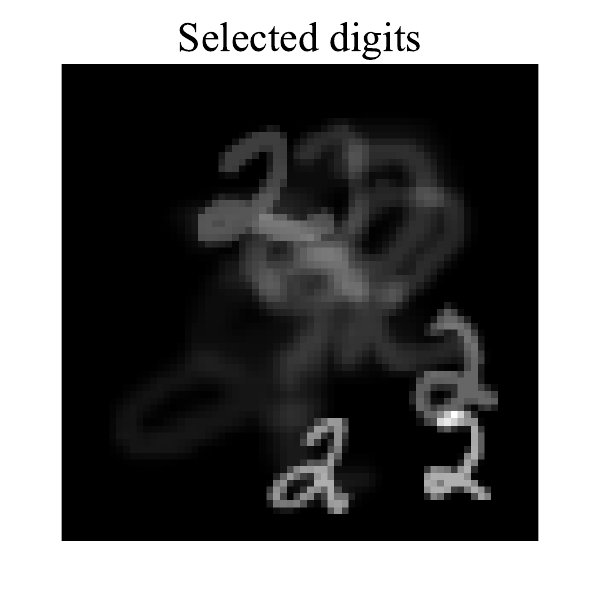}
\vspace{-0.2em}
\includegraphics[width=\linewidth, trim=10 10 10 10, clip]{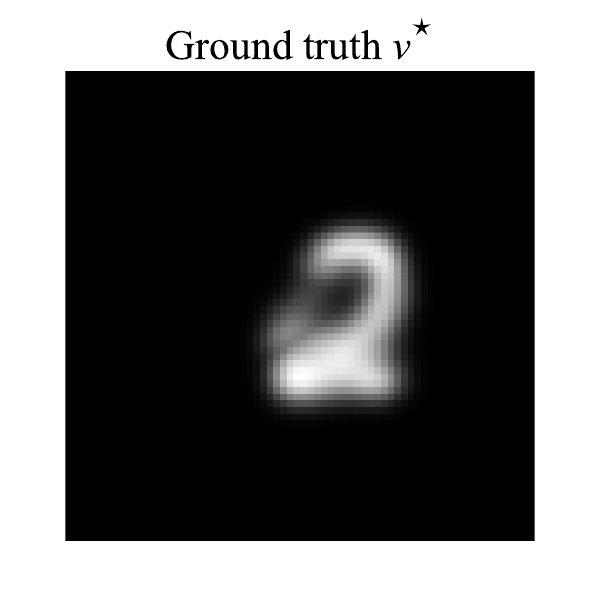}
\vspace{-0.2em}
\includegraphics[width=\linewidth, trim=10 10 10 10, clip]{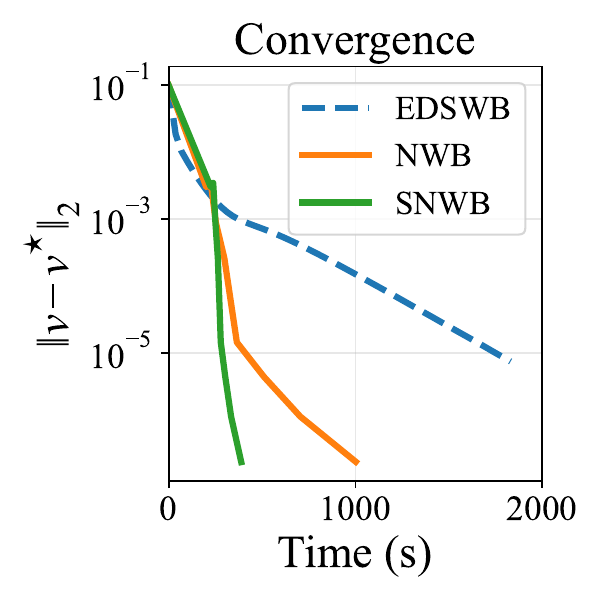}
\end{minipage}%
\hfill
\begin{minipage}{0.23\textwidth}
\centering
\includegraphics[width=\linewidth, trim=10 10 10 10, clip]{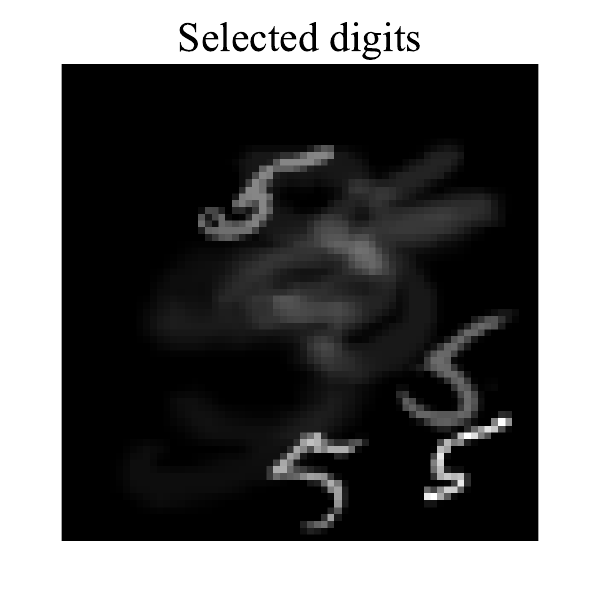}
\vspace{-0.2em}
\includegraphics[width=\linewidth, trim=10 10 10 10, clip]{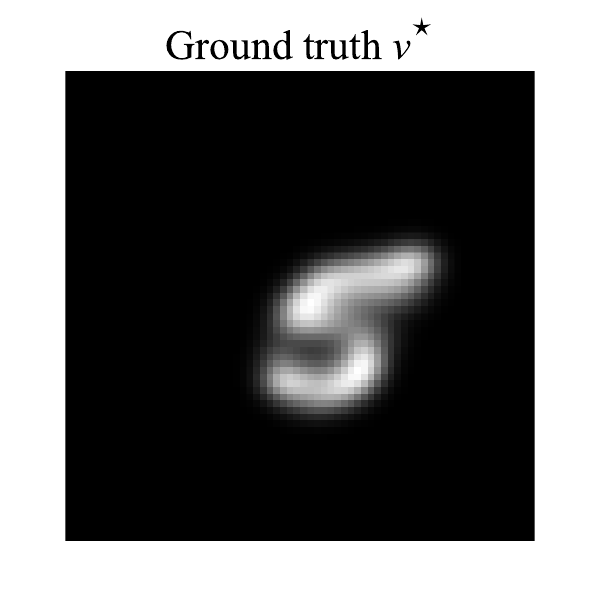}
\vspace{-0.2em}
\includegraphics[width=\linewidth, trim=10 10 10 10, clip]{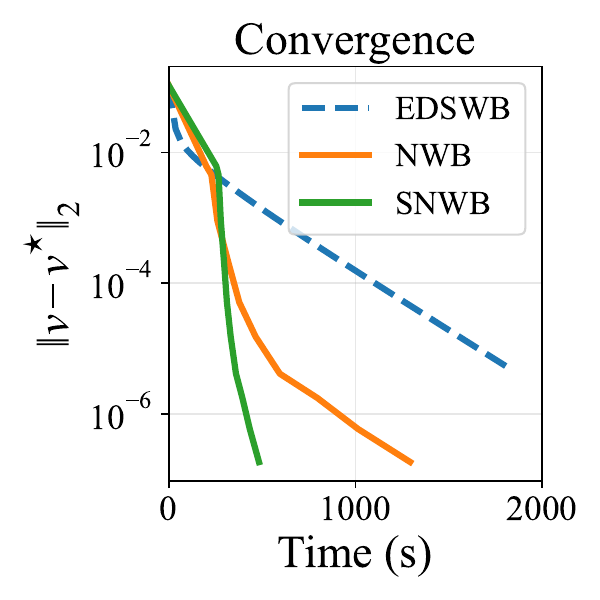}
\end{minipage}%
\hfill
\begin{minipage}{0.23\textwidth}
\centering
\includegraphics[width=\linewidth, trim=10 10 10 10, clip]{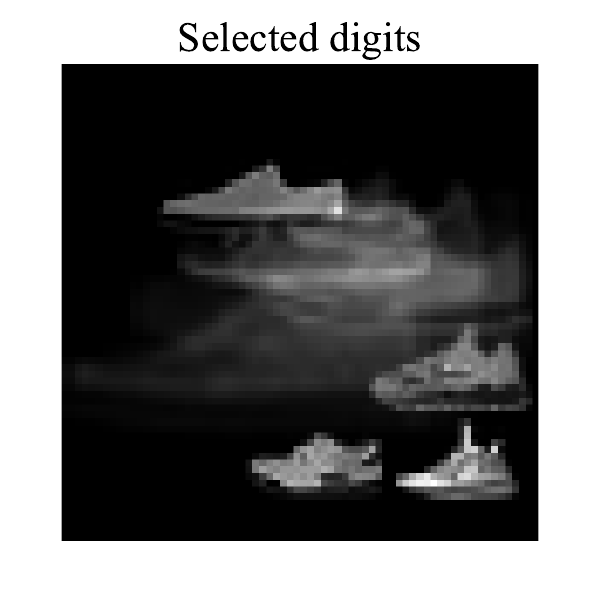}
\vspace{-0.2em}
\includegraphics[width=\linewidth, trim=10 10 10 10, clip]{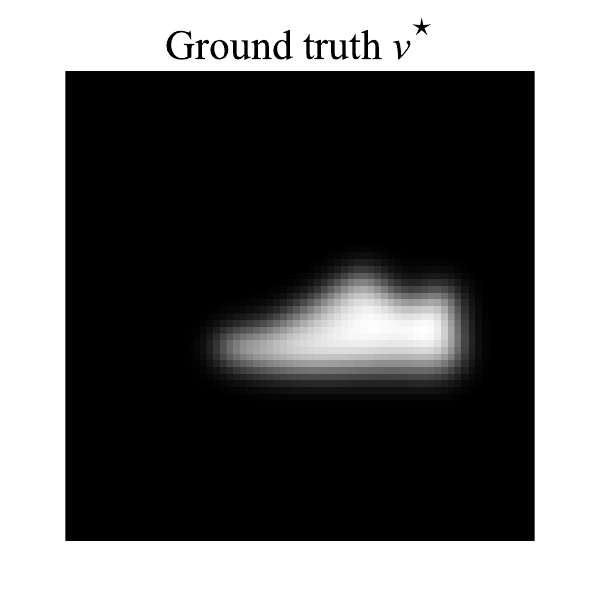}
\vspace{-0.2em}
\includegraphics[width=\linewidth, trim=10 10 10 10, clip]{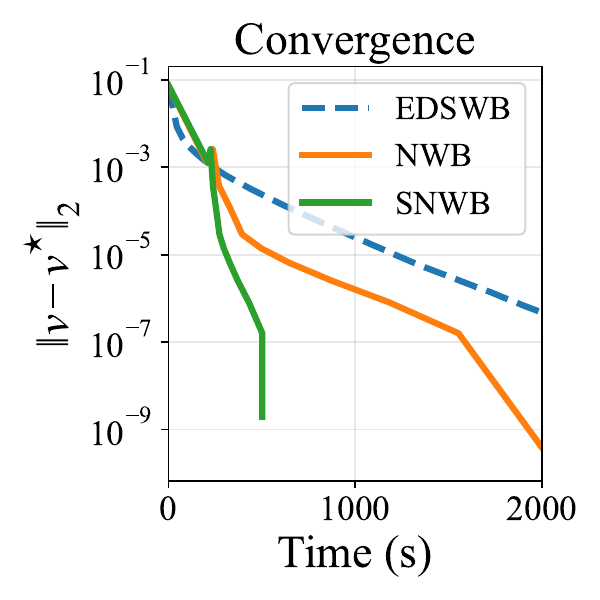}
\end{minipage}%
\hfill
\begin{minipage}{0.23\textwidth}
\centering
\includegraphics[width=\linewidth, trim=10 10 10 10, clip]{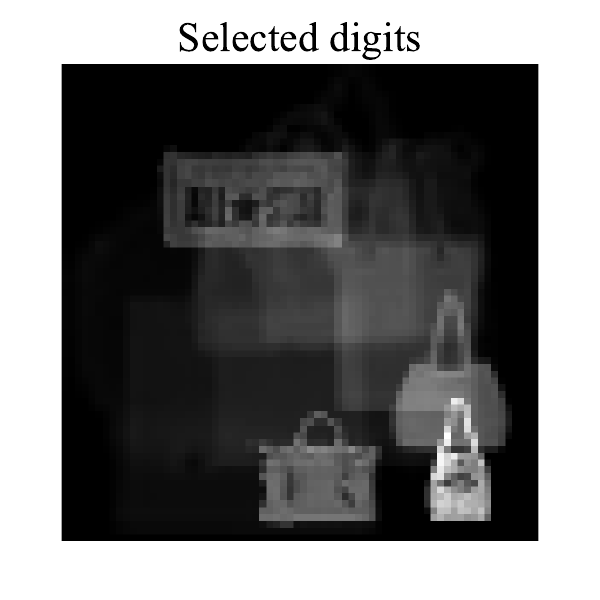}
\vspace{-0.2em}
\includegraphics[width=\linewidth, trim=10 10 10 10, clip]{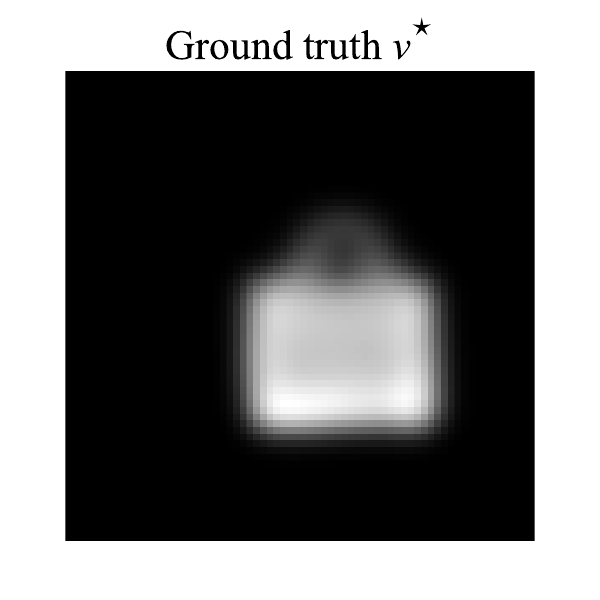}
\vspace{-0.2em}
\includegraphics[width=\linewidth, trim=10 10 10 10, clip]{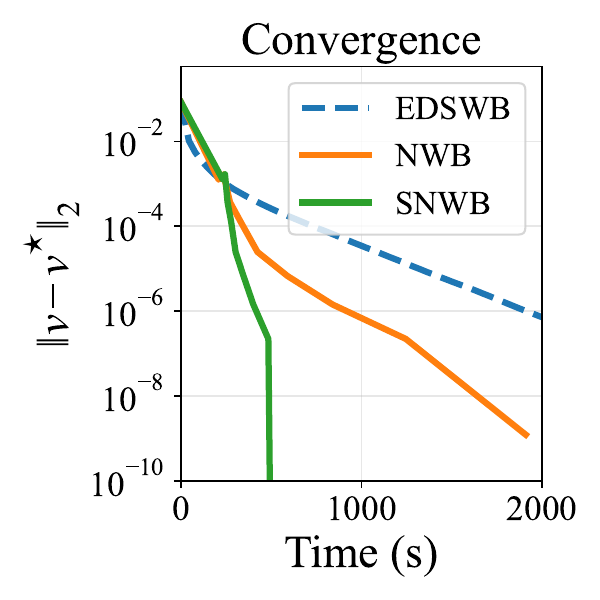}
\end{minipage}%

\caption{Qualitative and quantitative performance evaluation on MNIST and Fashion-MNIST benchmarks under $\tau = \eta / 2 = 7 \times 10^{-4}$. From top to bottom, the rows report the selected input samples, the corresponding ground-truth barycenter $v^\star$, and the convergence curves of different methods.}
\label{fig:fashionmnist_and_mnist_performance_evaluation_1_over_2}
\end{figure*}

\clearpage

\end{document}